\documentclass[11pt,reqno]{amsart}

\usepackage{mathrsfs}
\usepackage{amsmath}
\usepackage{amssymb}
\usepackage{amsfonts}
\usepackage{amsthm}
\usepackage{mathrsfs}
\usepackage{a4wide}
\usepackage{graphicx}
\usepackage{color}
\usepackage[sans]{dsfont}
\usepackage{linearA}
\usepackage{pgfkeys}
\usepackage{longtable}
\usepackage{pdflscape}
\usepackage{float}
\usepackage{cancel}
\usepackage{dsfont}
\usepackage{setspace}
\usepackage{array}
\usepackage{hyperref}
\hypersetup{colorlinks=true,allcolors=[rgb]{0,0,0.6}}
\usepackage[english]{babel}
\usepackage[applemac]{inputenc}
\usepackage[T1]{fontenc}
\usepackage{enumitem}
\usepackage[square,numbers,sort&compress]{natbib}
\usepackage[shadow,textwidth=22mm,textsize=tiny]{todonotes}
\usepackage[left=2.5cm,right=2.5cm,top=3cm,bottom=3cm]{geometry}
\usetikzlibrary {shapes}
\def\centerarc[#1](#2)(#3:#4:#5) [draw options] (center) (initial angle:final angle:radius) { \draw[#1] ($(#2)+({#5*cos(#3)},{#5*sin(#3)})$) arc (#3:#4:#5); }

\newtheorem{theorem}{Theorem}[section]
\newtheorem{lemma}[theorem]{Lemma}
\newtheorem{proposition}[theorem]{Proposition}
\newtheorem{corollary}[theorem]{Corollary}
\newtheorem{definition}[theorem]{Definition}

\newtheorem{remark}[theorem]{Remark}
\newtheorem{assumption}{Hypothesis}

\numberwithin{equation}{section}
\DeclareMathOperator{\slim}{s-lim}
\DeclareMathOperator*{\wlim}{w-lim}

\DeclareMathOperator{\Ran}{Ran}
\DeclareMathOperator{\Ker}{Ker}
\DeclareMathOperator{\Id}{Id}
\DeclareMathOperator{\Span}{Span}
\DeclareMathOperator{\im}{Im}
\DeclareMathOperator{\re}{Re}

	\newcommand{\Hi}{\mathcal{H}} 
	\newcommand{\C}{\mathbb{C}} 
	\newcommand{\R}{\mathbb{R}} 
	\newcommand{\N}{\mathbb{N}} 
	\newcommand{\norme}[1]{\left\Vert #1\right\Vert} 
	\newcommand{\scal}[2]{\left\langle{#1} \middle., {#2}\right\rangle} 
	\newcommand{\abso}[1]{\left|{#1}\right|}
	\newcommand{\Res}{\mathcal{R}}

\author[J. Faupin]{J{\'e}r{\'e}my Faupin}
\address[J. Faupin]{Institut Elie Cartan de Lorraine \\
Universit{\'e} de Lorraine,
57045 Metz Cedex 1, France}
\email{jeremy.faupin@univ-lorraine.fr}

\author[N. Frantz]{Nicolas Frantz}
\address[N. Frantz]{Institut Elie Cartan de Lorraine \\
Universit{\'e} de Lorraine,
57045 Metz Cedex 1, France}
\email{nicolas.frantz@univ-lorraine.fr}

\begin{document}
\bibliographystyle{abbrv} \title[Spectral decomposition of some non-self-adjoint operators]{Spectral decomposition of some non-self-adjoint operators}

\begin{abstract}
We consider non-self-adjoint operators in Hilbert spaces of the form $H=H_0+CWC$, where $H_0$ is self-adjoint, $W$ is bounded and $C$ is a metric operator, $C$ bounded and relatively compact with respect to $H_0$. We suppose that $C(H_0-z)^{-1}C$ is uniformly bounded in $z\in\mathbb{C}\setminus\mathbb{R}$. We define the spectral singularities of $H$ as the points of the essential spectrum $\lambda\in\sigma_{\mathrm{ess}}(H)$ such that $C(H\pm i\varepsilon)^{-1}CW$ does not have a limit as $\varepsilon\to0^+$. We prove that the spectral singularities of $H$ are in one-to-one correspondence with the eigenvalues, associated to resonant states, of an extension of $H$ to a larger Hilbert space. Next, we show that the asymptotically disappearing states for $H$, i.e. the set of vectors $\varphi$ such that $e^{\pm itH}\varphi\to0$ as $t\to\infty$, coincide with the generalized eigenstates of $H$ corresponding to eigenvalues $\lambda\in\mathbb{C}$, $\mp\mathrm{Im}(\lambda)>0$. Finally, we define the absolutely continuous spectral subspace of $H$ and show that it satisfies $\Hi_{\mathrm{ac}}(H)=\Hi_{\mathrm{p}}(H^*)^\perp$, where $\Hi_{\mathrm{p}}(H^*)$ stands for the point spectrum of $H^*$. We thus obtain a direct sum decomposition of the Hilbert spaces in terms of spectral subspaces of $H$. One of the main ingredients of our proofs is a spectral resolution formula for a bounded operator $r(H)$ regularizing the identity at spectral singularities. Our results apply to Schrödinger operators with complex potentials.
\end{abstract}

\maketitle

\tableofcontents

\section{Introduction}

We are interested in this paper in the evolution of a quantum system governed by the time-dependent Schrödinger equation
\begin{equation}
i\partial_t \varphi = H \varphi , \label{eq:Schrodinger}
\end{equation}
with a \emph{non-self-adjoint} operator $H$.

Non-self-adjoint `Hamiltonians' in Quantum Mechanics are considered in various contexts, see \cite{Ba15_01,Kr17_01} and references therein. We mention here two frameworks that are especially relevant for our study.

As effective or phenomenological operators, non-self-adjoint operators are used to describe non-conservative phenomena. A celebrated example is the \emph{optical model} in nuclear physics describing both the elastic and inelastic scattering of a neutron -- or a proton -- at a nucleus. It was introduced by Feshbach, Porter and Weisskopf \cite{FePoWe54_01} as an empirical model allowing, in particular, for the description of the formation of a compound nucleus \cite{Bo36_01}. In this model, the unnormalized state of the neutron at a positive time $t$ is given by the solution $\varphi_t=e^{-itH}\varphi$ to \eqref{eq:Schrodinger}, with $H = - \Delta + V(x)$ a dissipative Schrödinger operator in $L^2(\mathbb{R}^d)$, $\mathrm{Im}(V(x))\le0$. Part of the energy of the neutron may be transferred to the nucleus, possibly leading to the absorption, or capture, of the neutron by the nucleus. Mathematically, this is reflected in the dissipative nature of the equation since, in particular, given a normalized initial state $\varphi\in L^2(\mathbb{R}^d)$, we have that $\|e^{-itH}\varphi\|_{L^2}\le 1$ for all $t\ge0$. The probability of absorption 
\begin{equation*}
p_{\mathrm{abs}}(\varphi):=1-\lim_{t\to\infty}\big\|e^{-itH}\varphi\big\|_{L^2}^2
\end{equation*}
 does not vanish in general.
 
 The nuclear optical model leads to predictions that correspond to experimental scattering data to a high precision. Theoretical justifications of the model have been given in \cite{Fe58_01,Fe58_02,Fe62_01} (see also \cite{Ba15_01,Fe92_01,Ho71_01}). The idea of the justification consists in projecting out, in the Schrödinger equation associated to the total system nucleus -- neutron, the degrees of freedom corresponding to the nucleus. This can be performed using Schur's complement formula and leads to a Schrödinger equation for the neutron which is non-linear in energy. The effective, non-self-adjoint Hamiltonian for the neutron is then obtained by averaging over energy. 
 
\emph{$\mathcal{PT}$-symmetric} operators constitute another widely used class of non-self-adjoint operators in Quantum Mechanics. It was observed by Bender and Boettcher \cite{BeBo98_01} that a large class of `$\mathcal{PT}$-invariant Hamiltonians' have real spectra and can therefore be quantum mechanically relevant in many situations. For Schrödinger operators, $H=-\Delta+V(x)$ on $L^2(\mathbb{R}^d)$, with $V$ a complex potential, $\mathcal{PT}$-symmetry means that
\begin{equation*}
[ H , \mathcal{PT} ] = 0,
\end{equation*}
where $\mathcal{P}$ is the parity operator, $(\mathcal{P}\varphi)(x)=\varphi(-x)$, and $\mathcal{T}$ the time-reversal operator, $(\mathcal{T}\varphi)(x)=\overline{\varphi(x)}$. In recent years, $\mathcal{PT}$-symmetric operators have attracted lots of attention, from theoretical studies showing that $\mathcal{PT}$-invariant operators have real spectra under suitable conditions \cite{DoDuTa01_01,Mo02_01,Sh02_01}, to experimental studies revealing $\mathcal{PT}$-symmetry-like structures, in particular in optics \cite{Lo09_01,Re12_01,Ru10_01}. See \cite{El18_01} for more references and more recent developments, and \cite{WeBe20_01} for the study of one-dimensional $\mathcal{PT}$-symmetric Schrödinger operators having continuous spectra.

In this paper, we consider an abstract class of non-self-adjoint operators in a complex Hilbert space $\Hi$, of the form $H=H_0+V$. We suppose that $H_0$ is a self-adjoint operator with purely absolutely continuous spectrum and that $V$ is a relatively compact perturbation of $H_0$. In particular, the essential spectra of $H$ and $H_0$ coincide. We suppose furthermore that $V$ admits a factorization as $V=CWC$, with $W$ bounded and $C$ a strictly positive operator such that
\begin{equation*}
\sup_{z\in\mathbb{C}\setminus\mathbb{R}}\big\|C(H_0-z)^{-1}C\big\|_{\mathcal{L}(\Hi)}<\infty.
\end{equation*}
Such factorizations go back to the seminal work of Kato \cite{Ka65_01}, see also \cite{KoKu66_01}.

We are interested in a \emph{spectral decomposition} of the non-self-adjoint operator $H$, in relation with the long-time behavior of the solutions to \eqref{eq:Schrodinger}. Note that $H$ being a bounded perturbation of $H_0$, the operator $-iH$ generates a strongly continuous group $\{e^{-itH}\}_{t\in\mathbb{R}}$ and hence, for any $\varphi\in\Hi$, \eqref{eq:Schrodinger} admits a global solution $\mathbb{R}\ni t\mapsto e^{-itH}\varphi\in\Hi$.

Roughly speaking, our main contributions can be summarized as follows. First, defining the spaces of \emph{asymptotically disappearing states} as
\begin{equation*}
\Hi_{\mathrm{ads}}^\pm(H):=\Big\{\varphi\in\Hi, \, \lim_{t\to\pm\infty}\big\|e^{-itH}\varphi\big\|_\Hi=0\Big\},
\end{equation*}
 we will show that $\Hi_{\mathrm{ads}}^\pm(H)$ coincide with the vector space spanned by all eigenvectors, or generalized eigenvectors, corresponding to eigenvalues $\lambda$ of $\Hi$ such that $\mp\mathrm{Im}(\lambda)\ge0$. Next, defining the \emph{absolutely continuous spectral subspace} $\Hi_{\mathrm{ac}}(H)$ of $H$ as the closure of
\begin{equation*}
\Big\lbrace \varphi\in\Hi,~\exists\, c_\varphi>0,\forall \psi\in\Hi,~\int_\R\abso{\scal{e^{-itH}\varphi}{\psi}_\Hi}^2\mathrm{d}t\leq c_\varphi\norme{\psi}^2_\Hi \Big\rbrace,
\end{equation*}
we will prove that $\Hi_{\mathrm{ac}}(H)=\Hi_{\mathrm{p}}(H^*)^\perp$, where $H^*$ stands for the adjoint of $H$ and $\Hi_{\mathrm{p}}(H^*)$ is the point spectral subspace of $H^*$, i.e. the vector space spanned by all eigenstates or generalized eigenstates of $H^*$. These characterizations of $\Hi_{\mathrm{ads}}^\pm(H)$ and $\Hi_{\mathrm{ac}}(H)$ in turn imply a $J$-orthogonal decomposition of the Hilbert space (for some conjugation operator $J$) given by
\begin{align*}
		 	\mathcal{H} = \mathcal{H}_{\mathrm{ac}}(H) \oplus \Hi_{\mathrm{ads}}^+(H)\oplus\Hi_{\mathrm{ads}}^-(H)\oplus\Hi_{\mathrm{b}}(H),
\end{align*}
		where $\Hi_{\mathrm{b}}(H)$ is the space of `\emph{bound states}', i.e. the closure of the vector space spanned by all generalized eigenvectors of $H$ corresponding to real eigenvalues.
 
To prove these results, we will require that $H$ only have finitely many eigenvalues (counting algebraic multiplicities) and finitely many \emph{spectral singularities}. As in previous works concerning dissipative operators in Hilbert spaces \cite{Fa21_01,FaFr18_01,FaNi18_01}, the notion of spectral singularities plays a central role in this paper. In our context, we will define a spectral singularity as a point $\lambda$ of the essential spectrum of $H$ such that one of the two limits
\begin{equation*}
\lim_{\varepsilon\to0^+}C(H-\lambda\pm i\varepsilon)^{-1}CW
\end{equation*}
does not exist in the norm topology of $\mathcal{L}(\Hi)$. We will prove that $\lambda$ is a spectral singularity of $H$ if and only if it is an eigenvalue of an extension $H'$ of $H$ to a larger Hilbert space $\Hi'_C$,  corresponding to an eigenstate (a `\emph{resonant}' state) belonging to a suitable subspace of $\Hi'_C$ (the space $\Hi'_C$ will be defined as the anti-dual of $\mathrm{Ran}(C)$, equipped with a suitable norm).

Our results concerning the spaces of asymptotically disappearing states $\Hi_{\mathrm{ads}}^\pm(H)$ generalize previous results for dissipative operators recently obtained in \cite{FaFr18_01}. Our proof is more direct, in particular it does not rely on the scattering theory for the pair $(H,H_0)$, which was a crucial element of the proof in \cite{FaFr18_01}. 
Our results showing that $\Hi_{\mathrm{ac}}(H)=\Hi_{\mathrm{p}}(H^*)^\perp$ seem to be new. It is worth mentioning that, contrary to previous results on absolutely continuous spectral subspaces for non-self-adjoint operators (see \cite{Da80_01} for dissipative operators and \cite{KiNa09_01,Na76_01} and references therein for a more general context), we do not use the theory of dilations of contractive semigroups.

The abstract theory developed in this paper applies to Schrödinger operators $H=-\Delta+V(x)$ on $L^2(\mathbb{R}^d)$, under suitable decay assumptions on the complex potential $V$. In this case, spectral singularities correspond to real resonances. Some of our results on spectral singularities may thus be seen as abstract versions of corresponding well-known properties in the theory of resonances for Schrödinger operators \cite{DyZw19_01}. On the other hand, the characterizations of the subspaces $\Hi_{\mathrm{ads}}^\pm(H)$ and $\Hi_{\mathrm{ac}}(H)$ that we establish seem to be new even in the context of Schrödinger operators.

Before stating our results in more precise terms in Section \ref{sec:results}, we begin with describing in details the abstract setting studied in this paper in Section \ref{sec:setting}.

\medskip

\emph{Notation.} In what follows, given two Hilbert spaces $\Hi_1$, $\Hi_2$, the notation $\mathcal{L}(\Hi_1,\Hi_2)$ stands for the set of bounded operators from $\Hi_1$ to $\Hi_2$. If $\Hi_1=\Hi_2$, we set $\mathcal{L}(\Hi_1):=\mathcal{L}(\Hi_1,\Hi_1)$.

If $E$ is a Banach space and $E'$ its anti-dual, we denote by
\begin{equation*}
\langle u , \Phi \rangle_{E;E'} := \Phi(u),\quad \Phi\in E', \quad u \in E,
\end{equation*}
the usual duality bracket.

The domain of an operator $A$ on a Hilbert space $\Hi$ is denoted by $\mathcal{D}(A)$. The spectrum and resolvent set of $A$ are denoted by $\sigma(A)$ and $\rho(A)$, respectively. For $z\in\rho(A)$, we let 
\begin{equation*}
\mathcal{R}_A(z):=(A-z)^{-1}
\end{equation*}
be the resolvent of $A$. In the case where $A$ is the unperturbed operator $H_0$, we will also use the shorthand $\mathcal{R}_0:=\mathcal{R}_{H_0}$.

We let $\C^\pm:=\{z\in\C, \, \pm\mathrm{Im}(z)>0\}$ and $\bar\C^\pm:=\{z\in\C, \, \pm\mathrm{Im}(z)\ge0\}$. The complex open disc centered at $\lambda$ and of radius $r$ is denoted by
\begin{equation*}
\mathring D(\lambda,r):=\{z\in\C,\,|z-\lambda|<r\}.
\end{equation*}

\section{Abstract setting}\label{sec:setting}

\subsection{The model}	

	Let $(\Hi,\scal{.}{.}_\Hi)$ be a complex separable Hilbert space. On $\Hi$, we consider the operator 
		\begin{equation}\label{eq:defH}
			H:=H_0+V,
		\end{equation}
	where $H_0$ is self-adjoint and semi-bounded from below and $V\in\mathcal{L}(\Hi)$ is a bounded operator. In particular, $H$ is a closed operator with domain $\mathcal{D}(H)=\mathcal{D}(H_0)$ and its adjoint is given by 
	\begin{equation*}
		H^*=H_0+V^*, \quad \mathcal{D}(H^*)=\mathcal{D}(H_0).
	\end{equation*}
	Without loss of generality, we suppose that $H_0\geq 0$. 
	Since $H$ is a perturbation of the self-adjoint operator $H_0$ by the bounded operator $V$, $-iH$ is the generator of a strongly continuous one-parameter group $\left\lbrace e^{-itH} \right\rbrace_{t\in\R}$ satisfying 
	\begin{equation*}
		\norme{e^{-itH}}_{\mathcal{L}(\Hi)}\leq e^{\abso{t}\norme{V}_{\mathcal{L}(\Hi)}},\quad t\in\R,
	\end{equation*}
	(see e.g \cite{Da07_01} or \cite{EnNa20_01}).
	
We assume that there exists a \emph{metric operator} $C\in\mathcal{L}(H)$ such that $C$ is relatively compact with respect to $H_0$ and $V$ is of the form
		\begin{equation}\label{eq:CWC}
	 		V=CWC,
		\end{equation}
	with $W\in\mathcal{L}(\Hi)$. We recall that a metric operator is a strictly positive operator, i.e. $C\ge0$ and $\mathrm{Ker}(C)=\{0\}$, see e.g. \cite{AnTr12_12}.

\subsection{Spectral subspaces, spectral projections}\label{subsec:subspaces}

	Recall that $\sigma(H)$ stands for the spectrum of $H$ and $\rho(H)=\C\setminus\sigma(H)$ its resolvent set. As usual, the point spectrum of $H$ is defined as the set of all eigenvalues of $H$,
	\begin{equation*}
		\sigma_{\mathrm{p}}(H):=\big\{\lambda \in \mathbb{C} , \, \mathrm{Ker}(H-\lambda)\neq\{0\}\big\}.
	\end{equation*}
	The discrete spectrum of $H$, $\sigma_\mathrm{disc}(H)$, is the set of all isolated eigenvalues $\lambda$ with finite algebraic multiplicities $\mathrm{m}_\lambda(H)$, where
	\begin{equation*}
		\mathrm{m}_\lambda(H):=\mathrm{dim}\Big(\bigcup_{k=1}^\infty\mathrm{Ker}\big( (H-\lambda)^k \big) \Big ). 
	\end{equation*}
	 Under our assumptions, since $V$ is a relatively compact perturbation of $H_0$, the essential spectrum $\sigma_{\mathrm{ess}}(H):=\sigma(H)\backslash\sigma_\mathrm{disc}(H)$ coincides with the essential spectrum of $H_0$ and the discrete spectrum $\sigma_{\mathrm{disc}}(H)$ is at most countable and can only accumulate at points of $\sigma_{\mathrm{ess}}(H)$. See Figure \ref{fig1}. We define in addition the set of all eigenvalues embedded in the essential spectrum of $H$:
	\begin{equation*}
		\sigma_{\mathrm{emb}}(H):=\sigma_{\mathrm{p}}(H)\cap\sigma_{\mathrm{ess}}(H).
	\end{equation*}
	
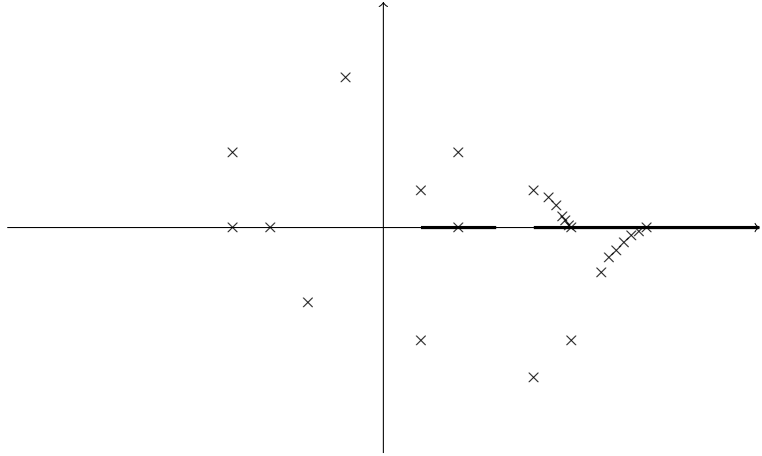
\begin{figure}
	\begin{center}
		\begin{tikzpicture}
			
			\draw[->] (-5,0) --(5,0);
			\draw[->] (0,-3) --(0,3);
			\draw[-, very thick] (0.5,0)--(1.5,0);
			\draw[-, very thick] (2,0)--(5,0);
						
			\draw (1,1) node {\tiny{$\times$}};
			
			\draw (0.5,0.5) node {\tiny{$\times$}};
			
			\draw (-2,1) node {\tiny{$\times$}};
			
			\draw (0.5,-1.5) node {\tiny{$\times$}};
			
			\draw (2,-2) node {\tiny{$\times$}};
			
			\draw (-1,-1) node {\tiny{$\times$}};
			
			\draw (-0.5,2) node {\tiny{$\times$}};
			
			\draw (2.5,-1.5) node {\tiny{$\times$}};

			\draw (1,0) node {\tiny{$\times$}};
			
			\draw(2,0.5) node {\tiny{$\times$}};
			
			\draw(2.2,0.4) node {\tiny{$\times$}};
			
			\draw(2.3,0.3) node {\tiny{$\times$}};
			
			\draw(2.38,0.15) node {\tiny{$\times$}};
			
			\draw(2.42,0.1) node {\tiny{$\times$}};
			
			\draw(2.47,0.03) node {\tiny{$\times$}};
			
			\draw (2.5,0) node {\tiny{$\times$}};
			
			\draw(-1.5,0)  node {\tiny{$\times$}};
			
			\draw(-2,0) node {\tiny{$\times$}};
			
			\draw(3.5,0) node {\tiny{$\times$}};
			
			\draw(3.4,-0.05)node {\tiny{$\times$}};
			
			\draw(3.3,-0.11)node {\tiny{$\times$}};
			
			\draw(3.2,-0.20) node {\tiny{$\times$}};
			
			\draw(3.1,-0.3) node {\tiny{$\times$}};
			
			\draw(3,-0.4) node {\tiny{$\times$}};
			
			\draw(2.9,-0.6) node {\tiny{$\times$}};

		\end{tikzpicture}
	\end{center}
	\caption{ \footnotesize  \textbf{Spectrum of $H$.} The essential spectrum of $H$, represented by thick lines, coincides with that of $H_0$ and is contained in $[ 0 , \infty )$. Eigenvalues of $H$ are represented by crosses. The discrete spectrum of $H$ consists of isolated eigenvalues of finite algebraic multiplicities which may accumulate at any point of the essential spectrum. The point spectrum of $H$ may also contain eigenvalues embedded in the essential spectrum. }\label{fig1}
\end{figure}

	\subsubsection{Eigenspaces corresponding to isolated eigenvalues}\label{subsubsec:isolated}

		For $\lambda\in\sigma_\mathrm{disc}(H)$, we denote by
		\begin{equation}\label{eq:Projection_de_Riesz_def_pour_H}
			\Pi_\lambda(H):=\frac{1}{2\pi  i}\int_\gamma \left(z\Id-H\right)^{-1}\mathrm{d}z
		\end{equation}
		the usual Riesz projection, where $\gamma$ is a circle oriented counterclockwise and centered at $\lambda$, of sufficiently small radius (so that $\lambda$ is the only point of the spectrum of $H$ contained in the interior of $\gamma$). The algebraic multiplicity of $\lambda$ satisfies $\mathrm{m}_\lambda(H)=\dim\Ran(\Pi_\lambda(H))$. Since the restriction of $H$ to $\Ran(\Pi_\lambda(H))$ may have a nontrivial Jordan form, $\Ran(\Pi_\lambda(H))$ is in general spanned by generalized eigenvectors of $H$ associated to $\lambda$, i.e., by vectors $u\in\mathcal{D}(H^k)$ such that $(H-\lambda)^{k}u=0$ for some $1\leq k\leq \mathrm{m}_\lambda(H)$. We set 
		\begin{equation*}
			\Hi_\mathrm{disc}(H):=\Span\left\lbrace u\in\Ran(\Pi_\lambda(H)),~\lambda\in\sigma_\mathrm{disc}(H)\right\rbrace^{\mathrm{cl}},
		\end{equation*}
		where $A^{\mathrm{cl}}$ stands for the closure of a subset $A\subset\Hi$. We will sometimes assume that the discrete spectrum of $H$ is finite. The spectral projection $\Pi_{\mathrm{disc}}(H)$ onto $\Hi_\mathrm{disc}(H)$ is then defined by
		\begin{equation}\label{eq:def_Pi_disc}
			\Pi_{\mathrm{disc}}(H):=\sum_{\lambda\in\sigma_{\mathrm{disc}}(H)}\Pi_\lambda(H).
		\end{equation}

		We will also consider the following three subspaces of $\Hi_\mathrm{disc}(H)$:
		\begin{align}
			\Hi_\mathrm{disc}^+(H)&:=\Span\left\lbrace u\in\Ran\left(\Pi_\lambda(H)\right),\lambda\in\sigma_\mathrm{disc}(H),\im(\lambda)<0\right\rbrace^{\mathrm{cl}},\label{eq:def_etat_discret_+}\\
			\Hi_\mathrm{disc}^0(H)&:=\Span\left\lbrace u\in\Ran\left(\Pi_\lambda(H)\right),\lambda\in\sigma_\mathrm{disc}(H),\im(\lambda)=0\right\rbrace^{\mathrm{cl}},\label{eq:def_etat_discret_0}\\
			\Hi_\mathrm{disc}^-(H)&:=\Span\left\lbrace u\in\Ran\left(\Pi_\lambda(H)\right),\lambda\in\sigma_\mathrm{disc}(H),\im(\lambda)>0\right\rbrace^{\mathrm{cl}}.\label{eq:def_etat_discret_-}
		\end{align}
		Observe that $\Hi_\mathrm{disc}^+(H)$ is the closure of the vector space spanned by all generalized eigenvectors corresponding to eigenvalues with \emph{negative} imaginary part, and likewise for $\Hi_\mathrm{disc}^-(H)$. The reason for these conventions will be understood later (see Theorem \ref{thm:caracterisation_etat_disparaissent_à_infini}).
		Clearly, if $\sigma_{\mathrm{disc}}(H)$ is finite, we have the following direct sum decomposition:
		\begin{equation*}
			\Hi_\mathrm{disc}(H)=\Hi_\mathrm{disc}^-(H)\oplus\Hi_\mathrm{disc}^0(H)\oplus\Hi_\mathrm{disc}^+(H).
		\end{equation*}
		The corresponding spectral projections $\Pi_{\mathrm{disc}}^\sharp(H)$, where $\sharp$ stands for $+$, $-$ or $0$, are defined as in \eqref{eq:def_Pi_disc}.
	
		\subsubsection{Eigenspaces corresponding to embedded eigenvalues}\label{subsubsec:embedded}
		Suppose now that $\lambda$ is an eigenvalue of $H$ embedded in its essential spectrum. The Riesz projection corresponding to $\lambda$ is then ill-defined, but, under some further conditions, one can define the spectral projection $\Pi_\lambda(H)$ as follows.
		
		In the following (see Hypothesis \ref{hyp:Conjugate_operator} below), we will suppose the existence of a \emph{conjugation operator} $J\in\mathcal{L}(\Hi)$ satisfying
			\begin{equation}\label{eq:existence_J}
			J^2=\Id, \quad J\mathcal{D}(H_0)\subset\mathcal{D}(H_0) \quad\text{and}\quad	\forall u\in\mathcal{D}(H_0), \quad JHu=H^* Ju.
			\end{equation}
		In particular, $J$ establishes a one-to-one correspondence between $\Ker((H-\lambda)^{k})$ and $\Ker((H^*-\bar{\lambda})^{k})$ for all $k\in\mathbb{N}$ and hence
		\begin{equation*}
		\mathrm{m}_\lambda(H)=\mathrm{m}_{\bar{\lambda}}(H^*).
		\end{equation*} 
		To shorten notation, let $\mathrm{m}_\lambda=\mathrm{m}_\lambda(H)=\mathrm{m}_{\bar\lambda}(H^*)$. In order to study the absolutely continuous spectral subspace of $H$ (see \eqref{eq:def_M(H)}--\eqref{eq:def_Hac}), we will suppose that for each embedded eigenvalue $\lambda\in\sigma_{\mathrm{ess}}(H)$, $\mathrm{m}_\lambda$ is finite and the symmetric bilinear form
		\begin{equation}\label{eq:invertibility_mat}
			\Ker((H-\lambda)^{\mathrm{m}_\lambda}) \ni (u,v)\mapsto \langle Ju,v\rangle_\Hi \quad\text{is non-degenerate}.
		\end{equation}
		This implies that there exists a basis $(\varphi_k)_{1\leq k\leq \mathrm{m}_\lambda}$ of $\Ker((H-\lambda)^{\mathrm{m}_\lambda})$ such that
		\begin{equation*}
		\langle J \varphi_i , \varphi_j \rangle_{\Hi} = \delta_{ij} , \quad 1\le i,j\le m_\lambda.
		\end{equation*}
		We can then define the spectral projection $\Pi_\lambda(H)$ onto the generalized eigenspace corresponding to $\lambda$ as
		\begin{equation}\label{eq:proj_emb}
			 \Pi_\lambda(H)u = \sum_{k=1}^{\mathrm{m}_\lambda}\scal{J\varphi_k}{u}_{\Hi}\varphi_k, \quad u \in \Hi.
		\end{equation}
		It is not difficult to observe that $\Pi_\lambda(H)$ is a projection commuting with $H$, such that $\Pi_\lambda(H)\in\mathcal{L}(\Hi)$ and $\Pi_\lambda(H)^*=\Pi_\lambda(H^*)$ (see Proposition \ref{prop:spectral_proj} below for a more precise statement). Note however that if \eqref{eq:invertibility_mat} does not hold, it is not clear how to define such a projection. In fact, in the simplest case where $\mathrm{m}_\lambda=1$, one easily verifies that the condition $\langle J\varphi,\varphi\rangle\neq0$ for any $\varphi\in\mathrm{Ker}(H-\lambda)\setminus\{0\}$ is necessary to have the existence of a projection onto $\mathrm{Ker}(H-\lambda)$ commuting with $H$.

		The closure of the vector space spanned by all generalized eigenstates corresponding to embedded eigenvalues of $H$ will be denoted by
		\begin{equation*}
			\Hi_{\mathrm{emb}}(H):=\Span\left\lbrace u\in\Ran(\Pi_\lambda(H)),~\lambda\in\sigma_\mathrm{emb}(H)\right\rbrace^{\mathrm{cl}}.
		\end{equation*}
		If $\sigma_{\mathrm{emb}}(H)$ is composed of finitely many eigenvalues with finite algebraic multiplicities and such that \eqref{eq:invertibility_mat} holds, we will also use the notation
		\begin{equation*}
			\Pi_{\mathrm{emb}}(H):=\sum_{\lambda\in\sigma_{\mathrm{emb}}(H)}\Pi_\lambda(H).
		\end{equation*}

		\subsubsection{Point spectral subspace}
		With the definitions of the spectral projections $\Pi_\lambda(H)$ given in Sections \ref{subsubsec:isolated} and \ref{subsubsec:embedded}, we have the following proposition, which covers both cases of isolated and embedded eigenvalues.
		\begin{proposition}\label{prop:spectral_proj}
			Suppose that there exists a conjugation operator $J$ such that \eqref{eq:existence_J} holds. Let $\lambda$ be an eigenvalue of $H$ with finite algebraic multiplicity $m_\lambda$. If $\lambda\in\sigma_{\mathrm{ess}}(H)$, suppose in addition that \eqref{eq:invertibility_mat} holds.   Then $\Pi_\lambda(H)\in\mathcal{L}(\Hi)$, $\Pi_\lambda(H)$ is a projection which preserves $\mathcal{D}(H)$ and commutes with $H$. Its range and adjoint are given respectively by
				\begin{equation*}
					\mathrm{Ran}(\Pi_\lambda(H)) = \Ker((H-\lambda)^{\mathrm{m}_\lambda}), \quad  \Pi_{\lambda}(H)^*=\Pi_{\overline{\lambda}}(H^*).
				\end{equation*}
We have
\begin{equation*}
	\left(H\Pi_\lambda(H)\right)^*=H^*\Pi_{\overline{\lambda}}(H^*),
\end{equation*}
and if $\lambda,\lambda'$ are two distinct eigenvalues of $H$, then
				\begin{equation*}
					\Pi_\lambda(H)\Pi_{\lambda'}(H)=0.
				\end{equation*}
		\end{proposition} 
		In the case of isolated eigenvalues, Proposition \ref{prop:spectral_proj} follows from the definition \eqref{eq:Projection_de_Riesz_def_pour_H} of the Riesz projection $\Pi_\lambda(H)$ (see e.g. \cite[Theorem XII.5]{ReSi80_01}). In the general case, it suffices to observe that the restriction of $H$ to $\Hi_{\mathrm{p}}(H)$ has a discrete spectrum (since $\Hi_{\mathrm{p}}(H)$ is finite dimensional) and that the Riesz projections associated to its eigenvalues are given by the restrictions of \eqref{eq:Projection_de_Riesz_def_pour_H} or \eqref{eq:proj_emb} to $\Hi_{\mathrm{p}}(H)$. 
		
		We will assume below that $H$ only has a finite number of eigenvalues with finite algebraic multiplicities.   Under this simplifying assumption, we set
		\begin{equation*}
			\Pi_\mathrm{p}(H):=\sum_{\lambda\in\sigma_{\mathrm{p}}(H)} \Pi_\lambda(H),
		\end{equation*}
		and observe that
		\begin{equation*}
			\Pi_\mathrm{p}(H)^*=\sum_{\lambda\in\sigma_{\mathrm{p}}(H)} \Pi_{\overline{\lambda}}(H^*)=:\Pi_\mathrm{p}(H^*). 
		\end{equation*}
		The point spectral subspace of $H$ is then defined by
		\begin{equation*}
			\Hi_{\mathrm{p}}(H):=\mathrm{Ran}(\Pi_{\mathrm{p}}(H)),
		\end{equation*}
		and likewise for $H^*$. It satisfies
		\begin{equation*}
			\Hi_{\mathrm{p}}(H)=\Hi_{\mathrm{disc}}(H)\oplus\Hi_{\mathrm{emb}}(H).
		\end{equation*}
		We also observe that under our assumptions, the point spectral subspace corresponding to eigenvalues with positive/negative imaginary parts identifies to the corresponding discrete spectral subspace:
				\begin{align}
			\Hi_\mathrm{p}^+(H)&:=\Span\left\lbrace u\in\Ran\left(\Pi_\lambda(H)\right),\lambda\in\sigma_\mathrm{p}(H),\im(\lambda)<0\right\rbrace^{\mathrm{cl}}=\Hi_\mathrm{disc}^+(H),\\
			\Hi_\mathrm{p}^-(H)&:=\Span\left\lbrace u\in\Ran\left(\Pi_\lambda(H)\right),\lambda\in\sigma_\mathrm{p}(H),\im(\lambda)>0\right\rbrace^{\mathrm{cl}}=\Hi_\mathrm{disc}^-(H).
		\end{align}
		
  		\subsubsection{Subspaces of asymptotically disappearing states}
  	
  			We define the subspaces of \emph{asymptotically disappearing states} as
  				\begin{equation*}
  					\Hi_\mathrm{ads}^\pm(H):=\left\lbrace u\in\Hi,~\lim_{t\rightarrow \pm\infty}\norme{e^{-itH}u}_\Hi=0\right\rbrace^{\mathrm{cl}}.
  				\end{equation*}
  			Note that $\Hi^+_\mathrm{ads}(H)$ and $\Hi^-_\mathrm{ads}(H)$ are closed. Using that for a generalized eigenvector $\varphi \in \Hi^\pm_\mathrm{p}(H)$, the norm $\|e^{-itH}\varphi\|_\Hi$ decays exponentially as $t\to\pm\infty$, it is not difficult to verify that
			\begin{equation*}
					\Hi_\mathrm{p}^\pm(H) \subset \Hi_\mathrm{ads}^\pm(H),
			\end{equation*}
			(see Proposition \ref{prop:easy_incl}). In Theorem \ref{thm:caracterisation_etat_disparaissent_à_infini} below, we will give conditions under which this inclusion becomes an equality.
  		
   		\subsubsection{Absolutely continuous spectral subspace}\label{subsub:defi_absolutely_continuous_spectral_subspace}
	
			Let 
			\begin{equation}\label{eq:def_M(H)}
				\mathcal{M}(H):=\left\lbrace u\in\Hi,~\exists c_u>0,\forall v\in\Hi,~\int_\R\abso{\scal{e^{-itH}u}{v}_\Hi}^2\mathrm{d}t\leq c_u\norme{v}^2_\Hi \right\rbrace.
			\end{equation}
			We define the \emph{absolutely continuous spectral subspace} of $H$, $\Hi_\mathrm{ac}(H)$, as the closure of $\mathcal{M}(H)$ in $\Hi$,
			\begin{equation}\label{eq:def_Hac}
				\Hi_\mathrm{ac}(H) := \mathcal{M}(H)^{\mathrm{cl}}.
			\end{equation}
			Note that if $H$ is self-adjoint, then $\mathcal{M}(H)$ is closed and coincides with the usual absolutely continuous spectral subspace of $H$.
			
			Such a definition of an absolutely continuous spectral subspace for non-self-adjoint operators goes back to \cite{Da78_01}, where dissipative operators, $\mathrm{Im}(V)\le0$, are considered, and the integral in \eqref{eq:def_M(H)} is taken over $[0,\infty)$ instead of $\mathbb{R}$. For self-adjoint operators, our definition and that of \cite{Da78_01} coincide; this is however not the case for non-self-adjoint operators. See Section \ref{subsec:ideas} for a discussion comparing our results and those of \cite{Da78_01,Da80_01} for the absolutely continuous spectral subspace of dissipative operators. Note also that other definitions of an absolutely continuous spectral subspace for non-dissipative perturbations of a self-adjoint operator have been considered in the literature, using the theory of dilations. See \cite{KiNa09_01,Na76_01} and references therein.
			
			Our definition of $\Hi_\mathrm{ac}(H)$ is motivated by the following fact: it is not difficult to see that if $v$ is a generalized eigenvector of $H^*$, then $t\mapsto \scal{e^{-itH}u}{v}_\Hi$ cannot belong to $L^2(\R)$ unless $\scal{u}{v}=0$ (see Section \ref{subsec:abs_cont}). In other words
			\begin{equation*}
				\Hi_{\mathrm{ac}}(H)\subset\Hi_{\mathrm{p}}(H^*)^\perp.
			\end{equation*}
			Theorem \ref{thm:caracterisation_espace_absolument_continu} will show that, under suitable assumptions,
			$
				\Hi_\mathrm{ac}(H) = \Hi_{\mathrm{p}}(H^*)^\perp.
			$
			This generalizes the equality 			$
				\Hi_\mathrm{ac}(H) = \Hi_{\mathrm{p}}(H)^\perp
			$
			which holds for self-adjoint operators without singular continuous spectrum.

	\subsection{Extension of the Hilbert space}\label{subsec:extension}

		In this section, we construct an extension of the Hilbert space $\Hi$ containing the `outgoing and incoming resonant states' that will be introduced below (see Definition \ref{def:resonant}). To this end we consider a Gelfand triple defined in terms of the metric operator $C$ appearing in the definition \eqref{eq:defH}--\eqref{eq:CWC} of $H$.
		
		\subsubsection{Gelfand triple}
		Recall that $C$ is supposed to be relatively compact with respect to $H_0$. Let 
		\begin{equation*}
			\Hi_C:=\Ran(C).
		\end{equation*}
		Since $C$ is self-adjoint and injective, its inverse $C^{-1}$ is a self-adjoint unbounded operator with dense domaine $\mathcal{D}(C^{-1}):=\Hi_C$. We equip $\Hi_C$ with the scalar product 
		\begin{equation*}
			\scal{u}{v}_{\Hi_C}:=\scal{C^{-1}u}{C^{-1}v}_\Hi,\quad  u,v\in\Hi_C.
		\end{equation*}
		It is not difficult to verify that the identity operator from $\Hi_C$ to $\Hi$ is a continuous embedding, $\Hi_C\hookrightarrow\Hi$. 
		
		Let $\Hi_C'$ be the anti-dual of $\Hi_C$ (the set of anti-linear continuous maps from $(\Hi_C,\|\cdot\|_\Hi)$ to $\C$). Setting
		\begin{equation*}
			\scal{u}{v}_{C}:=\scal{Cu}{Cv}_\Hi, \quad  u,v\in\Hi,
		\end{equation*}
		one verifies that $\Hi_C'$ identifies with the completion of $\Hi$ for the norm $\norme{.}_C$ associated to $\langle\cdot,\cdot\rangle_C$. Thus we obtain the Gelfand triple
		\begin{equation*}
			\Hi_C\hookrightarrow\Hi\hookrightarrow\Hi_C'.
		\end{equation*}
	
		Now, given $A\in\mathcal{L}(\Hi_C)$ a bounded operator in $\Hi_C$, the anti-dual of $A$, denoted by $A'\in\mathcal{L}(\Hi'_C)$, is defined by 
		\begin{equation}\label{eq:definition_dual_operateur}
				 \langle u , A'\Psi \rangle_{\Hi_C,\Hi'_C}:=\langle A^* u,\Psi\rangle_{\Hi_C,\Hi'_C}, \quad \Psi\in\Hi_C', \quad u\in\Hi_C.
		\end{equation}
		Since the restriction of $C$ to $\Hi_C$ belongs to $\mathcal{L}(\Hi_C)$, we can consider its anti-dual defined by \eqref{eq:definition_dual_operateur}; we still use the symbol $C'$ to denote the anti-dual of $C|_{\Hi_C}$. It is not difficult to show that, for all $\Psi\in\Hi_C'$, $C'\Psi$ extends to an anti-linear continuous form on $\Hi$ which identifies to an element of $\Hi$ \emph{via} the (anti-linear version of) the Riesz representation theorem. The map $C':\Hi_C'\to\Hi$ then extends to a bounded operator $C'\in\mathcal{L}(\Hi_C',\Hi)$ satisfying
		\begin{equation*}
			\norme{C'}_{\mathcal{L}(\Hi_C',\Hi)}\leq 1. 
		\end{equation*}
		Moreover, $C'$ is an extension of $C$ and for all $\Psi\in\Hi_C'$, there exists a sequence $(v_n)_{n\in\N}\subset \Hi$ such that
			\begin{equation*}
				\big \|C'\Psi - Cv_n \big \|_\Hi \to 0 , \quad n \to \infty.
			\end{equation*}
		Note that the anti-dual of $C^{-1}\in\mathcal{L}(\Hi_C,\Hi)$, denoted by $(C^{-1})'\in\mathcal{L}(\Hi,\Hi_C')$, satisfies $C'^{-1}=(C^{-1})'$. Note also that $\Hi'_C$ is equipped with the scalar product
		\begin{equation*}
			  \scal{\Psi}{\Phi}_{\Hi_{C'}}:=\scal{C'\Psi}{C'\Phi}_\Hi , \quad \Psi,\Phi\in\Hi_C', 
		\end{equation*}
		which is an extension of $\scal{\cdot}{\cdot}_C$ to $\Hi_C'$. 
		
		\subsubsection{Extension of $H$}
		Our next concern is to define the anti-dual of the operator 
		\begin{equation*}
		H=H_0+V=H_0+CWC.
		\end{equation*}
		First, we observe that since $V=CW C$ belongs to $\mathcal{L}(\Hi'_C,\Hi_C)$, its anti-dual $V'\in\mathcal{L}(\Hi_C',\Hi_C)$ is well-defined and given by
		\begin{equation*}
			V'=CW C'.
		\end{equation*}
		Now, in order for the anti-dual of the unbounded operator $H_0$ to be well-defined, we will assume (see Hypothesis \ref{hyp:domaine_H_0_C} below) that 
		\begin{equation*}
				\mathcal{D}({H_0}|_{\Hi_C}):=\left\lbrace u\in\mathcal{D}(H_0)\cap\Hi_C, H_0u\in\Hi_C\right\rbrace
			\end{equation*}
		is dense in $\Hi_C$ for the topology of $\Hi_C$.  The anti-dual $H_0'$ of $H_0$ is then defined by
		\begin{equation*}
			\mathcal{D}(H'_0):=\left\lbrace \Psi\in\Hi_C',\exists \alpha>0,\forall u\in\mathcal{D}({H_0}|_{\Hi_C}),|\langle H_0u,\Psi\rangle_{\Hi_C,\Hi'_C}|\leq\alpha\norme{u}_{\Hi_C}\right\rbrace,
		\end{equation*}
		and 
		\begin{equation*}
			\langle u,H_0'\Psi\rangle_{\Hi_C,\Hi'_C}:=\langle H_0u,\Psi\rangle_{\Hi_C,\Hi'_C}, \quad \Psi\in\mathcal{D}(H_0'), \quad u\in\mathcal{D}({H_0}|_{\Hi_C}).
		\end{equation*}
		In the same way, we can define 
		\begin{equation*}
			\mathcal{D}(H|_{\Hi_C}):=\left\lbrace u\in\mathcal{D}(H_0)\cap\Hi_C,Hu\in\Hi_C\right\rbrace,
		\end{equation*}
		\begin{equation*}
			\langle u,H'\Psi \rangle_{\Hi_C,\Hi'_C}:=\langle H^*u,\Psi\rangle_{\Hi_C,\Hi'_C}, \quad \Psi\in\mathcal{D}(H'), \quad u\in\mathcal{D}({H}|_{\Hi_C}).
		\end{equation*}
		Then it is not difficult to see that $\mathcal{D}({H_0}|_{\Hi_C})=\mathcal{D}(H|_{\Hi_C})$,  $\mathcal{D}(H_0')=\mathcal{D}(H')$, and
		\begin{equation*}
			H'=H_0'+V'=H_0'+CW C'. 
		\end{equation*}
		Clearly, $H'$ is an extension of $H$ since for all $u\in\mathcal{D}(H_0)$, we have 
		\begin{equation*}
			 \forall v\in\mathcal{D}({H_0}|_{\Hi_C}), \quad \abso{\scal{u}{H_0v}_\Hi}\leq\norme{H_0u}_{\Hi}\norme{C}_\Hi\norme{v}_{\Hi_C},
		\end{equation*}
		which implies that $u\in\mathcal{D}(H_0')$ and that $H'u=Hu$, using the anti-linear version of the Riesz representation theorem.
	
		\subsubsection{The resolvent of $H_0$} One of our main hypotheses (see Hypothesis \ref{hyp:principe_absorption_limite}) will imply that the limits
				\begin{equation}\label{eq:limit_abs_H0}
			C\Res_0(\lambda\pm i0^+)C:=\lim_{\varepsilon\rightarrow 0^+}C\Res_0(\lambda\pm i\varepsilon)C
		\end{equation}
		exist for a.e. $\lambda\in\sigma_{\mathrm{ess}}(H_0)$, for the topology of $\mathcal{L}(\Hi)$. In other words, the family of operators $(\mathcal{R}_0(\lambda\pm i\varepsilon))_{\varepsilon>0}$ converges in $\mathcal{L}(\Hi_C,\Hi'_C)$ as $\varepsilon\to0^+$ and its limit is denoted by 
		\begin{equation*}
			\mathcal{R}_0(\lambda\pm i0^+)\in\mathcal{L}(\Hi_C,\Hi'_C).
		\end{equation*}
		
\subsection{Regular spectral points and spectral singularities}

In this section we define the notions of regular spectral points and spectral singularities that we consider in this paper. 

	\begin{definition}[Regular spectral point and spectral singularity]\label{def:point_spectral_regulier_classique_pour_H}
		Let $\lambda\in\sigma_\mathrm{ess}(H)$. 
	\begin{enumerate}[label=(\roman*)]
		\item We say that $\lambda$ is an outgoing/incoming regular spectral point of $H$ if $\lambda$ is not an accumulation point of eigenvalues located in $\lambda\pm i\left( 0,\infty\right)$ and if the limit 
		\begin{equation}\label{eq:def_reg_spec_pt}
			C\Res_H(\lambda\pm i0^+)CW:=\lim_{\varepsilon\rightarrow 0^+} C\Res_H(\lambda\pm i\varepsilon)CW
		\end{equation}
		exists in the norm topology of $\mathcal{L}(\Hi)$. If $\lambda$ is not an outgoing/incoming regular spectral point, we say that $\lambda$ is an outgoing/incoming spectral singularity of $H$.
		\item We say that $\lambda$ is a regular spectral point of $H$ if it is both an incoming and an outgoing regular spectral point of $H$. If $\lambda$ is not a regular spectral point, we say that $\lambda$ is a spectral singularity of $H$.
	\end{enumerate}
	\end{definition}	

Our definition of a spectral singularity is related to that of \cite{Sc60_01} and to the notion of spectral projections for non-self-adjoint operators \cite{Du58_01}, which will also be an important tool in our paper. See \eqref{eq:spectr_proj_I} for the definition of the spectral projection $\mathds{1}_I(H)$ corresponding to a spectral interval $I\subset\sigma_{\mathrm{ess}}(H)$ without spectral singularities. In \cite{Sc60_01}, a spectral singularity corresponds to an exceptional point $\lambda_0$  outside of which the `spectral resolution' $I \mapsto \mathds{1}_I(H)$ is countably additive and uniformly bounded. In our context, this is a weaker requirement than that of Definition \ref{def:point_spectral_regulier_classique_pour_H}, see Section \ref{subsec:spectr_proj}.

It should also be noted that Definition \ref{def:point_spectral_regulier_classique_pour_H} generalizes the definition of spectral singularities considered in the context of dissipative operators in \cite{FaFr18_01,FaNi18_01,Fa21_01}. As we explain below, for Schrödinger operators, $H=-\Delta+V(x)$ on $L^2(\mathbb{R}^d)$ with $V$ a complex, decaying potential, spectral singularities correspond to real resonances.
		
	Various characterizations of the notions introduced in Definition \ref{def:point_spectral_regulier_classique_pour_H} will be given in Section \ref{sec:spectr_sing}. Note in particular that, under the assumption that the limits \eqref{eq:limit_abs_H0} exist, the operator $W$ could have been put to the left of $C\Res_H(\lambda\pm i0^+)C$ in \eqref{eq:def_reg_spec_pt}, i.e. the limits in \eqref{eq:def_reg_spec_pt} exist if and only if the limits $WC\Res_H(\lambda\pm i\varepsilon)C$, $\varepsilon\to0^+$ exist. Moreover, under the same assumption, we will show that $\lambda$ is an outgoing/incoming regular spectral point of $H$ if and only if $\mathrm{Id}+\Res_0(\lambda\pm i0^+)V'$ is invertible in $\mathcal{L}(\Hi_C')$. (Here it should be recalled that $V'\in\mathcal{L}(\Hi_C',\Hi_C)$ and, assuming that the limits \eqref{eq:limit_abs_H0} exist, that $\mathcal{R}_0(\lambda\pm i0^+)\in\mathcal{L}(\Hi_C,\Hi'_C)$). Hence spectral singularities are naturally associated to  resonant states defined as follows.
		\begin{definition}[Incoming/outgoing resonant states]\label{def:resonant}
			Let $\lambda\in\sigma_\mathrm{ess}(H)$ be a spectral singularity of $H$. The space $\Hi_C'^+(\lambda)\subset\Hi'_C$ of outgoing resonant states corresponding to $\lambda$ is defined by
			\begin{equation*}
				\Hi_C'^+(\lambda):=\Ker\left(\Id+\Res_0(\lambda+i0^+)V'\right).
			\end{equation*}
			The space $\Hi_C'^-(\lambda)\subset\Hi'_C$ of incoming resonant states is defined by
			\begin{equation*}
				\Hi_C'^-(\lambda):=\Ker\left(\Id+\Res_0(\lambda-i0^+)V'\right).
			\end{equation*}
		\end{definition} 
		Clearly, as kernels of bounded operators, the vector spaces $\Hi_C'^{\pm}(\lambda)$ are closed. We will prove that eigenvectors associated to embedded eigenvalues of $H$ belong to $\Hi_C'^\pm(\lambda)$.
		
		Note that in the case of complex Schrödinger operators, $H=-\Delta+V(x)$ on $L^2(\mathbb{R}^d)$, $\lambda$ is usually called an outgoing/incoming resonance if the quotient vector space 
		\begin{equation*}
		\Ker\left((\Id+\Res_0(\lambda\pm i0^+)V')|_{\Hi_C'}\right)\slash \Ker\left((\Id+\Res_0(\lambda\pm i0^+)V')|_{\Hi}\right)\neq\lbrace 0\rbrace,
		\end{equation*}
where $(\Id+\Res_0(\lambda\pm i0^+)V')|_{\Hi_C'}$ stands for the restriction of $(\Id+\Res_0(\lambda\pm i0^+)V')$ to $\Hi_C'$, and likewise for  $(\Id+\Res_0(\lambda\pm i0^+)V')|_{\Hi}$. An outgoing/incoming resonant state then corresponds to an element of $\Ker((\Id+\Res_0(\lambda\pm i0^+)V')|_{\Hi_C'})$ which does not belong to $\Hi$. See Section \ref{subsec:appli} for more details. 		
		

\section{Assumptions and main results}\label{sec:results}

	\subsection{Hypotheses}
		In this section we detail our main abstract hypothesis. In Section \ref{subsec:appli} we will show that they are satisfied in the case of complex Schrödinger operators, under suitable assumptions on the potential.
		
		 In our first hypothesis, we require that $H_0$ satisfies a \textit{limiting absorption principe} (with weight $C$) at any point of the essential spectrum.

	\begin{assumption}[Limiting absorption principle for $H_0$]\label{hyp:principe_absorption_limite}
		We have
		\begin{equation}\label{eq:LAP_H0}
			\sup_{z\in\C^\pm}\big\|C\Res_0(z)C\big\|_{\mathcal{L}(\Hi)}<\infty.
		\end{equation}
	\end{assumption}

	Note that \eqref{eq:LAP_H0} implies (see e.g. \cite[Proposition 4.1]{CFKS}) that the spectrum of $H_0$ is purely absolutely continuous, i.e. that $\sigma_{\mathrm{pp}}(H_0)=\emptyset$, $\sigma_{\mathrm{ac}}(H_0)=\sigma(H_0)$, $\sigma_{\mathrm{sc}}(H_0)=\emptyset$, where $\sigma_{\mathrm{pp}}(H_0)$, $\sigma_{\mathrm{ac}}(H_0)$, $\sigma_{\mathrm{sc}}(H_0)$ stand for the usual pure point, absolutely continuous and singular continuous spectra of the self-adjoint operator $H_0$. 
	
	By Fatou's Theorem, \eqref{eq:LAP_H0} yields that the limits $C\Res_0(\lambda\pm i0^+)C$ exist for almost every $\lambda\in\sigma_{\mathrm{ess}}(H)$, in the norm topology of $\mathcal{L}(\Hi)$, and that the map $\R\ni\lambda\mapsto C\Res_0(\lambda\pm i0^+)C\in\mathcal{L}(\Hi)$ is bounded (observe that $C\Res_0(\lambda\pm i0^+)C=C\Res_0(\lambda)C$ if $\lambda\in\R\setminus\sigma_{\mathrm{ess}}(H)$).
	
	Note also that Hypothesis \ref{hyp:principe_absorption_limite} implies (see \cite{Ka65_01} or \cite[Theorem XIII.25 and its corollary]{ReSi80_01}) that $C$ is \emph{relatively smooth} with respect to $H_0$ in the sense of Kato, i.e. that there exists a constant $c_0$ such that 
		\begin{equation}\label{eq:Kato_smooth}
			\forall u\in\Hi, \quad \int_\R\norme{Ce^{-itH_0}u}_\Hi^2\mathrm{d}t\leq c_0^2\norme{u}_\Hi^2.
		\end{equation}
		Recall that \eqref{eq:Kato_smooth} is equivalent to 
		\begin{equation}\label{eq:Smooth_Ka65_01}
			\forall u\in\Hi, \quad \int_\R\left(\norme{C\Res_0(\lambda-i0^+)u}_\Hi^2+\norme{C\Res_0(\lambda+i0^+)u}_\Hi^2\right)\mathrm{d}\lambda\leq 2\pi c_0^2\norme{u}_\Hi^2,
		\end{equation}
		where $\lambda\mapsto C\Res_0(\lambda\pm i0^+)u$ denotes the limit of $\lambda\mapsto C\Res_0(\lambda\pm i\varepsilon)u$ in $L^2(\R;\Hi)$ as $\varepsilon\to0^+$.

		Next we will assume that the point spectral subspace of $H$ is finite.
		
	\begin{assumption}[Eigenvalues of $H$]\label{hyp:valeurs_propres_H}
		$H$ has only a finite number of eigenvalues with finite algebraic multiplicity. 
	\end{assumption}

Hypothesis \ref{hyp:valeurs_propres_H} prevents the essential spectrum of $H$ from having an accumulation point of eigenvalues. It does not exclude, however, the presence of eigenvalues embedded in the essential spectrum of $H$.

	Our next hypothesis concerns the spectral singularities of $H$. We will assume that $H$ has finitely many spectral singularities with a `finite order', in the sense that the map $z\mapsto C\Res_H(z)CW$ blows up at most polynomially as $z$ approaches any spectral singularity $\lambda\in\sigma_{\mathrm{ess}}(H)$. We will also allow for singularities `at infinity', in the sense that $z\mapsto C\Res_H(z)C$ may blow up polynomially as $z$ tends to $\infty$ ($z$ close to the real axis).
	
	\begin{assumption}[Spectral singularities for $H$]\label{hyp:singularités_spectrales}
		$H$ only has a finite number of spectral singularities $\left\lbrace \lambda_1,\ldots\lambda_n\right\rbrace\subset \sigma_{\mathrm{ess}}(H)$ and there exist $\varepsilon_0>0$ and integers $\nu_1,\dots,\nu_n,\nu_\infty\ge0$ such that 
		\begin{equation}\label{eq:limit_uniform}
			\sup_{\mathrm{Re}(z)\in \sigma_{\mathrm{ess}}(H), \pm\mathrm{Im}(z)\in(0,\varepsilon_0)} \frac{1}{|z-z_0|^{\nu_\infty}} \Big( \prod_{j=1}^n\frac{ |z-\lambda_j|^{\nu_j} }{ |z-z_0|^{\nu_j} } \Big ) \big\|C\Res_H(z)CW\big\|_{\mathcal{L}(\Hi)}<\infty,
		\end{equation}
	where $z_0$ is an arbitrary complex number such that $z_0\in\rho(H)$, $z_0\in\C\setminus\R$.
	\end{assumption}
	The factors $|z-\lambda_j|^{\nu_j}$ `regularize' the singularities of $z\mapsto C\Res_H(z)CW$ as $z$ approaches $\lambda_j$. Divided them by $|z-z_0|^{\nu_j}$ produces bounded terms. The factor $|z-z_0|^{-\nu_\infty}$ regularizes a possible singularity at $\infty$.
	
	Observe that since $\lambda_1,\dots,\lambda_n$ are the only spectral singularities of $H$, for all $\lambda\in \sigma_{\mathrm{ess}}(H)\setminus\{\lambda_1,\dots,\lambda_n\}$, the limits $C\Res_H(\lambda\pm i0^+)CW$ exist in the norm topology of $\mathcal{L}(\Hi)$. The condition \eqref{eq:limit_uniform} then implies that the maps
	\begin{equation}
		\sigma_{\mathrm{ess}}(H)\setminus\{\lambda_1,\dots,\lambda_n\} \ni \lambda \mapsto \frac{1}{|\lambda-z_0|^{\nu_\infty}} \Big( \prod_{j=1}^n\frac{ |\lambda-\lambda_j|^{\nu_j} }{ |\lambda-z_0|^{\nu_j} } \Big ) C\Res_H(\lambda\pm i0^+)C \in \mathcal{L}(\Hi)
	\end{equation}
	are bounded. Since, as we will show below (see Proposition \ref{cor:embed}), embedded eigenvalues are special spectral singularities, Hypothesis \ref{hyp:singularités_spectrales} is also a condition on embedded eigenvalues.

	As mentioned above, to study the absolutely continuous spectral subspace of $H$, we require the existence of a conjugation operator $J$ satisfying, in particular, $JH=H^* J$. 
	\begin{assumption}[Conjugation operator]\label{hyp:Conjugate_operator}
		There exists an anti-linear continuous map $J:\Hi\to\Hi$ such that 
		\begin{enumerate}[label=(\roman*)]
			\item $J^2=\mathrm{Id}$,
			\item $\forall u,v\in\Hi$, $\scal{Ju}{Jv}_\Hi=\scal{v}{u}_\Hi$,
			\item $J\mathcal{D}(H_0)\subset\mathcal{D}(H_0)$ and $\forall u\in\mathcal{D}(H_0)$, $JH_0u=H_0Ju$.
			\item $JC=CJ$ and $JW=W^* J$. 
		\end{enumerate}
		Moreover, for all embedded eigenvalues $\lambda\in\sigma_{\mathrm{ess}}(H)$, the symmetric bilinear form
		\begin{equation}\label{eq:invertibility_mat2}
			\Ker\big((H-\lambda)^{\mathrm{m}_\lambda}\big) \ni (u,v)\mapsto \langle Ju,v\rangle \quad\text{is non-degenerate}.
		\end{equation}
	\end{assumption}
	Here it should be recalled (see Section \ref{subsec:subspaces}) that the main purpose of \eqref{eq:invertibility_mat2} is to allow us to define suitable spectral projections for embedded eigenvalues.

	Our last technical hypothesis is required in order for the anti-dual operators $H_0'$, $H'$ to be well-defined, see Section \ref{subsec:extension} for more details.
	\begin{assumption}\label{hyp:domaine_H_0_C}
		The domain of the restriction of $H_0$ to $\Hi_C$, defined as
			\begin{equation*}
				\mathcal{D}({H_0}|_{\Hi_C}):=\left\lbrace u\in\mathcal{D}(H_0)\cap\Hi_C, H_0u\in\Hi_C\right\rbrace,
			\end{equation*}
		is dense in $\Hi_C$ for the topology of $\Hi_C$. 
	\end{assumption}

	\subsection{Main results}

		Now we can state our main results. First, we characterize the outgoing (respectively incoming) spectral singularities of $H$ as eigenvalues of $H'$ associated to eigenvectors belonging to the space of outgoing resonant states $\Hi_C'^+(\lambda)$ (respectively $\Hi_C'^-(\lambda)$).
		
		\begin{theorem}\label{thm:spec_sing}
			Suppose that Hypothesis \ref{hyp:domaine_H_0_C} holds. Let $\lambda\in \sigma_\mathrm{ess}(H)$ and suppose that the limits
		\begin{equation*}
			C\Res_0(\lambda \pm i0^+)C:=\lim_{\varepsilon\rightarrow 0^+}C\Res_0(\lambda \pm i\varepsilon)C
		\end{equation*}
		exist in the norm topology of $\mathcal{L}(\Hi)$. The following conditions are equivalent:
				\begin{enumerate}[label=(\roman*)]
					\item\label{it:spec_sing} $\lambda$ is an outgoing/incoming spectral singularity of $H$,
					\item\label{it:res_state} $\lambda$ is an eigenvalue of $H'$ associated to an eigenvector $\Psi\in\Hi_C'^\pm(\lambda)$. 
				\end{enumerate}
		\end{theorem}
		Theorem \ref{thm:spec_sing} shows that $\lambda$ is a spectral singularity of $H$ if and only if the equation $H'\Psi=\lambda\Psi$ has a solution $\Psi$ in $\Hi_C'^\pm(\lambda)$ (recall that $H'$ is an extension of $H$, acting in the Hilbert space $\Hi'_C$ which contains the original Hilbert space $\Hi$). For non-self-adjoint Schrödinger operators, this corresponds to $\lambda$ being a real resonance of $H=-\Delta+V(x)$ if and only if the equation $(-\Delta+V(x))\Psi=\lambda\Psi$ has a distributional solution such that $\Psi$ belongs to suitable weighted $L^2$-spaces. See Section \ref{subsec:appli} for a more detailed discussion.
		
		Theorem \ref{thm:spec_sing} also has various consequences that we will detail in Section \ref{sec:spectr_sing}. Here we mention the following two consequences. We will show that the eigenvalues of $H$ embedded in the essential spectrum are special spectral singularities, see Proposition \ref{cor:embed}. Moreover, in the particular case where $H$ is dissipative, we will prove that $H$ cannot have outgoing spectral singularities unless its self-adjoint part, $\mathrm{Re}(H)$, already has some, see Section \ref{sec:dissipative case}.

		Our next result shows that the subspace $\Hi_\mathrm{ads}^\pm(H)$ of asymptotically disappearing states at $\pm\infty$ (recall that $\Hi_\mathrm{ads}^\pm(H)$ has been defined in Section \ref{subsec:subspaces}) coincides with the vector space spanned by all generalized eigenstates of $H$ corresponding to eigenvalues $\lambda$ of $H$ such that $\pm\mathrm{Im}(\lambda)<0$. In other words, the only solutions to \eqref{eq:Schrodinger} that vanish at as $t\to\pm\infty$ are linear combination of generalized eigenstates corresponding to non-real eigenvalues.

		\begin{theorem}\label{thm:caracterisation_etat_disparaissent_à_infini}
			Suppose that Hypotheses \ref{hyp:principe_absorption_limite}-\ref{hyp:singularités_spectrales} hold. Then 
			\begin{equation}
				\Hi_\mathrm{ads}^\pm(H)=\Hi_\mathrm{p}^\pm(H).
			\end{equation}
		\end{theorem}
		An analogous result was proven in \cite{FaFr18_01} in the particular case of dissipative operators, answering a question left as an open problem in \cite{Da80_01}. The proof in \cite{FaFr18_01} relies in an essential way on the existence and properties of wave operators. Besides the fact that we are considering non-dissipative operators, our proof here is more direct -- we do not use scattering theory -- and allows for more general assumptions (compare Hypothesis 2.5 in \cite{FaFr18_01} and Hypothesis \ref{hyp:singularités_spectrales} of the present paper, where a singularity at infinity of the weighted resolvent is allowed). The core of our argument is a suitable spectral decomposition formula, see Proposition \ref{prop:resolution} in the next subsection.

		Next we will prove that the absolutely continuous spectral subspace of $H$ defined in Section \ref{subsec:subspaces} coincides with the orthogonal complement of the point spectral subspace.
		
		 \begin{theorem}\label{thm:caracterisation_espace_absolument_continu}
		 	Suppose that Hypotheses \ref{hyp:principe_absorption_limite}-\ref{hyp:Conjugate_operator} hold. Then
		 	\begin{equation*}
		 		\Hi_\mathrm{ac}(H)=\Ran(\mathrm{Id}-\Pi_\mathrm{p}(H))=\Hi_\mathrm{p}(H^*)^\perp.
		 	\end{equation*}
		 \end{theorem}
		 	As mentioned before, in the particular case where $H$ is dissipative, Theorem \ref{thm:caracterisation_espace_absolument_continu} may be compared with a result of \cite{Da80_01}. The absolutely continuous spectral subspace for dissipative operators is defined in \cite{Da80_01} in the same way as in \eqref{eq:def_M(H)}, but with the integral taken over $[0,\infty)$ instead of $\mathbb{R}$. Using the theory of dilations of dissipative operators, it is then proven in \cite{Da80_01} that within such a definition, the absolutely continuous spectral subspace coincides with the orthogonal complement of `bound states' (generalized eigenstates corresponding to real eigenvalues). In our context where $H$ is not necessarily dissipative, the argument of \cite{Da80_01} fails and there is no reason to choose positive times over negative times to define $\Hi_\mathrm{ac}(H)$.
			
			Our definition of $\Hi_\mathrm{ac}(H)$ is also justified by the following two facts. First, Theorem \ref{thm:caracterisation_espace_absolument_continu} generalizes the well-known identity
			$
				\Hi_\mathrm{ac}(H) = \Hi_{\mathrm{p}}(H)^\perp
			$
			which holds for self-adjoint operators without singular continuous spectrum. Another justification comes from dissipative scattering theory. Indeed, combined with the results of \cite{Ma75_01,Da78_01,Da80_01,FaFr18_01}, Theorem \ref{thm:caracterisation_espace_absolument_continu} implies that, in the particular case where $H$ is dissipative, we have $\mathrm{Ran}(W_-(H,H_0))^{\mathrm{cl}}=\Hi_\mathrm{ac}(H)$, where
			$W_-(H,H_0) := \slim \, e^{-\mathrm{i}tH} e^{\mathrm{i}tH_0}$, $t\to\infty$, is the usual wave operator. This again generalizes the well-known relation which holds in the self-adjoint case.
			
			We mention that if $H$ has no embedded eigenvalues, Hypothesis \ref{hyp:Conjugate_operator} can be dropped in the statement of Theorem \ref{thm:caracterisation_espace_absolument_continu}. This is also the case if $H$ is supposed to be dissipative (see Proposition \ref{prop:ac-diss}). In the general case, it is however clear that some assumption should be added to treat the pathological case where the map in \eqref{eq:invertibility_mat2} is degenerate. Indeed, considering the simplest case where $\mathrm{m}_\lambda=1$, if $\langle J\varphi,\varphi\rangle=0$ for any $\varphi\in\mathrm{Ker}(H-\lambda)$, then one can check that $\mathrm{Ker}(H-\lambda)\subset\Hi_{\mathrm{p}}(H^*)^\perp$. Therefore, since $\Hi_\mathrm{ac}(H)$ should not contain eigenstates of $H$, we do not expect that the equality $\Hi_\mathrm{ac}(H)=\Hi_\mathrm{p}(H^*)^\perp$ holds. Another possibility to handle such pathological cases might be to suitably modify the definition of $\Hi_\mathrm{ac}(H)$. We do not consider this possibility here.
			
			To prove Theorem \ref{thm:caracterisation_espace_absolument_continu}, we cannot rely on the theory of unitary dilations as in the dissipative case studied in \cite{Da80_01}. In the same way as for Theorem \ref{thm:caracterisation_etat_disparaissent_à_infini}, our proof of Theorem \ref{thm:caracterisation_espace_absolument_continu} relies on the spectral decomposition formula stated in Proposition \ref{prop:resolution}.

		The next remark gives the $J$-orthogonal spectral decomposition of the Hilbert space mentioned in the introduction.
\begin{remark}		 
		 Combining Theorems \ref{thm:caracterisation_etat_disparaissent_à_infini} and \ref{thm:caracterisation_espace_absolument_continu}, using in addition that $J:\Hi_{\mathrm{p}}(H)\to\Hi_{\mathrm{p}}(H^*)$ is bijective and that the map $\Hi_{\mathrm{p}}(H) \ni (u,v)\mapsto \langle Ju,v\rangle$ is non-degenerate, we obtain the following $J$-orthogonal direct sum decompositions of the Hilbert space: 
		 \begin{align*}
		 	\mathcal{H} &= \mathcal{H}_{\mathrm{ac}}(H) \oplus \mathcal{H}_{\mathrm{p}}(H) \\
			&= \mathcal{H}_{\mathrm{ac}}(H) \oplus \Hi_{\mathrm{disc}}(H)\oplus\Hi_{\mathrm{emb}}(H) \\
			&= \mathcal{H}_{\mathrm{ac}}(H) \oplus \Hi_{\mathrm{ads}}^+(H)\oplus\Hi_{\mathrm{ads}}^-(H)\oplus\Hi_{\mathrm{disc}}^0(H)\oplus\Hi_{\mathrm{emb}}(H),		 
		 \end{align*}
		the direct sum $\Hi_{\mathrm{disc}}^0(H)\oplus\Hi_{\mathrm{emb}}(H)=:\mathcal{H}_{\mathrm{b}}(H)$ being the space of `bound states', i.e. the closure of the vector space spanned by all generalized eigenvectors of $H$ corresponding to real eigenvalues (either isolated or embedded). 
		\end{remark}

		\subsection{Application to Schrödinger operators}\label{subsec:appli}
		We suppose in this section that
		\begin{equation*}
			\Hi=L^2(\mathbb{R}^d), \quad H_0=-\Delta \quad\text{and}\quad V\text{ is a complex-valued potential.}
		\end{equation*}
		For simplicity, we suppose that $d=3$.
		
		\subsubsection{Spectral singularities for short-range complex potentials}
		We begin with showing that Theorem \ref{thm:spec_sing} can be applied to $H=-\Delta+V(x)$, under a short-range condition on $V$. It is well-known that the limits
		\begin{equation}\label{eq:LAP-Delta}
			\langle x\rangle^{-s}\Res_0(\lambda\pm i0^+)\langle x\rangle^{-s} , \quad \lambda >0,
		\end{equation}
		exist in the norm topology of $\mathcal{L}(\Hi)$, for any $s>\frac12$, where $\langle x\rangle:=(1+x^2)^{\frac12}$. Hence, assuming that $V$ satisfies the short-range condition
		\begin{equation}\label{eq:short-range}
			x\mapsto \langle x\rangle^{\sigma}V(x) \in L^\infty(\mathbb{R}^3), \quad \sigma>1,
		\end{equation}
		we can choose $C$ to be the multiplication operator by $\langle x\rangle^{-\sigma/2}$. The Hilbert space $\Hi_C$ then identifies with the weighted $L^2$-space
		\begin{equation*}
			\Hi_C=L^2_{\sigma/2}:=\big\{f :\mathbb{R}^3\to\C , \, x\mapsto \langle x\rangle^{\frac{\sigma}{2}}f(x) \in L^2(\mathbb{R}^3) \big \},
		\end{equation*}
		equipped with the usual norm, its dual being given by
		\begin{equation*}
			\Hi'_C=L^2_{-\sigma/2}=\big\{f :\mathbb{R}^3\to\C , \, x\mapsto \langle x\rangle^{-\frac{\sigma}{2}}f(x) \in L^2(\mathbb{R}^3) \big \}.
		\end{equation*}
		Since the set of smooth functions with compact supports is contained in $\mathcal{D}(H_0)\cap \Hi_C$, one easily deduces that Hypothesis \ref{hyp:domaine_H_0_C} is satisfied. Applying Theorem \ref{thm:spec_sing}, we thus obtain the following proposition.
		\begin{proposition}\label{prop:3.4}
			Suppose that $V$ is a complex-valued potential satisfying \eqref{eq:short-range}. Let $C(x)=\langle x\rangle^{-\sigma/2}$. Then for all $\lambda>0$, the following conditions are equivalent
				\begin{enumerate}[label=(\roman*)]
					\item\label{it:spec_sing_schr} $\lambda$ is an outgoing/incoming spectral singularity of $H$ in the sense of Definition \ref{def:point_spectral_regulier_classique_pour_H},
					\item\label{it:res_state_schr} There exists $\Psi\in\Hi_C'^\pm(\lambda) \subset L^2_{-\sigma/2}$, $\Psi\neq0$, such that
						\begin{equation*}
							(-\Delta+V(x)-\lambda)\Psi=0.
						\end{equation*} 
				\end{enumerate}
		\end{proposition}
		If $\Psi$ in \ref{it:res_state_schr} belongs to $L^2(\mathbb{R}^3)$, $\lambda$ is an eigenvalue of $H$. Otherwise, $\lambda$ is usually called a real resonance associated to a resonant state $\Psi\in\Hi_C'^\pm(\lambda)\setminus L^2(\mathbb{R}^3)$. Such a resonant state satisfies the outgoing/incoming Sommerfeld radiation condition
		\begin{equation*}
u(x)=|x|^{\frac12 } e^{\pm i \lambda^{\frac12} | x | } \Big ( a ( \frac{x}{|x|} ) + o(1) \Big ) , \quad |x| \to \infty ,
\end{equation*}
with $a \in L^2( S^2 )$, $a\neq 0$.
		
		Note that if $V$ is real-valued, Agmon's fundamental work \cite{Ag75_01} shows that $H=-\Delta+V$ has no spectral singularities in $(0,\infty)$. Likewise, if $H$ is dissipative, a simple argument combined with \cite{Ag75_01} proves that $H$ cannot have outgoing spectral singularities in $(0,\infty)$ (see \cite{Wa12_01}). In general, however, spectral singularities cannot be excluded (see again \cite{Wa12_01} for an example showing that, for any $\lambda>0$, there exists a smooth, compactly supported potential $V$ such that $\lambda$ is an incoming spectral singularity of $H$ in the dissipative case).
		
		At the threshold energy $0$, the limiting absorption principle states that the limits
		\begin{equation*}
			\langle x\rangle^{-s}\Res_0(\pm i0^+)\langle x\rangle^{-s} , 
		\end{equation*}
		exist in the norm topology of $\mathcal{L}(\Hi)$ for any $s>1$. Note that, as we will argue in Section \ref{sec:spectr_sing}, the two limits above in fact coincide with the limit
		\begin{equation*}
			\langle x\rangle^{-s}\Res_0(0)\langle x\rangle^{-s} := \lim_{z\to0,z\in\mathbb{C}\setminus\mathbb{R}_+}\langle x\rangle^{-s}\Res_0(z)\langle x\rangle^{-s} . 
		\end{equation*}		
		(Note also that the limit $\langle x\rangle^{-s}\Res_0(0)\langle x\rangle^{-s'}$ exists, more generally, provided that $s,s'>1/2$ and $s+s'>2$, see \cite{JeKa79_01}.) Strengthening the short-range condition \eqref{eq:short-range} to
		\begin{equation}\label{eq:short-range-zero}
			x\mapsto \langle x\rangle^{\sigma}V(x) \in L^\infty(\mathbb{R}), \quad \sigma>2,
		\end{equation}
		we can choose $C$ to be the multiplication operator by $\langle x\rangle^{-\sigma/2}$ and proceed as before. Since, as we will see in Section \ref{sec:spectr_sing}, $0$ is an outgoing spectral singularity of $H$ if and only if it is an incoming spectral singularity of $H$, this gives the following proposition.
		\begin{proposition}
			Suppose that $V$ is a complex-valued potential satisfying \eqref{eq:short-range-zero}. Let $C(x)=\langle x\rangle^{-\sigma/2}$. Then the following conditions are equivalent
				\begin{enumerate}[label=(\roman*)]
					\item\label{it:spec_sing_schr-zero} $0$ is a spectral singularity of $H$ in the sense of Definition \ref{def:point_spectral_regulier_classique_pour_H},
					\item\label{it:res_state_schr-zero} There exists $\Psi\in\Hi_C'^\pm(\lambda) \subset L^2_{-\sigma/2}$ such that
						\begin{equation*}
							(-\Delta+V(x))\Psi=0.
						\end{equation*} 
				\end{enumerate}
		\end{proposition}

Anticipating results that we will prove in the abstract setting in the case where $H$ is dissipative (see Section \ref{sec:dissipative case}), we also have the following proposition.
		\begin{proposition}\label{prop:schr_diss}
			Suppose that $V$ is a complex-valued potential such that $\mathrm{Im}(V)\le0$ and $\mathrm{Im}(V)<0$ on a non-trivial open set. Suppose that \eqref{eq:short-range} holds and let $C(x)=\langle x\rangle^{-\sigma/2}$. Then 
			\begin{equation*}
			H \text{ has no positive outgoing spectral singularities}.
			\end{equation*}
			In particular, $H$ has no positive embedded eigenvalues. Suppose in addition that \eqref{eq:short-range-zero} holds. Then 
			\begin{equation*}
				0 \text{ is not a spectral singularity of }H.
			\end{equation*}
		\end{proposition}
		See Section \ref{sec:dissipative case} for a proof of Proposition \ref{prop:schr_diss}. Note also that the results of Proposition \ref{prop:schr_diss} have been established in \cite{Wa11_01,Wa12_01}, using different arguments.

		\subsubsection{Asymptotically disappearing states for compactly supported complex potentials}
		We now show that Theorems \ref{thm:caracterisation_etat_disparaissent_à_infini} and \ref{thm:caracterisation_espace_absolument_continu} can be applied to $H=-\Delta+V(x)$ under the condition that
\begin{equation}\label{eq:V_comp_supp}
V \in L^\infty_{ \mathrm{c} }( \mathbb{R}^3 ) := \{ u \in L^\infty( \mathbb{R}^3 ) , \, u \text{ is compactly supported} \} .
\end{equation}

Similarly as in \eqref{eq:LAP-Delta}, we have
		\begin{equation*}
			\sup_{z\in\C^\pm}\big\|\langle x\rangle^{-s}\Res_0(z)\langle x\rangle^{-s}\big\|_{\mathcal{L}(\Hi)}<\infty,
		\end{equation*}
for any $s>1/2$. Hence, choosing $C(x)=\langle x\rangle^{-s}$, Hypothesis \ref{hyp:principe_absorption_limite} is satisfied.
		
Assuming \eqref{eq:V_comp_supp}, it is known that $H$ has only finitely many eigenvalues with finite algebraic multiplicities. See, e.g., \cite{FrLaSa16_01} and references therein. In particular, Hypothesis \ref{hyp:valeurs_propres_H} is satisfied.

To verify that Hypothesis \ref{hyp:singularités_spectrales} holds, we can rely on the theory of resonances, defined as poles of the meromorphic extension of the weighted resolvent, see e.g. \cite{DyZw19_01}. Assuming \eqref{eq:V_comp_supp}, the map 
\begin{equation*}
\{ z \in \mathbb{C} , \mathrm{Im}( z ) > 0 \} \ni z \mapsto ( H - z^2 )^{-1} : L^2( \mathbb{R}^3 ) \to L^2( \mathbb{R}^3 ) 
\end{equation*}
is meromorphic and extends to a meromorphic map
\begin{equation}
\mathbb{C} \ni z \mapsto \Res(z^2) : L^2_{ \mathrm{c} }( \mathbb{R}^3 ) \to L^2_{ \mathrm{loc} }( \mathbb{R}^3 ) , \label{eq:mero_ext}
\end{equation}
where $L^2_{\mathrm{c}}(\mathbb{R}^3):=\{u\in L^2(\mathbb{R}^3),u\text{ is compactly supported} \}$ and $L^2_{\mathrm{loc}}( \mathbb{R}^3 ) :=\{u:\mathbb{R}^3\to\mathbb{C},u\in L^2(K) \text{ for all compact set } K \subset \mathbb{R}^3 \}$. Poles of the map in \eqref{eq:mero_ext} are called resonances of $H$. One then verifies that a real resonance $\pm\lambda_0$ of $H$, with $\lambda_0\ge0$, corresponds to an outgoing/incoming spectral singularity $\lambda_0^2$ in the sense of Definition \ref{def:point_spectral_regulier_classique_pour_H}. Moreover, $H$ has only finitely many spectral singularities $\{\lambda_1,\dots,\lambda_n\}$ and Hypothesis \ref{hyp:singularités_spectrales} is satisfied, with $\nu_j$ the multiplicity of the corresponding resonances $\pm\sqrt{\lambda_j}$, and $\nu_\infty=0$. See e.g. \cite{DyZw19_01} and references therein for an exposition of the theory of resonances of Schrödinger operators, and \cite[Section 6]{FaFr18_01} for a more detailed comparison between the notions of resonances and spectral singularities considered in this paper.
		
		Applying Theorem \ref{thm:caracterisation_etat_disparaissent_à_infini}, we obtain the following result.
		\begin{proposition}\label{prop:schr_ads}
			Suppose that $V$ is a complex-valued potential such that $V \in L^\infty_{ \mathrm{c} }( \mathbb{R}^3 )$. Then 
			\begin{equation*}
			\Hi^\pm_{\mathrm{ads}}(H)=\Hi_{\mathrm{p}}^\pm(H).
			\end{equation*}
		\end{proposition}
		
		To apply Theorem \ref{thm:caracterisation_espace_absolument_continu}, we need to verify in addition that Hypothesis \ref{hyp:Conjugate_operator} holds. Clearly, we can take the conjugation operator $J$ as the complex conjugation. We then obtain
		\begin{proposition}\label{prop:schr_H-ac}
			Suppose that $V$ is a complex-valued potential such that $V \in L^\infty_{ \mathrm{c} }( \mathbb{R}^3 )$. Assume that, for all embedded eigenvalues $\lambda\in[0,\infty)$, the symmetric bilinear form
		\begin{equation}\label{eq:non-deg-Schr}
			\Ker\big((H-\lambda)^{\mathrm{m}_\lambda}\big) \ni (u,v)\mapsto \int_{\mathbb{R}^3}u(x)v(x)\mathrm{d}x \quad\text{is non-degenerate}.
		\end{equation}
		Then
		 	\begin{equation}\label{eq:Hac-Schr}
		 		\Hi_\mathrm{ac}(H)=\Hi_\mathrm{p}(H^*)^\perp.
		 	\end{equation}
		\end{proposition}
Obviously, if $H$ has no embedded eigenvalues (in particular, if $H$ is dissipative), Condition \eqref{eq:non-deg-Schr} can be dropped. Otherwise, as explained before, \eqref{eq:non-deg-Schr} seems necessary for \eqref{eq:Hac-Schr} to hold. We mention that an assumption comparable to \eqref{eq:Hac-Schr} has been used, for thresholds eigenvalues, in the recent works \cite{Aa21_01,Wa20_01} to study the large-time behaviors of solutions to Schrödinger equations with a complex potentials.
		
		\subsection{Organisation of the paper and ingredients of the proof}\label{subsec:ideas}	
		
		The proof of Theorem \ref{thm:spec_sing} (given in Section \ref{sec:proof_thm_spec_sing}) relies, in particular, on a `boundary value version' of the \textit{Birmann-Schwinger principle} that we state and prove in Section \ref{sec:Birmann}, see Proposition \ref{lm:equivalence_point_spectral_regulier}. This proposition extends a related result for dissipative operators proven in \cite[Lemma 4.1]{FaNi18_01}. Further results concerning spectral singularities are proven in Section \ref{sec:spectr_sing}.

		As mentioned in the previous section, the main ingredient in the proofs of Theorems \ref{thm:caracterisation_etat_disparaissent_à_infini} and \ref{thm:caracterisation_espace_absolument_continu} is a \emph{spectral decomposition formula} suitably modified to take into account the spectral singularities $\lbrace \lambda_j\rbrace_{j=1}^n$ of $H$. It can be stated as follows. Assuming Hypothesis \ref{hyp:singularités_spectrales}, with $\nu_1,\dots,\nu_n,\nu_\infty$ defined by this hypothesis, we  set, for all $z\in\C\setminus\{z_0\}$,
\begin{equation}\label{eq:def_r}
r(z) := (z-z_0)^{-(\nu_1+\cdots+\nu_n+\nu_\infty)} \prod_{j=1}^n (z-\lambda_j)^{\nu_j},
\end{equation}
where we recall that $z_0\in\rho(H)$, $z_0\in\C\setminus\R$. We then write
\begin{equation*}
r(H) = \Res_H(z_0)^{-(\nu_1+\cdots+\nu_n+\nu_\infty)} \prod_{j=1}^n (H-\lambda_j)^{\nu_j},
\end{equation*}
which defines a bounded operator in $\mathcal{L}(\Hi)$. Note that if $\lambda_j$ is an embedded eigenvalue of $H$, then for any generalized eigenstate $\varphi_j$ corresponding to $\lambda_j$, we have $r(H)\varphi_j=0$ (provided that $\nu_j$ is large enough). We will prove the following proposition.

\begin{proposition}\label{prop:resolution}
Suppose that Hypotheses \ref{hyp:principe_absorption_limite}, \ref{hyp:valeurs_propres_H} and \ref{hyp:singularités_spectrales} hold. Then
\begin{align}
r(H) = r(H)\Pi_{\mathrm{disc}}(H) + \wlim_{\varepsilon\rightarrow 0^+}\frac{1}{2\pi i}\int_{\sigma_{\mathrm{ess}}(H)}r(\lambda)\big(\Res_{H}(\lambda+i\varepsilon)-\Res_{H}(\lambda-i\varepsilon)\big)\mathrm{d}\lambda . \label{eq:resolution}
\end{align}
\end{proposition}
Proposition \ref{prop:resolution} generalizes the well-known resolution of the identity formula for self-adjoint operators to a class of non self-adjoint operators with finitely many spectral singularities. In particular, if $H$ has no spectral singularities and $\nu_\infty=0$ in Hypothesis \ref{hyp:singularités_spectrales}, then we can take $r=1$ and \eqref{eq:resolution} reduces to
\begin{align*}
\mathrm{Id} = \Pi_{\mathrm{disc}}(H) + \wlim_{\varepsilon\rightarrow 0^+}\frac{1}{2\pi i}\int_{\sigma_{\mathrm{ess}}(H)}\big(\Res_{H}(\lambda+i\varepsilon)-\Res_{H}(\lambda-i\varepsilon)\big)\mathrm{d}\lambda ,
\end{align*}		
which corresponds to Stone's formula in the particular case where $H$ is self-adjoint.

Equation \eqref{eq:resolution} is also related to the notion of spectral projections for non-self-adjoint operators \cite{Du58_01,Sc60_01,DuSc71_01}, defined by
\begin{equation}\label{eq:spectr_proj_I}
\mathds{1}_I(H) := \wlim_{\varepsilon\to0^+} \frac{1}{ 2 i \pi } \int_I \big ( \Res_H ( \lambda + i \varepsilon ) - \Res_H ( \lambda - i \varepsilon ) \big ) \mathrm{d} \lambda , 
\end{equation}
where $I \subset \sigma_{\mathrm{ess}}(H)$ is a closed interval without spectral singularities. We mention that such spectral projections were used in a stationary approach to non-unitary scattering theory, for differential operators in \cite{Mo67_01,Mo68_01}, and in an abstract setting in \cite{Go70_01,Go71_01,Hu71_01}.

We will recall in Section \ref{sec:spectr_res} that the spectral projections \eqref{eq:spectr_proj_I} are well-defined on intervals without spectral singularities, and show that they induce a bounded Borel functional calculus. In intervals containing spectral singularities, we will construct a `regularized' functional calculus, which in turn allows us to prove Proposition \ref{prop:resolution}. Based on the latter, the proofs of Theorems \ref{thm:caracterisation_etat_disparaissent_à_infini} and \ref{thm:caracterisation_espace_absolument_continu} are given in Sections \ref{subsec:ads} and \ref{subsec:abs_cont}, respectively.

Some extensions of results already appearing in the literature are collected in appendices.
		
		
\section{Spectral singularities}\label{sec:spectr_sing}

In this section we prove various characterizations of our definition of spectral singularities (see Definition \ref{def:point_spectral_regulier_classique_pour_H}). We will consider an arbitrary $\lambda\in\sigma_{\mathrm{ess}}(H)$. Our main assumption will be that the limits
		\begin{equation}\label{eq:LAP_lambda}
			C\Res_0(\lambda\pm i0^+)C:=\lim_{\varepsilon\rightarrow 0^+}C\Res_0(\lambda \pm i\varepsilon)C.
		\end{equation}
exist in the norm topology of $\mathcal{L}(\Hi)$.

We begin in Section \ref{sec:Birmann} with a characterization of spectral singularities analogous to the Birmann-Schwinger principle for eigenvalues. Next we prove Theorem \ref{thm:spec_sing} in Section \ref{sec:proof_thm_spec_sing}. In Section \ref{sec:spec_sing_adjoint}, we define the set of spectral singularities for the adjoint operator $H^*$ and show that it coincides with the set of spectral singularities of $H$. Section \ref{sec:embed} proves that eigenvalues embedded in the essential spectrum of $H$ can be seen as particular spectral singularities. In Section \ref{subsec:local}, assuming that \eqref{eq:LAP_lambda} is regular with respect to $\lambda$ in a suitable sense, we show that the notion of spectral regularity introduced in Definition \ref{def:point_spectral_regulier_classique_pour_H} is a local property. We also consider the special case of spectral singularities located at thresholds of the essential spectrum, and show that in this case outgoing and incoming spectral singularities coincide.  Finally, Section \ref{sec:dissipative case} is devoted to the particular case where $H$ is a dissipative operator.

\subsection{Birmann-Schwinger principle for spectral singularities}\label{sec:Birmann}

		Assuming that $H_0$ satisfies a limiting absorption principle at $\lambda$, as stated in \eqref{eq:LAP_lambda}, we have the following characterizations of the definition of a regular spectral point.  Item \ref{lm:equivalence_point_spectral_regulier_def_inversible} can be seen as a `boundary value' version of the Birmann-Schwinger principle (see e.g. \cite{GeLaMiZi,BeELGE21_01} and references therein). The proof of the next proposition is a quite straightforward extension to that of \cite[Lemma 4.1]{FaNi18_01}, where the result is proven for dissipative operators. It is therefore deferred to Appendix \ref{App:Propriétés_point_spect_reg}.

		\begin{proposition}\label{lm:equivalence_point_spectral_regulier}
			Let $\lambda\in \sigma_\mathrm{ess}(H)$ and suppose that the limits \eqref{eq:LAP_lambda} exist in the norm topology of $\mathcal{L}(\Hi)$. Then the following conditions are equivalent:
			\begin{enumerate}[label=(\roman*)]
				\item\label{lm:equivalence_point_spectral_regulier_def_limite} $\lambda$ is an outgoing/incoming regular spectral point of $H$,
				\item\label{lm:equivalence_point_spectral_regulier_def_inversible} $\Id+C\Res_0(\lambda\pm i0^+)CW$ is invertible in $\mathcal{L}(\Hi)$,
				\item\label{lm:equivalence_point_spectral_regulier_def_inversible H_C} $\mathrm{Id}+\Res_0(\lambda\pm i0^+)V'$ is invertible $\mathcal{L}(\Hi_C')$. 
			\end{enumerate}
		\end{proposition}

\begin{proof}
See Appendix \ref{App:Propriétés_point_spect_reg}.
\end{proof}

It should be noted that, since $C$ is relatively compact with respect to $H_0$, the operator $C\Res_0(\lambda\pm i\varepsilon)CW$ is compact in $\mathcal{L}(\Hi)$, for all $\varepsilon>0$. Hence $C\Res_0(\lambda\pm i0^+)CW$ is also compact in $\mathcal{L}(\Hi)$. By the Fredholm alternative, \ref{lm:equivalence_point_spectral_regulier_def_inversible} is then equivalent to $\mathrm{Ker}(\Id+C\Res_0(\lambda\pm i0^+)CW)=\{0\}$. Likewise, \ref{lm:equivalence_point_spectral_regulier_def_inversible H_C} is equivalent to $\mathrm{Ker}(\mathrm{Id}+\Res_0(\lambda\pm i0^+)V')=\{0\}$ since $\Res_0(\lambda\pm i0^+)V'$ is compact in $\mathcal{L}(\Hi'_C)$. 

	\subsection{Proof of Theorem \ref{thm:spec_sing}} \label{sec:proof_thm_spec_sing}

	Now we turn to the proof of Theorem \ref{thm:spec_sing}, which characterizes outgoing/incoming spectral singularities as eigenvalues of the extended operator $H'$ corresponding to eigenvectors belonging to the space $\Hi_C^\pm(\lambda)$ of outgoing/incoming resonant states.
	
	   Before proving Theorem \ref{thm:spec_sing}, we need two preliminary lemmas. The first one is the following well-known estimate of the operator norm $\|\mathcal{R}_0(\lambda\pm i\varepsilon)C\|$, assuming that the limits \eqref{eq:LAP_lambda} exist.

	\begin{lemma}\label{lm:tech_est}
					Let $\lambda\in \sigma_\mathrm{ess}(H)$ and suppose that the limits \eqref{eq:LAP_lambda} exist in the norm topology of $\mathcal{L}(\Hi)$. There exists $c_0>0$ such that 
		\begin{equation}\label{eq:Estimé_norme_C_resolvante_H_0}
			\forall \varepsilon>0, \quad \norme{\Res_0\left(\lambda\pm i\varepsilon\right)C}_\mathcal{\mathcal{L}(\Hi)}\leq c_0\varepsilon^{-\frac12}.
		\end{equation}
	\end{lemma}
	
	\begin{proof}
	See Appendix \ref{App:Propriétés_point_spect_reg}.
	\end{proof} 
	
	Next we show that, under our assumptions, $\Res_0(\lambda\pm i0^+)$ are right inverses of $H_0'-\lambda$.
	
		\begin{lemma}\label{lm:Inverse_à_droite}
			Suppose that Hypothesis \ref{hyp:domaine_H_0_C} holds. Let $\lambda\in\sigma_\mathrm{ess}(H)$ and suppose that the limits \eqref{eq:LAP_lambda} exist in the norm topology of $\mathcal{L}(\Hi)$. Then for all $v\in \Hi_C$, $\Res_0(\lambda\pm i0^+)v\in\mathcal{D}(H'_0)$ and
			\begin{equation}\label{eq:right_inv}
				(H_0'-\lambda)\Res_0(\lambda\pm i0^+)v = v.
			\end{equation}
		\end{lemma}
	
		\begin{proof}
			Let $v=C\varphi\in\Hi_C$. Since $\Res_0(\lambda\pm i\varepsilon)$ converges to $\Res_0(\lambda\pm i0^+)$ in $\mathcal{L}(\Hi_C,\Hi_C')$, we have, for all $u\in\mathcal{D}(H_0|_{\Hi_C})$, 
			\begin{align*}
				\big \langle (H_0-\lambda)u,\Res_0(\lambda\pm i0^+)v\big\rangle_{\Hi_C;\Hi'_C}&=\lim_{\varepsilon\rightarrow 0^+}\scal{(H_0-\lambda)u}{\Res_0(\lambda\pm i\varepsilon)C\varphi}_\Hi\\
				&=\scal{u}{C\varphi}_{\Hi}\pm\lim_{\varepsilon\rightarrow 0^+}i\varepsilon\scal{u}{\Res_0(\lambda\pm i\varepsilon)C\varphi}_\Hi.
			\end{align*}
		It follows from \eqref{eq:Estimé_norme_C_resolvante_H_0} that 
		\begin{equation*}
			\norme{\Res_0(\lambda\pm i\varepsilon)C}_\Hi=\mathcal{O}\big(\varepsilon^{-\frac12}\big), \quad \varepsilon\rightarrow 0^+.
		\end{equation*}
		Thus we obtain that, for all $u\in\mathcal{D}(H|_{\Hi_C})$,
		\begin{equation*}
			\big \langle (H_0-\lambda)u,\Res_0(\lambda\pm i0^+)v\big\rangle_{\Hi_C;\Hi'_C} = \scal{u}{v}_{\Hi}.
		\end{equation*}
		Since $|\scal{u}{v}_{\Hi}|\le\|u\|_{\Hi_C}\|v\|_{\Hi'_C}$, this shows that $\Res_0(\lambda\pm i0^+)v\in\mathcal{D}(H'_0)$ and that \eqref{eq:right_inv} holds.
		\end{proof}
	
	Now we are ready to prove Theorem \ref{thm:spec_sing}
	
		\begin{proof}[Proof of Theorem \ref{thm:spec_sing}] 
		\ref{it:spec_sing}$\Rightarrow$\ref{it:res_state} Suppose for instance that $\lambda$ is an outgoing spectral singularity. By Proposition \ref{lm:equivalence_point_spectral_regulier}, $\mathrm{Id}+\Res_0(\lambda+i0^+)V'$ is not invertible in $\mathcal{L}(\Hi'_C)$. Since $\Res_0(\lambda+i0^+)V'$  is compact in $\mathcal{L}(\Hi'_C)$, it follows from the Fredholm alternative that there exists $\Psi\in\Hi'_C$, $\Psi\neq0$,  such that 
			\begin{equation}\label{eq:preuve_caracterisation_vecteur_propre_singularité_direct_non_inversible}
				-\Res_0(\lambda+i0^+)V'\Psi=\Psi. 
			\end{equation}
			By Lemma \ref{lm:Inverse_à_droite}, this implies that $\Psi\in\mathcal{D}(H'_0)=\mathcal{D}(H')$ and that
			\begin{equation}\label{eq:preuve_caracterisation_vecteur_propre_singularité_direct_non_inversible2}
				-V'\Psi=(H'_0-\lambda)\Psi. 
			\end{equation}
			Since $H'=H'_0+V'$, this proves \ref{it:res_state}. 
			
			\ref{it:res_state}$\Rightarrow$\ref{it:spec_sing} Suppose now that $\lambda$ is an eigenvalue of $H'$ associated to an eigenvector $\Psi\in\Hi_C'^+(\lambda)$, $\Psi\neq0$. Then 
			\begin{equation*}\label{eq:preuve_thm_caractérisation_sing_et_vp_ecriture_reciproque_ecriture_vp}
				(H'-\lambda)\Psi=0 \quad \text{with}\quad \Psi=-\Res_0(\lambda+i0^+)V'\Psi.
			\end{equation*}
			In particular,
			\begin{equation*}
				(\Id+\Res_0(\lambda+i0^+)V')\Psi=0 ,
			\end{equation*}
			and hence $\Id+\Res_0(\lambda+i0^+)V'$ is not invertible in $\mathcal{L}(\Hi_C')$. By Proposition \ref{lm:equivalence_point_spectral_regulier}, this proves \ref{it:spec_sing}.
		\end{proof}

	\subsection{Spectral singularities of the adjoint operator}\label{sec:spec_sing_adjoint}

	Recall that the regular spectral points and spectral singularities of $H$ have been defined in Definition \ref{def:point_spectral_regulier_classique_pour_H}. The corresponding definition for the adjoint operator $H^*$ is the following.	
		\begin{definition}[Regular spectral point and spectral singularity for $H^*$]\label{def:point_spectral_regulier_classique_pour_H_star}
		Let $\lambda\in\sigma_\mathrm{ess}(H)$. 
	\begin{enumerate}[label=(\roman*)]
		\item We say that $\lambda$ is an outgoing/incoming regular spectral point of $H^*$ if $\lambda$ is not an accumulation point of eigenvalues located in $\lambda\mp i\left( 0,\infty\right)$ and if the limit
		\begin{equation}\label{eq:def_reg_spec_pt_H_Star}
			C\Res_{H^*}(\lambda\mp i0^+)CW^*:=\lim_{\varepsilon\rightarrow 0^+} C\Res_{H^*}(\lambda\mp i\varepsilon)CW^*
		\end{equation}
		exists in the norm topology of $\mathcal{L}(\Hi)$. If $\lambda$ is not an outgoing/incoming regular spectral point, we say that $\lambda$ is an outgoing/incoming spectral singularity of $H^*$.
		\item We say that $\lambda$ is a regular spectral point of $H^*$ if it is both an incoming and an outgoing regular spectral point of $H^*$. If $\lambda$ is not a regular spectral point, we say that $\lambda$ is a spectral singularity of $H^*$.
	\end{enumerate}
	\end{definition}	
	
The following proposition shows that, under our assumptions, $\lambda$ is an outgoing/incoming regular spectral point of $H$ if and only $\lambda$ is an outgoing/incoming regular spectral point of $H^*$. 
	
			\begin{proposition}
			Let $\lambda\in\sigma_\mathrm{ess}(H)$ and suppose that the limits \eqref{eq:LAP_lambda} exist in the norm topology of $\mathcal{L}(\Hi)$. Then the following conditions are equivalents:
			\begin{enumerate}[label=(\roman*)]
				\item\label{it:a1} $\lambda$ is a regular outgoing/incoming spectral point of $H$,
				\item\label{it:a2} $\lambda$ is not an accumulation point of eigenvalues located in $\lambda \pm i(0,\infty)$ and 
				\begin{equation*}
					\lim_{\varepsilon\rightarrow 0^+} WC\Res_H(\lambda\pm i\varepsilon)C
				\end{equation*}
				exists in the norm topology of $\mathcal{L}(\Hi)$,
				\item\label{it:a3} $\lambda$ is a regular outgoing/incoming spectral point of $H^*$.
			\end{enumerate}
		\end{proposition}
	
		\begin{proof}
		Taking adjoints, it is clear that $\ref{it:a2}\Leftrightarrow\ref{it:a3}$. We prove that $\ref{it:a1}\Rightarrow\ref{it:a2}$. Suppose for instance that $\lambda$ is an outgoing regular spectral point of $H$. By Proposition \ref{lm:equivalence_point_spectral_regulier}, $\Id+C\Res_0(\lambda+ i0^+)CW$ is invertible in $\mathcal{L}(\Hi)$. We claim that $\Id+WC\Res_0(\lambda+ i0^+)C$ is invertible in $\mathcal{L}(\Hi)$. Indeed, for all $\varepsilon>0$, a direct computation gives 
			\begin{equation*}
				\Id=\left(\Id-(W(\Id+C\Res_0(\lambda+ i\varepsilon)CW)^{-1}C\Res_0(\lambda+i\varepsilon)C)\right)\left(\Id-WC\Res_0(\lambda+ i\varepsilon)C\right).
			\end{equation*}
			Letting $\varepsilon\to0^+$, we obtain that
			\begin{equation*}
					\Id=\left(\Id-(W(\Id+C\Res_0(\lambda+ i0^+)CW)^{-1}C\Res_0(\lambda+i0^+)C)\right)\left(\Id+WC\Res_0(\lambda+ i0^+)C\right).
				\end{equation*}
				Thus $\Id+WC\Res_0(\lambda+ i0^+)C$ is injective. Since $WC\Res_0(\lambda+ i0^+)C$ is compact, Fredholm's alternative implies that $\Id+WC\Res_0(\lambda+ i0^+)C$ is bijective in $\mathcal{L}(\Hi)$. 
				Now, writing 
				\begin{equation*}
					WC\Res_H(\lambda+i\varepsilon)C=(\Id+WC\Res_0(\lambda+i\varepsilon)C)^{-1}WC\Res_0(\lambda+i\varepsilon)C,
				\end{equation*}
				for $\varepsilon>0$ small enough and next letting $\varepsilon\to0^+$, we deduce that
				$\lim_{\varepsilon\rightarrow 0^+} WC\Res_H(\lambda+ i\varepsilon)C$ exists in $\mathcal{L}(\Hi)$.
				
				The proof of $\ref{it:a1}\Rightarrow\ref{it:a2}$ in the case of an incoming regular spectral point as well as the proof of $\ref{it:a2}\Rightarrow\ref{it:a1}$ are analogous.
		\end{proof}

			\subsection{Embedded eigenvalues}\label{sec:embed}

		In this section, we prove that given our definition of spectral singularities (see Definition \ref{def:point_spectral_regulier_classique_pour_H}), an eigenvalue of $H$ embedded in the essential spectrum is both an incoming and an outgoing spectral singularity.

		\begin{proposition}\label{cor:embed}
			Let $\lambda\in \sigma_\mathrm{ess}(H)$ and suppose that the limits \eqref{eq:LAP_lambda} exist in the norm topology of $\mathcal{L}(\Hi)$. If $\lambda$ is an eigenvalue of $H$, then $\lambda$ is both an outgoing and an incoming spectral singularity of $H$.
		\end{proposition}
	
		\begin{proof}
		Let $\lambda$ be an eigenvalue of $H$. There exists $u\in\mathcal{D}(H)$, $u\neq 0$, such that $(H-\lambda)u=0$.	Suppose by contradiction that $\lambda$ is an outgoing regular spectral point of $H$. Since $\lambda$ is not an accumulation point of eigenvalues located in $\lambda+i(0,\infty)$, we can write, for $\varepsilon>0$ small enough, 
		\begin{equation}\label{eq:b1}
			0=C\Res_0(\lambda+i\varepsilon)(H-\lambda)u=(\Id+C\Res_0(\lambda+i\varepsilon)CW)Cu+i\varepsilon C\Res_0(\lambda+i\varepsilon)u.
		\end{equation}
		Lemma \ref{lm:tech_est} yields
		\begin{equation*}
			\lim_{\varepsilon\rightarrow 0^+}\varepsilon C\Res_0(\lambda+i\varepsilon)u=0. 
		\end{equation*}
		Inserting this into \eqref{eq:b1}, we obtain that
		\begin{equation*}
			(\Id+C\Res_0(\lambda+i0^+)CW)Cu=0,
		\end{equation*}
		which is impossible since $\Id+C\Res_0(\lambda+i0^+)CW$ is injective by Lemma \ref{lm:equivalence_point_spectral_regulier}.
		
		Similarly, $\lambda$ cannot be an incoming regular spectral point of $H$. This concludes the proof of the proposition.
		\end{proof}

It should also be noted that, by Theorem \ref{thm:spec_sing}, the eigenvectors associated to an embedded eigenvalue $\lambda$ belong to $\Hi_C'^\pm(\lambda)$.

\subsection{Local spectral regularity}\label{subsec:local}

In this section, we show that the notion of spectral regularity introduced in Definition \ref{def:point_spectral_regulier_classique_pour_H} is a local property.

We will need to distinguish the case of a spectral singularity embedded in the essential spectrum of $H$ from the case of a `threshold spectral singularity'. Here we will say that a point $\lambda\in\sigma_{\mathrm{ess}}(H)$ is a \emph{spectral threshold} of $H$ if there exists $r>0$ such that either
		\begin{equation}\label{eq:left_thr}
			\sigma_\mathrm{ess}(H)\cap \mathring D(\lambda,r)\subset \lbrack \lambda,\infty) , 
		\end{equation}
		or
		\begin{equation}\label{eq:right_thr}
		 	\sigma_\mathrm{ess}(H)\cap \mathring D(\lambda,r)\subset (-\infty,\lambda].
		\end{equation}
We say that $\lambda$ is a \emph{left threshold} if \eqref{eq:left_thr} holds and a \emph{right threshold} if \eqref{eq:right_thr} holds.

If $\lambda$ belongs to $\sigma_{\mathrm{ess}}(H)$, we will in this section make the assumption that there exists $r>0$ such that the maps
\begin{equation}\label{eq:ext_cont_reg}
\mathring{D}(\lambda,r)\cap\C^\pm\ni z\mapsto C\Res_0(z)C
\end{equation}
extend by continuity to $\mathring{D}(\lambda,r)\cap\bar\C^\pm$. 

The next proposition shows that the notion of spectral regularity is a local property. Note that Item \ref{prop:reg_spec_local_voisi} of Proposition \ref{prop:reg_spec_local} corresponds to the definition of a regular spectral point in \cite{FaFr18_01,FaNi18_01}, in the particular case where $H$ is dissipative. The proof of the next result being a quite straightforward extension of the proof of \cite[Lemma 4.1]{FaNi18_01}, it is deferred to Appendix \ref{App:Propriétés_point_spect_reg}.
\begin{proposition}\label{prop:reg_spec_local}
		Let $\lambda\in\sigma_\mathrm{ess}(H)$. Suppose that there exists $r>0$ such that the maps \eqref{eq:ext_cont_reg} extend by continuity to $\mathring{D}(\lambda,r)\cap\bar\C^\pm$. The following conditions are equivalent:
			\begin{enumerate}[label=(\roman*)]
				\item\label{prop:reg_spec_local_fixe} $\lambda$ is an outgoing/incoming regular spectral point,
				\item\label{prop:reg_spec_local_voisi}  There exists a compact interval $K_\lambda\subset\R$ whose interior contains $\lambda$, such that $K_\lambda$ does not have any accumulation point of eigenvalues of $H$ located in $\C^\pm$, and such that the limit 
				\begin{equation*}
					C\Res_H(\mu\pm i0^+)CW:=\lim_{\varepsilon\rightarrow 0^+}C\Res_H(\mu\pm i\varepsilon)CW
				\end{equation*}
				exists uniformly in $\mu\in K_\lambda$ in the norm topology of $\mathcal{L}(\Hi)$. 
			\end{enumerate}
		\end{proposition}	
\begin{proof}
See Appendix \ref{App:Propriétés_point_spect_reg}.
\end{proof}

Proposition \ref{prop:reg_spec_local} has the following consequence.
\begin{corollary}\label{prop:singularite_null_measure_set}
Suppose that for all $\lambda\in\sigma_\mathrm{ess}(H)$, there exists $r>0$ such that the maps \eqref{eq:ext_cont_reg} extend by continuity to $\mathring{D}(\lambda,r)\cap\bar\C^\pm$. Then the set of spectral singularities of $H$ is a closed set whose Lebesgue measure vanishes.
\end{corollary}

\begin{proof}
Let $E:=E^+\cup E^-$, where
\begin{equation*}
E^\pm:=\lbrace \lambda\in\sigma_{\mathrm{ess}}(H_0), \lambda\text{ is an outgoing/incoming spectral singularity of } H \rbrace.
\end{equation*}
It follows from Proposition \ref{prop:reg_spec_local} that $E^+$ and $E^-$ are closed. Hence $E$ is closed.  Moreover, by the assumption that $z\mapsto C\Res_0(z)CW$ extends by continuity to the real axis, we can apply \cite[Theorem 1.8.3]{Ya92_01}, which implies that $\Id+C\Res_0(\lambda\pm i0^+)CW$ is invertible in $\mathcal{L}(\Hi)$ for a.e. $\lambda\in\sigma_{\mathrm{ess}}(H)$. By Proposition \ref{lm:equivalence_point_spectral_regulier}, this proves that the Lebesgue measures of $E^+$ and $E^-$ vanish.
\end{proof}

Our next concern is to characterize outgoing/incoming regular spectral points $\lambda\in\sigma_{\mathrm{ess}}(H)$ as nontangential limits of the weighted resolvent $C\mathcal{R}_H(z)CW$, as $z\to\lambda$. If $\lambda$ is a left spectral threshold, we will assume that there exists $r>0$ such that the map
\begin{equation}\label{eq:ext_cont_reg_left}
\{\lambda+\nu , \, |\nu|<r , \, 0<\mathrm{arg}(\nu)<2\pi\} = \mathring{D}(\lambda,r)\setminus[\lambda,\infty)\ni z\mapsto C\Res_0(z)C
\end{equation}
extends by continuity to $\{\lambda+\nu , \, |\nu|<r , \, 0\le\mathrm{arg}(\nu)\le2\pi\}$. Il $\lambda$ is a right spectral threshold, we will assume that there exists $r>0$ such that the map
\begin{equation}\label{eq:ext_cont_reg_right}
\{\lambda+\nu \, , \, |\nu|<r , \, -\pi<\mathrm{arg}(\nu)<\pi\} = \mathring{D}(\lambda,r)\setminus(-\infty,\lambda] \ni z\mapsto C\Res_0(z)C
\end{equation}
extends by continuity to $\{\lambda+\nu \, , \, |\nu|<r , \, -\pi\le\mathrm{arg}(\nu)\le\pi\}$. The next proposition proves, in particular, that under these assumptions, outgoing and incoming spectral singularities at thresholds coincide.

			\begin{proposition}\label{prop:equivalence_point_spec_regulier_bord_Lindelof_classique} 
			$ $
			\begin{enumerate}
			\item\label{it:non-thres} Let $\lambda$ be in the interior of $\sigma_{\mathrm{ess}}(H)$. Suppose that there exists $r>0$ such that the maps \eqref{eq:ext_cont_reg} extend by continuity to $\mathring{D}(\lambda,r)\cap\bar\C^\pm$. The following conditions are equivalent:
			\begin{enumerate}[label=(\roman*)]
				\item $\lambda$ is an outgoing/incoming regular spectral point,
				\item There exist a complex neighborhood $\mathcal{O}_\lambda$ of $\lambda$ such that $\mathcal{O}_\lambda^\pm:=\mathcal{O}_\lambda\cap\C^\pm\subset\rho(H)$ and a continuous map $\gamma:( 0,1\rbrack\rightarrow\mathcal{O}_\lambda^\pm$ such that 
	\begin{equation*}
		\lim_{\varepsilon\rightarrow 0^+}\gamma(\varepsilon)=\lambda \quad \text{and} \quad 	\lim_{\varepsilon\rightarrow 0^+}C\Res_H(\gamma(\varepsilon))CW
	\end{equation*}
	exists in the norm topology of $\mathcal{L}(\Hi)$.
			\end{enumerate}
			\item\label{it:threshold} Let $\lambda \in \sigma_{\mathrm{ess}}(H)$ be a spectral threshold of $H$ such that \eqref{eq:left_thr} holds. Suppose that there exists $r>0$ such that the map \eqref{eq:ext_cont_reg_left} extends by continuity to $\{\lambda+\nu , \, |\nu|<r , \, 0\le\mathrm{arg}(\nu)\le2\pi\}$. The following conditions are equivalent:
			\begin{enumerate}[label=(\roman*)]
				\item\label{it:reg_spec_point_thres} $\lambda$ is an outgoing regular spectral point of $H$,
				\item\label{it:reg_spec_point_thres2} $\lambda$ is an incoming regular spectral point of $H$,				
				\item\label{it:reg_spec_point_thres3} There exist a complex neighborhood $\mathcal{O}_\lambda$ of $\lambda$ such that $\mathcal{O}_\lambda\setminus[\lambda,\infty)\subset\rho(H)$ and a continuous map $\gamma:( 0,1\rbrack\rightarrow\mathcal{O}_\lambda\setminus[\lambda,\infty)$ such that 
	\begin{equation}\label{eq:limit_gamma}
		\lim_{\varepsilon\rightarrow 0^+}\gamma(\varepsilon)=\lambda \quad \text{and} \quad 	\lim_{\varepsilon\rightarrow 0^+}C\Res_H(\gamma(\varepsilon))CW
	\end{equation}
	exists in the norm topology of $\mathcal{L}(\Hi)$.
			\end{enumerate}
			The same holds if \eqref{eq:right_thr} holds instead of \eqref{eq:left_thr}, assuming that there exists $r>0$ such that the map \eqref{eq:ext_cont_reg_right} extends by continuity to $\{\lambda+\nu , \, |\nu|<r , \, -\pi\le\mathrm{arg}(\nu)\le\pi\}$ and replacing $\mathcal{O}_\lambda\setminus[\lambda,\infty)$ by $\mathcal{O}_\lambda\setminus(-\infty,\lambda]$ in \ref{it:reg_spec_point_thres3}.
			\end{enumerate}
	\end{proposition}

	\begin{proof}
		Consider the most difficult case \eqref{it:threshold}. We prove that $\eqref{it:threshold}\ref{it:reg_spec_point_thres}\Rightarrow\eqref{it:threshold}\ref{it:reg_spec_point_thres3}$.
		 Let $\lambda\in\sigma_{\mathrm{ess}}(H)$ be an outgoing regular spectral point of $H$ and suppose that \eqref{eq:left_thr} holds.
		  By Proposition \ref{lm:equivalence_point_spectral_regulier}, we know that $\Id+C\Res_0(\lambda+ i0^+)CW$ is invertible in $\mathcal{L}(\Hi)$. Since, by assumption, the map \eqref{eq:ext_cont_reg_left} extends by continuity to $\{\lambda+\nu , \, |\nu|<r , \, 0\le\mathrm{arg}(\nu)\le2\pi\}$ for some $r>0$, and since the set of invertible operators in $\mathcal{L}(\Hi)$ is open, this implies that there exists a complex neighborhood $\mathcal{O}_\lambda\subset \mathring{D}(\lambda,r)$ of $\lambda$,
		   such that the map
		  \begin{equation}\label{eq:map_analytic_bounded}
		  	 \mathcal{O}_\lambda\setminus[\lambda,\infty)\ni z\mapsto \big( \Id+C\Res_0(z)CW\big)^{-1}
		  \end{equation}
		 is analytic. The usual Birmann-Schwinger principle yields $\mathcal{O}_\lambda\setminus[\lambda,\infty)\subset\rho(H)$. It then suffices to take $\gamma(\varepsilon)=\lambda+i\delta\varepsilon$, with $\delta>0$ small enough.
		 
		 Next we prove that $\eqref{it:threshold}\ref{it:reg_spec_point_thres3}\Rightarrow\eqref{it:threshold}\ref{it:reg_spec_point_thres}$. In the same way as in Proposition \ref{lm:equivalence_point_spectral_regulier}, the existence of the limit \eqref{eq:limit_gamma} is equivalent to the invertibility of $\mathrm{Id}+C\Res_0(\gamma(0^+))CW$ in $\mathcal{L}(\Hi)$. Note that the limit $C\Res_0(\gamma(0^+))C$ exists in the norm topology of $\mathcal{L}(\Hi)$ since we have assumed that there exists $r>0$ such that the map $z\mapsto C\Res_0(z)C$ extends by continuity to $\{\lambda+\nu , \, |\nu|<r , \, 0\le\mathrm{arg}(\nu)\le2\pi\}$. We can then argue as before; this gives the existence of a complex neighborhood $\mathcal{O}_\lambda$ of $\lambda$ such that the map \eqref{eq:map_analytic_bounded} is analytic and extends by continuity to $\{\lambda+\nu , \, |\nu|<r , \, 0\le\mathrm{arg}(\nu)\le2\pi\}$. In particular, $\Id+C\Res_0(\lambda+ i0^+)CW$ is well-defined and invertible in $\mathcal{L}(\Hi)$. Applying Proposition \ref{lm:equivalence_point_spectral_regulier}, this shows that $\lambda$ is an incoming regular spectral point of $H$.
		 
		 The proof of $\eqref{it:threshold}\ref{it:reg_spec_point_thres2}\Leftrightarrow\eqref{it:threshold}\ref{it:reg_spec_point_thres3}$ is identical. One proceeds analogously to prove \eqref{it:threshold} in the case where \eqref{eq:right_thr} holds instead of \eqref{eq:left_thr}.
		 
		 Finally, the argument easily adapts to prove \eqref{it:non-thres}.
	\end{proof}

Note that in the case of an outgoing/incoming regular spectral point $\lambda\in\sigma_{\mathrm{ess}}(H)$, our proof shows that $z\mapsto C\Res_H(z)CW$ has a nontangential limit at $\lambda$, in the sense that the limit
\begin{equation}\label{eq:limit_indep_gamma}
\lim_{\varepsilon\to0^+}C\Res_H(\gamma(\varepsilon))CW
\end{equation}
does not depend on the continuous curve $\gamma:(0,1]\to\C^\pm$ such that $\gamma(\varepsilon)\to\lambda$ as $\varepsilon\to0^+$. See Figure \ref{fig4}. This property can also be deduced from Lindelöf's Theorem (see e.g. \cite{Rud1980}) if the map $z\mapsto C\Res_H(z)CW$ is known to be analytic and bounded in the domain $\mathcal{O}_\lambda^\pm$.

\begin{figure}[ht]
			\begin{center}
				\begin{tikzpicture}
					\draw[<->] (-2,0) --(2,0);
					
					\draw (0,0) circle(1);
					\draw[fill=gray!30](1,0) arc(0:180:1);
					\draw (0,0) node {$\bullet$};
					\draw (0,0) node[above] {$\lambda$};
					\draw[-,very thick] (-1.2,0) -- (1.2,0);
					\draw(0.3,-0.9)node[above]{$\mathcal{O}_\lambda$};	
					\draw[very thin] plot[domain=0:-0.5]	(\x,{-\x^1.6}) node[below]{$\gamma$};
				\end{tikzpicture}
			\end{center}
			\caption{ \footnotesize  \textbf{Outgoing spectral singularity located inside the essential spectrum.} The figure shows an example of a curve $\gamma:(0,1]\to\mathcal{O}^+_\lambda$, where $\mathcal{O}^+_\lambda=\mathcal{O}_\lambda\cap \C^+$ and $\mathcal{O}_\lambda$ is a complex neighborhood of $\lambda$. The thick line represents the essential spectrum of $H$.}\label{fig4}
		\end{figure}
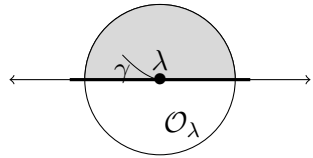

Likewise, in the case of a regular spectral point $\lambda$ located at a spectral threshold of $H$, the limit \eqref{eq:limit_indep_gamma} does not depend on the continuous curve $\gamma:(0,1]\to\C\setminus [\lambda,\infty)$ (or $\gamma:(0,1]\to\C\setminus (-\infty,\lambda]$, depending on whether \eqref{eq:left_thr} or \eqref{eq:right_thr} holds), which is the reason why the incoming and outgoing spectral singularities coincide. See Figure \ref{fig2}.

		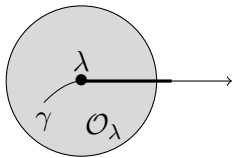
\begin{figure}[ht]
			\begin{center}
				\begin{tikzpicture}
					\draw[->] (0,0) --(2,0);
					
					\draw[fill=gray!30](0,0) circle (1);
					\draw (0,0) node {$\bullet$};
					\draw (0,0) node[above] {$\lambda$};
					\draw[-,very thick] (0,0) -- (1.2,0);
					\draw(0.3,-0.9)node[above]{$\mathcal{O}_\lambda$};	
					\draw[very thin] plot[domain=0:-0.5]	(\x,{\x^1.8}) node[below]{$\gamma$};
				\end{tikzpicture}
			\end{center}
			\caption{ \footnotesize  \textbf{Spectral singularity located at a threshold.} The figure shows an example of a curve $\gamma:(0,1]\to\mathcal{O}_\lambda\setminus[\lambda,\infty)$, where $\mathcal{O}_\lambda$ is a complex neighborhood of $\lambda$. The thick line represents the essential spectrum of $H$.}\label{fig2}
		\end{figure}

\subsection{Spectral singularities for dissipative operators}\label{sec:dissipative case} 
		
		In this section, we focus on the particular case where the operator $H$ is dissipative. Recall that $H=H_0+V$ with $V\in\mathcal{L}(\Hi)$. We write
		\begin{equation*}
			V = V_1 - i V_2, \quad V_1=CW_1C, \quad V_2=CW_2C,
		\end{equation*}
		where $W_1:=\re(W)$ and $W_2:=-\im(W)$. Here the real and imaginary parts of a bounded operator $A\in\mathcal{L}(H)$ are defined as usual by $\re(A):=\frac{1}{2}(A+A^*)$, $\im(A):=\frac{1}{2i}(A-A^*)$. We suppose that
		\begin{equation*}
			W_2\ge0,
		\end{equation*}
		so that 
		\begin{equation*}
			H=H_0+V_1-iV_2=: H_{V_1}-iV_2
		\end{equation*}
		is indeed dissipative. Our purpose is to prove that if $\lambda$ is an outgoing regular spectral point of the self-adjoint part $H_{V_1}=H_0+V_1$, then $\lambda$ is also an outgoing regular spectral point for $H$. In other words, adding the `dissipative part' $-iV_2$ cannot create outgoing spectral singularities.

	We begin with recalling the following easy lemma (see \cite[Lemma 6.1]{Da80_01} or \cite[Lemma 3.1]{FaFr18_01}). We include a proof for the convenience of the reader.
		
		\begin{lemma}\label{lm:eigen-diss}
			Let $\lambda\in\R$ be a real eigenvalue of $H$. Then $\lambda$ is an eigenvalue of $H_{V_1}$ and 
			\begin{equation}\label{eq:lien_espace_propre_H_et_H_V1}
				\Ker(H-\lambda)\subset\Hi_\mathrm{disc}(H_{V_1})\cap\Ker(V_2).
			\end{equation}
		\end{lemma}
		
		\begin{proof}
			Let $u\in\Ker(H-\lambda)$, $u\neq0$. Then 
			\begin{equation*}
				\lambda\norme{u}_\Hi^2=\scal{u}{Hu}_\Hi=\scal{u}{H_{V_1}u}_\Hi-i\Big \|W_2^\frac{1}{2}Cu\Big\|_\Hi^2.
			\end{equation*}
			Since $\lambda\in\R$, identifying the real and imaginary parts, we obtain that $u\in\Ker(W_2^\frac{1}{2}C)\subset\Ker(V_2)$ and therefore $Hu=H_{V_1}u=\lambda u$. This establishes \eqref{eq:lien_espace_propre_H_et_H_V1}.
		\end{proof}
		
		Note that Definition \ref{def:point_spectral_regulier_classique_pour_H} of a regular spectral point of $H$ applies to $H_{V_1}$ as well. In other words, $\lambda\in\sigma_{\mathrm{ess}}(H_{V_1})$ is an outgoing/incoming regular spectral point of $H_{V_1}$ if 
			\begin{equation*}
				C\Res_{V_1}(\lambda\pm i0^+)CW_1:=\lim_{\varepsilon\rightarrow 0^+}C\Res_{V_1}(\lambda\pm i\varepsilon)CW_1
			\end{equation*}
			exits in $\mathcal{L}(\Hi)$. Otherwise, $\lambda$ is an outgoing/incoming spectral singularity of $H_{V_1}$.  Moreover, Proposition \ref{lm:equivalence_point_spectral_regulier} applied to $H_{V_1}$ shows that $\lambda$ is an outgoing/incoming regular spectral point of $H_{V_1}$ if and only if $\Id+C\Res_0(\lambda\pm i0^+)CV_1$ is invertible in $\mathcal{L}(\Hi)$, which is also equivalent to $\Id+\Res_0(\lambda\pm i0^+)V_1'$ being invertible in $\mathcal{L}(\Hi'_C)$.

		The next proposition is the main result of this subsection.

		\begin{proposition}\label{prop:no_spectr_sing_diss}
			Suppose that Hypothesis \ref{hyp:domaine_H_0_C} holds. Let $\lambda\in\sigma_\mathrm{ess}(H)$ and suppose that the limits \eqref{eq:LAP_lambda} exist in the norm topology of $\mathcal{L}(\Hi)$. If $\lambda$ is an outgoing regular spectral point of $H_{V_1}$ then $\lambda$ is an outgoing regular spectral point of $H$.
		\end{proposition}
	
		\begin{proof}
			Suppose that $\lambda$ is an outgoing regular spectral point of $H_{V_1}$. Suppose also, by contradiction, that $\lambda$ is an outgoing spectral singularity of $H$. By Theorem \ref{thm:spec_sing}, there exists $\Psi\in\Hi_C^+$ such that 
			\begin{equation}\label{eq:preuve_dissipatif_pas_de_singularités_entrantes}
				\Psi=-\Res_0(\lambda+i0^+)V'\Psi \quad \text{and} \quad (H'-\lambda)\Psi=0. 
			\end{equation}
			Recall from Section \ref{subsec:extension} that $V'\in\mathcal{L}(\Hi_C',\Hi_C)$ is given by $V'=CWC'$. Since $W=W_1-iW_2$, this yields
			\begin{equation*}
				\langle V'\Psi,\Psi\rangle_{\Hi_C,\Hi'_C}=\scal{W_1 C'\Psi}{C'\Psi}_\Hi+i\scal{W_2 C'\Psi}{C'\Psi}_\Hi,
			\end{equation*}
			and hence, since $W_2\ge0$,
			\begin{equation}\label{eq:d1}
				\im\big(\langle V'\Psi,\Psi\rangle_{\Hi_C,\Hi'_C}\big)=\scal{W_2C'\Psi}{C'\Psi}_\Hi\geq 0. 
			\end{equation}
			Now, using \eqref{eq:preuve_dissipatif_pas_de_singularités_entrantes}, we have
			\begin{align}
				\im\big(\langle V'\Psi,\Psi\rangle_{\Hi_C,\Hi'_C}\big)&=-\im\scal{V'\Psi}{\Res_0(\lambda+i0^+)V'\Psi}_{\Hi_C,\Hi'_C} \notag\\
				&=\frac{-1}{2i}\left(\scal{W C'\Psi}{C\Res_0(\lambda+i0^+)V'\Psi}_\Hi-\scal{C\Res_0(\lambda+i0^+)V'\Psi}{W C'\Psi}_\Hi\right)\notag\\
				&=\frac{-1}{2i}\left(\scal{W C'\Psi}{C\Res_0(\lambda+i0^+)V'\Psi}_\Hi-\scal{W C'\Psi}{C\Res_0(\lambda-i0^+)V'\Psi}_\Hi\right)\notag\\
				&=\lim_{\varepsilon\rightarrow 0^+}-\varepsilon\scal{W C'\Psi}{C\Res_0(\lambda+i\varepsilon)\Res_0(\lambda-i\varepsilon)V'\Psi}_\Hi\notag\\
				&=\lim_{\varepsilon\rightarrow 0^+}-\varepsilon\scal{\Res_0(\lambda-i\varepsilon)V'\Psi}{\Res_0(\lambda-i\varepsilon)V'\Psi}_\Hi\leq 0. \label{eq:d2}
			\end{align}
			Equations \eqref{eq:d1} and \eqref{eq:d2} imply that $W_2C'\Psi=0$. 
			Inserting this into \eqref{eq:preuve_dissipatif_pas_de_singularités_entrantes}, we obtain that 
			\begin{equation*}
				\Psi=-\Res_0(\lambda+i0^+)V_1'\Psi \quad\text{and}\quad (H_0'+V_1'-\lambda)\Psi=0.
			\end{equation*}
			which is a contradiction since $\lambda$ is a regular spectral point of $H_{V_1}$.
		\end{proof}

		\begin{remark}
			The previous proof actually shows that, under the conditions of Proposition \ref{prop:no_spectr_sing_diss},
			\begin{equation}\label{eq:W2_spec_sing}
				\lambda\text{ is an outgoing spectral singularity of }H \, \Rightarrow \, W_2C'\Psi=0.
			\end{equation}
			Since $C'$ is injective, this in turn yields
			\begin{equation*}
				W_2\text{ is injective} \, \Rightarrow \, H\text{ has no outgoing spectral singularities}.
			\end{equation*}
		\end{remark}
		
		In the context of Schrödinger operators, one can also combine \eqref{eq:W2_spec_sing} with the unique continuation principle (see e.g. \cite[Theorem XIII.63]{ReSi80_01}) to obtain Proposition \ref{prop:schr_diss}.
				
		\begin{proof}[Proof of Proposition \ref{prop:schr_diss}]
		Let $\lambda>0$. Assuming that \eqref{eq:short-range} holds, suppose by contradiction that $\lambda$ is an outgoing spectral singularity of $H$. Then, by Proposition \ref{prop:3.4}, there exists $\Psi\in L^2_{-\sigma/2}$, $\Psi\neq0$, such that
								\begin{equation*}
							(-\Delta+V(x)-\lambda)\Psi=0.
						\end{equation*} 
						This implies that $\Psi\in H^2_{\mathrm{loc}}(\mathbb{R}^3)$ and that $|\Delta\Psi(x)|\le (\|V\|_{L^\infty}+\lambda)|\Psi(x)|$ for a.e. $x\in\mathbb{R}^3$. Now, since there exists a non-trivial open set $U$ such that $\mathrm{Im}(V)<0$ on $U$, \eqref{eq:W2_spec_sing} and the fact that $C(x)=\langle x\rangle^{-\sigma/2}$ imply that $\Psi(x)=0$ on $U$. By \cite[Theorem XIII.63]{ReSi80_01}, we conclude that $\Psi=0$. This is a contradiction.
						
						The same argument holds for $\lambda=0$ under the condition that \eqref{eq:short-range-zero} holds.
		\end{proof}
		
	\section{Spectral resolution formula and spectral subspaces}\label{sec:spectr_res}

In this section, we begin with constructing a functional calculus for $H$. In Section \ref{subsec:spectr_proj}, we consider the simplest case of intervals without spectral singularities, next, in Section \ref{subsec:regul_funct_calc}, we construct a regularized functional calculus in intervals possibly containing spectral singularities. The latter is subsequently used in Section \ref{subsec:spectr_res} to establish the spectral resolution formula stated in Proposition \ref{prop:resolution}. Finally we prove Theorems \ref{thm:caracterisation_etat_disparaissent_à_infini} and \ref{thm:caracterisation_espace_absolument_continu} in Sections \ref{subsec:ads} and \ref{subsec:abs_cont}, respectively.

		
	\subsection{Functional calculus in intervals not containing spectral singularities}\label{subsec:spectr_proj}
			Consider first a closed interval $I\subset\mathbb{R}$ that does not contain any spectral singularities of $H$. We will furthermore assume that a limiting absorption principle holds for $H$ in $I$, in the sense that there exists $\varepsilon_0>0$ such that
		\begin{equation}\label{eq:limit_uniform_I}
			\sup_{\mathrm{Re}(z)\in I, \pm\mathrm{Im}(z)\in(0,\varepsilon_0)}  \big\|C\Res_H(z)CW\big\|_{\mathcal{L}(\Hi)}<\infty.
		\end{equation}
			Note that, in the same way as for $H_0$, Fatou's Theorem and \eqref{eq:limit_uniform_I} yield that the limits $C\Res_H(\lambda\pm i0^+)CW$ exist for almost every $\lambda\in I$, in the norm topology of $\mathcal{L}(\Hi)$, and that the map $I \ni\lambda\mapsto C\Res_H(\lambda\pm i0^+)CW\in\mathcal{L}(\Hi)$ is bounded.	The main purpose of this subsection is then to define a spectral projection for $H$ in $I$ by mimicking Stone's formula, setting
			\begin{equation}\label{eq:Formule de Stone}
				\mathds{1}_I(H):=\wlim_{\varepsilon\rightarrow 0^+}\frac{1}{2\pi i}\int_I\big(\Res_H(\lambda+i\varepsilon)-\Res_H(\lambda-i\varepsilon)\big)\mathrm{d}\lambda. 
			\end{equation}
			The next proposition shows that, under our assumptions, $\mathds{1}_I(H)$ is a well-defined non-orthogonal projection. 	The proof is similar to that given in \cite{FaFr18_01} for dissipative operators.  For the convenience of the reader, a sketch of the proof of Proposition \ref{prop:Projection_Spectrale_Stone} focusing on the differences with \cite{FaFr18_01} is reported in Appendix \ref{App:Operateurs_spectraux}.
			
			\begin{proposition}\label{prop:Projection_Spectrale_Stone}
				Suppose that Hypothesis \ref{hyp:principe_absorption_limite} holds.  Let $I\subset\mathbb{R}$ be a closed interval and suppose that there exists $\varepsilon_0>0$ such that \eqref{eq:limit_uniform_I} holds. Then the weak limit  \eqref{eq:Formule de Stone} exists in $\mathcal{L}(\Hi)$ and we have
				\begin{equation}\label{eq:Projection_spectral_propriété_de_projection}
					\mathds{1}_{I_1}(H) \mathds{1}_{I_2}(H)=\mathds{1}_{I_1\cap I_2}(H),
				\end{equation}
				for any closed intervals $I_1,I_2\subset I$ without spectral singularity, with the convention that $\mathds{1}_\emptyset(H)=0$. In particular, $\mathds{1}_I(H)$ is a projection. Its adjoint is given by 
				\begin{equation}\label{eq:Projection_spectral_propriété_d'adjonction}
					\mathds{1}_I(H)^*=\mathds{1}_I(H^*)=\wlim_{\varepsilon\rightarrow 0^+}\frac{1}{2\pi i}\int_I\big(\Res_{H^*}(\lambda+i\varepsilon)-\Res_{H^*}(\lambda-i\varepsilon)\big)\mathrm{d}\lambda. 
				\end{equation}
			\end{proposition}
			\begin{proof}
				See Appendix \ref{App:Operateurs_spectraux}.
			\end{proof}
			
We note the following representation formula which follows from our proof:
			\begin{align}
 \mathds{1}_I(H) &= \mathds{1}_I(H_0)  - \frac{1}{2i\pi} \int_I  \Res_{0}(\lambda\mp i0^+)CWC\Res_{0}(\lambda\pm i0^+)  \mathrm{d}\lambda \notag \\
				&\quad + \frac{1}{2i\pi} \int_I   \Res_0(\lambda\pm i0^+) CWC\Res_H(\lambda\pm i0^+)CWC\Res_0(\lambda\pm i0^+) \mathrm{d}\lambda, \label{eq:repres_1_I}
			\end{align}
in the sense of quadratic forms on $\Hi\times\Hi$. We recall that, for all $u\in\Hi$, $\lambda\mapsto C\Res_0(\lambda\pm i0^+)u$ denotes the limit of $\lambda\mapsto C\Res_0(\lambda\pm i\varepsilon)u$ in $L^2(\R;\Hi)$ as $\varepsilon\to0^+$, while, for a.e. $\lambda$, $C\Res_H(\lambda\pm i0^+)CW$ is the limit of $C\Res_H(\lambda\pm i\varepsilon)CW$ as $\varepsilon\to0^+$, in the norm topology  of $\mathcal{L}(\Hi)$.

Under the same assumptions and using similar arguments, we also have the following functional calculus. We denote by $\mathrm{C}_{\mathrm{b}}(I)$ the set of bounded continuous functions on $I$.
	\begin{proposition}\label{prop:functional_calculus}
				Under the conditions of Proposition \ref{prop:Projection_Spectrale_Stone}, the map
				\begin{align}\label{eq:algebra_morphism}
					\mathrm{C}_{\mathrm{b}}(I)\ni f&\mapsto f(H):= \wlim_{\varepsilon\rightarrow 0^+}\frac{1}{2\pi i}\int_If(\lambda)\big(\Res_{H}(\lambda+i\varepsilon)-\Res_{H}(\lambda-i\varepsilon)\big)\mathrm{d}\lambda\in\mathcal{L}(\Hi)
				\end{align}
				is a Banach algebra morphism. Moreover, for all $t\in\R$, 
				\begin{equation}\label{eq:Projection_spectral_calcul_fonctionnel_semi-group}
					e^{itH}\mathds{1}_I(H)=\wlim_{\varepsilon\rightarrow 0^+}\frac{1}{2\pi i}\int_I e^{it\lambda}\big(\Res_H(\lambda+i\varepsilon)-\Res_H(\lambda-i\varepsilon)\big)\mathrm{d}\lambda
				\end{equation}
				and for all $z_0\in\rho(H)$,
				\begin{equation}\label{eq:Projection_spectral_calcul_fonctionnel_resolvent}
					\Res_H(z_0)\mathds{1}_I(H)=\wlim_{\varepsilon\rightarrow 0^+}\frac{1}{2\pi i}\int_I (\lambda-z_0)^{-1}\big(\Res_H(\lambda+i\varepsilon)-\Res_H(\lambda-i\varepsilon)\big)\mathrm{d}\lambda.
				\end{equation}
			\end{proposition}
			\begin{proof}
				See Appendix \ref{App:Operateurs_spectraux}.
			\end{proof}

We mention that this functional calculus uniquely extends to a Borel functional calculus, i.e. an algebra 	morphism $L^\infty(I)\ni f\to f(H)$. See e.g. \cite[Theorem 2.4]{GeGeHa13_01}.

\subsection{`Regularized' functional calculus}\label{subsec:regul_funct_calc}

			Our next concern is to regularize the definition \eqref{eq:Formule de Stone} in the case where the spectral interval $I$ contains spectral singularities. More generally, we will now consider a closed interval $I\subset\mathbb{R}$ and a bounded holomorphic function 
			\begin{equation*}
				h:U\to\C, \quad \{z\in\C,\, \mathrm{Re}(z)\in I , |\mathrm{Im}(z)|\le\varepsilon_0\}\subset U,
			\end{equation*}
			with $\varepsilon_0>0$ and $U$ open, such that
					\begin{equation}\label{eq:limit_uniform_I_sing}
			\sup_{\mathrm{Re}(z)\in I, \pm\mathrm{Im}(z)\in(0,\varepsilon_0)} |h(z)| \big\|C\Res_H(z)CW\big\|_{\mathcal{L}(\Hi)}<\infty.
		\end{equation}
We will also assume that
\begin{equation}\label{eq:h_assumpt}
\lambda\mapsto\sup_{0<\varepsilon<\varepsilon_0}|h'(\lambda\pm i\varepsilon)| \in L^2(I),
\end{equation}
where $h'$ stands for the derivative of $h$. The `regularized spectral projection' for $H$ in $I$ is then defined by
	\begin{equation}\label{eq:Formule de Stone_regul}
		\big(h \mathds{1}_I\big)(H):=\wlim_{\varepsilon\rightarrow 0^+}\frac{1}{2\pi i}\int_I\big(h(\lambda+i\varepsilon)\Res_H(\lambda+i\varepsilon)-h(\lambda-i\varepsilon)\Res_H(\lambda-i\varepsilon)\big)\mathrm{d}\lambda. 
	\end{equation}
In this context, Proposition \ref{prop:Projection_Spectrale_Stone} should be modified as follows.

			\begin{proposition}\label{prop:Projection_Spectrale_Stone_regul}
				Suppose that Hypothesis \ref{hyp:principe_absorption_limite} holds.  Let $I\subset\mathbb{R}$ be a closed interval. Suppose that there exist $\varepsilon_0>0$ and a bounded holomorphic function $h$ defined on a complex neighborhood of $\{z\in\C,\, \mathrm{Re}(z)\in I , |\mathrm{Im}(z)|\le\varepsilon_0\}$ such that \eqref{eq:limit_uniform_I_sing} and \eqref{eq:h_assumpt} hold. Then the weak limit  \eqref{eq:Formule de Stone_regul} exists in $\mathcal{L}(\Hi)$ and its adjoint is given by
				\begin{equation*}
					\big(h\mathds{1}_I\big)(H)^*=\big(\overline{h}\mathds{1}_I\big)(H^*)=\wlim_{\varepsilon\rightarrow 0^+}\frac{1}{2\pi i}\int_I\big(\overline{h(\lambda+i\varepsilon)}\Res_{H^*}(\lambda+i\varepsilon)-\overline{h(\lambda-i\varepsilon)}\Res_{H^*}(\lambda-i\varepsilon)\big)\mathrm{d}\lambda. 
				\end{equation*}
			\end{proposition}
			\begin{proof}
				See Appendix \ref{App:Operateurs_spectraux}.
			\end{proof}

It is not difficult to verify that, if $H$ is self-adjoint, then $(h \mathds{1}_I)(H)=h(H)\mathds{1}_I(H)$ (see the proof of Proposition \ref{prop:Projection_Spectrale_Stone_regul}). However, under the assumptions of Proposition \ref{prop:Projection_Spectrale_Stone_regul}, this formula does \textit{not} make sense since $\mathds{1}_I(H)$ is ill-defined in general.

Similarly as in \eqref{eq:repres_1_I}, our proof will show that
			\begin{align}
(h \mathds{1}_I)(H) &= (h\mathds{1}_I)(H_0)  - \frac{1}{2i\pi} \int_I h(\lambda) \Res_{0}(\lambda\mp i0^+)CWC\Res_{0}(\lambda\pm i0^+)  \mathrm{d}\lambda \notag \\
				&\quad + \frac{1}{2i\pi} \int_I h(\lambda)  \Res_0(\lambda\pm i0^+) CWC\Res_H(\lambda\pm i0^+)CWC\Res_0(\lambda\pm i0^+) \mathrm{d}\lambda, \label{eq:repres_1_I_h}
			\end{align}
in the sense of quadratic forms on $\Hi\times\Hi$.

One can also define a `regularized functional calculus' on the set of functions 
\begin{equation*}
	\mathrm{C}_{\mathrm{b},\mathrm{reg}}(I):=\big\{ f:I\to\C , \, \exists g\in \mathrm{C}_{\mathrm{b}}(I) , \, f=hg \}.
\end{equation*}

	\begin{proposition}\label{prop:functional_calculus_reg}
				Under the conditions of Proposition \ref{prop:Projection_Spectrale_Stone_regul}, the map
				\begin{align*}
					\mathrm{C}_{\mathrm{b},\mathrm{reg}}(I)\ni f&\mapsto f(H):= \wlim_{\varepsilon\rightarrow 0^+}\frac{1}{2\pi i}\int_Ig(\lambda)\big(h(\lambda+i\varepsilon)\Res_{H}(\lambda+i\varepsilon)-h(\lambda-i\varepsilon)\Res_{H}(\lambda-i\varepsilon)\big)\mathrm{d}\lambda\in\mathcal{L}(\Hi)
				\end{align*}
				is an algebra morphism and there exists $\mathrm{c}>0$ such that
				\begin{equation}\label{eq:algebra_morph_regul}
					\|f(H)\|_{\mathcal{L}(\Hi)} \le \mathrm{c} \|g\|_{L^\infty},
				\end{equation}
				for all $f \in \mathrm{C}_{\mathrm{b},\mathrm{reg}}(I)$, with $f=hg$.
				Moreover, for all $t\in\R$, 
				\begin{equation}\label{eq:Projection_spectral_calcul_fonctionnel_semi-group_regul}
					e^{itH}\big(h\mathds{1}_I\big)(H)=\wlim_{\varepsilon\rightarrow 0^+}\frac{1}{2\pi i}\int_I e^{it\lambda}\big(h(\lambda+i\varepsilon)\Res_H(\lambda+i\varepsilon)-h(\lambda-i\varepsilon)\Res_H(\lambda-i\varepsilon)\big)\mathrm{d}\lambda
				\end{equation}
				and for all $z_0\in\rho(H)$,
				\begin{equation}\label{eq:Projection_spectral_calcul_fonctionnel_resolvent_regul}
					\Res_H(z_0)\big(h\mathds{1}_I\big)(H)=\wlim_{\varepsilon\rightarrow 0^+}\frac{1}{2\pi i}\int_I (\lambda-z_0)^{-1} \big(h(\lambda+i\varepsilon)\Res_H(\lambda+i\varepsilon)-h(\lambda-i\varepsilon)\Res_H(\lambda-i\varepsilon)\big)\mathrm{d}\lambda.
				\end{equation}
			\end{proposition}
			\begin{proof}
				See Appendix \ref{App:Operateurs_spectraux}.
			\end{proof}
			
Again, this functional calculus uniquely extends to a Borel functional calculus. For other definitions of functional calculi for general operators on Banach spaces under an assumption of polynomial growth of the resolvent near the real axis, we refer to \cite{Da95_01,GeGeHa13_01}.
			
\subsection{Spectral resolution formula}\label{subsec:spectr_res}
We now turn to the proof of the resolution formula stated in Proposition \ref{prop:resolution}. It relies in particular on the following resolvent bounds.
\begin{lemma}\label{lm:res_bound}
Suppose that Hypothesis \ref{hyp:principe_absorption_limite} holds.  Let $I\subset\mathbb{R}$ be a closed interval and suppose that there exists $\varepsilon_0>0$ such that \eqref{eq:limit_uniform_I} holds. 
\begin{enumerate}[label=(\roman*)]
\item\label{it:res_bound1} There exists $\mathrm{c}>0$ such that, for a.e. $\lambda\in I$, for all $\varepsilon\in(0,\varepsilon_0)$,
\begin{equation*}
\|R_H(\lambda\pm i\varepsilon)\|_{\mathcal{L}(\Hi)}\le\mathrm{c}\varepsilon^{-1}.
\end{equation*}
\item\label{it:res_bound2} There exists $\mathrm{c}>0$ such that, for all $\varepsilon\in(0,\varepsilon_0)$, for all $u\in\Hi$,
\begin{equation*}
\int_I \|R_H(\lambda\pm i\varepsilon) u\|_\Hi^2\mathrm{d}\lambda \le \mathrm{c}\varepsilon^{-1} \|u\|^2_\Hi.
\end{equation*}
\end{enumerate}
\end{lemma}

\begin{proof}
See Appendix \ref{App:Operateurs_spectraux}.
\end{proof}

To prove Proposition \ref{prop:resolution}, we will rely on the following construction. Assume that Hypotheses \ref{hyp:valeurs_propres_H} and \ref{hyp:singularités_spectrales} hold. Let $\sigma_{\mathrm{disc},\mathrm{real}}(H)=\{e_1,\dots,e_p\}$ be the set of real, discrete eigenvalues of $H$, with $e_1<\cdots<e_p$. Note that $\sigma_{\mathrm{disc},\mathrm{real}}(H)$ is finite by Hypothesis \ref{hyp:valeurs_propres_H}.  Let $\delta>0$ be the distance between $\sigma_{\mathrm{disc},\mathrm{real}}(H)$ and $\sigma_{\mathrm{ess}}(H)$. For $\varepsilon>0$ small enough, we consider the complex open set $U_\varepsilon$ such that $\sigma(H)\subset U_\varepsilon$ and the boundary of $U_\varepsilon$ is given by (i) small circles surrounding each discrete eigenvalue of $H$ and no other point of $\sigma(H)$, (ii) rectangles whose opposite sides are given by the complex segments $[ e_\ell + \delta/2 + i\varepsilon , e_{\ell+1}-\delta/2 + i\varepsilon ]$ and $[ e_\ell + \delta/2 - i\varepsilon , e_{\ell+1}-\delta/2 - i\varepsilon ]$ and (iii) the curve given by the complex segments $[e_p+\delta/2\pm i\varepsilon,\varepsilon^{-3}\pm i\varepsilon]$, $[e_p+\delta/2-i\varepsilon,e_p+\delta/2+i\varepsilon]$ and the (long) circle arc centered at the origin and joining the complex points $\varepsilon^{-3}+ i\varepsilon$ and $\varepsilon^{-3}- i\varepsilon$. The circles and rectangles defined by (i), (ii) are oriented counterclockwise, while the curve defined by (iii) is oriented clockwise. See Figure \ref{fig:contour}. We denote by $\Gamma_{\mathrm{(i)}}$ the union of curves defined by (i) and by $\Gamma_{\varepsilon,\sharp}$ the unions of curves defined by $\sharp$, where $\sharp$ stands for (ii) or (iii). If $\sigma_{\mathrm{disc},\mathrm{real}}(H)$ is empty, then $\Gamma_{\varepsilon,\mathrm{(ii)}}$ is absent and we replace $e_p-\delta>2$ by $-1$ in the definition of $\Gamma_{\varepsilon,\mathrm{(iii)}}$ (fixing arbitrarily $-1$ as a real number such that $-1<\inf\sigma_{\mathrm{ess}}(H)$).

\begin{figure}[H] 
\begin{center}
\begin{tikzpicture}[scale=0.6, every node/.style={scale=0.9}]

   \draw[->](-3,2) -- (14,2);      
      \draw[->](4,-4) -- (4,8);      
  \draw[-, very thick] (4,2)--(5,2);
    \draw[-, very thick] (7,2)--(13.95,2);
  
  
              

                  
         \draw (4,2) +(3:5.5) arc (3:357:5.5);
         
            \draw[-](9.5,2.3) -- (6.6,2.3);    
                       \draw[->](9.5,2.3) -- (8,2.3);    
            \draw[-](6.6,2.3) -- (6.6,1.7);
            \draw[-](6.6,1.7) -- (9.5,1.7);                  

            \draw[-](3.6,2.3) -- (5.4,2.3);      
            \draw[->](5.4,2.3) -- (4.45,2.3);                  
            \draw[-](3.6,1.7) -- (5.4,1.7);
            \draw[-](3.6,2.3) -- (3.6,1.7);                  
            \draw[-](5.4,2.3) -- (5.4,1.7);

         			\draw(6,2)  node {\tiny{$\times$}};

         \draw (6,2) +(4:0.4) arc (4:364:0.4);
         \draw[->] (6,2) +(90:0.4) arc (90:100:0.4);
         
         			\draw(8,4)  node {\tiny{$\times$}};

         \draw (8,4) +(4:0.4) arc (4:364:0.4);
         \draw[->] (8,4) +(90:0.4) arc (90:100:0.4);

         			\draw(0,2)  node {\tiny{$\times$}};

         \draw (0,2) +(4:0.4) arc (4:364:0.4);
         \draw[->] (0,2) +(90:0.4) arc (90:100:0.4);
         
         			\draw(1,3)  node {\tiny{$\times$}};

         \draw (1,3) +(4:0.4) arc (4:364:0.4);
         \draw[->] (1,3) +(90:0.4) arc (90:100:0.4);
         
         			\draw(5,0)  node {\tiny{$\times$}};

         \draw (5,0) +(4:0.4) arc (4:364:0.4);
         \draw[->] (5,0) +(90:0.4) arc (90:100:0.4);
         
         			\draw(7,-1)  node {\tiny{$\times$}};

         \draw (7,-1) +(4:0.4) arc (4:364:0.4);
         \draw[->] (7,-1) +(90:0.4) arc (90:100:0.4);
         
         			\draw(8.5,2)  node {\tiny{$\times$}};

        \end{tikzpicture}
\caption{ \footnotesize  \textbf{The contour $\Gamma_\varepsilon$.} The crosses and thick lines represent the eigenvalues and essential spectrum of $H$, respectively. }\label{fig:contour}.
\end{center}
\end{figure}
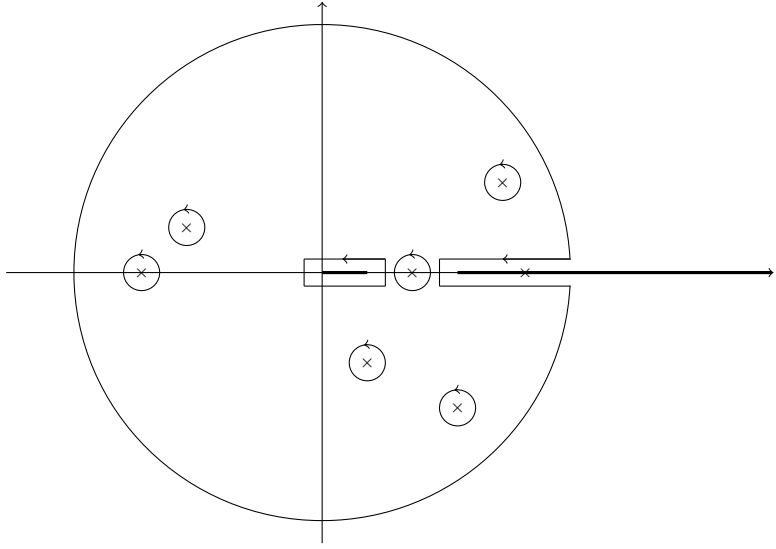

We are now ready to prove Proposition \ref{prop:resolution}. Recall that the function $r$ has been defined in \eqref{eq:def_r}, for some $z_0\in\rho(H)$, $z_0\in\mathbb{C}\setminus\mathbb{R}$.

\begin{proof}[Proof of Proposition \ref{prop:resolution}]

In order to be able to apply Proposition \ref{prop:functional_calculus_reg} for a suitable function $h$, it is convenient to add a decaying term to the regularizing function $r$. Hence we set
\begin{equation}\label{eq:def_tilde_r}
\tilde r(z) := (z-z_0)^{-1}r(z) = (z-z_0)^{-(\nu_1+\cdots+\nu_n+\nu_\infty+1)} \prod_{j=1}^n (z-\lambda_j)^{\nu_j},
\end{equation}
so that $\tilde r$ satisfies \eqref{eq:h_assumpt} in any interval $I\subset\mathbb{R}$. Let $\bar{\C}:=\C\cup\{\infty\}$ denote the Riemann sphere. We note that $\sigma(H)\cup \{\infty\} \subset U_\varepsilon\cup\{\infty\}$ and that $U_\varepsilon\cup\{\infty\}$ is an open set in $\bar{\C}$. The function $\tilde r$ is extended to $\bar\C$ by setting $\tilde r(\infty)=0$. By the Riesz-Dunford functional calculus (see e.g. \cite[Section VIII.9]{DuSc71_01}), using in particular that $\tilde r$ is analytic in a neighborhood of $\sigma(H)\cup\{\infty\}$ in $\bar\C$, we have that
\begin{align}
\tilde r (H) = - \frac{1}{2i\pi}\int_{\Gamma_{\varepsilon}}\tilde r(z) \Res_H(z) \mathrm{d}z. \label{eq:Gamma0}
\end{align}

We consider successively the contributions to this integral from $\Gamma_{\mathrm{(i)}}$, $\Gamma_{\varepsilon,\mathrm{(ii)}}$ and $\Gamma_{\varepsilon,\mathrm{(iii)}}$. The contribution from $\Gamma_{\mathrm{(i)}}$ gives, by definition, the Riesz projection onto the discrete spectral subspace of $H$,
\begin{equation}
- \frac{1}{2i\pi}\int_{\Gamma_{\mathrm{(i)}}}\tilde r(z) \Res_H(z) \mathrm{d}z = \tilde r(H) \Pi_{\mathrm{disc}}(H). \label{eq:Gamma1}
\end{equation}

The contribution from $\Gamma_{\varepsilon,\mathrm{(ii)}}$ gives, for each rectangle, four terms. The integrals over the vertical segments are of order $\mathcal{O}(\varepsilon)$. This easily follows from the fact that $z\mapsto|\tilde r(z)|\|\Res_H(z)\|_{\mathcal{L}(\Hi)}$ is uniformly bounded on these segments, whose lengths are equal to $2\varepsilon$. The sum of the integrals over the horizontal segments can be rewritten as
\begin{align*}
 \frac{1}{2i\pi}\int_{e_\ell+\delta/2}^{e_{\ell+1}-\delta/2} \big( \tilde r(\lambda+i\varepsilon) \Res_H(\lambda+i\varepsilon) - \tilde r(\lambda-i\varepsilon) \Res_H(\lambda-i\varepsilon) \big) \mathrm{d}\lambda.
\end{align*}
Applying Proposition \ref{prop:functional_calculus_reg}, we deduce that
\begin{align*}
&\wlim_{\varepsilon\rightarrow 0^+} \int_{e_\ell+\delta/2}^{e_{\ell+1}-\delta/2} \big( \tilde r(\lambda+i\varepsilon) \Res_H(\lambda+i\varepsilon) - \tilde r(\lambda-i\varepsilon) \Res_H(\lambda-i\varepsilon) \big)\mathrm{d}\lambda \\
&= \int_{e_\ell+\delta/2}^{e_{\ell+1}-\delta/2} \tilde r(\lambda) \big( \Res_H(\lambda+i0^+) -  \Res_H(\lambda-i0^+) \big) \mathrm{d}\lambda.
\end{align*}
Therefore, the contribution to \eqref{eq:Gamma0} from $\Gamma_{\varepsilon,\mathrm{(ii)}}$ gives
\begin{equation}
\wlim_{\varepsilon\rightarrow 0^+} - \frac{1}{2i\pi} \int_{\Gamma_{\varepsilon,\mathrm{(ii)}}}\tilde r(z) \Res_H(z) \mathrm{d}z = \frac{1}{2i\pi} \sum_{\ell=1}^{p-1}\int_{e_\ell+\delta/2}^{e_{\ell+1}-\delta/2} \tilde r(\lambda) \big( \Res_H(\lambda+i0^+) - \Res_H(\lambda-i0^+) \big) \mathrm{d}\lambda. \label{eq:Gamma2}
\end{equation}

It remains to consider the contribution from $\Gamma_{\varepsilon,\mathrm{(iii)}}$. As before, the integral over the small vertical segment is of order $\mathcal{O}(\varepsilon)$. The sum of the integrals over the horizontal segments can be rewritten as
\begin{align*}
 \frac{1}{2i\pi}\int_{e_p+\delta/2}^{\varepsilon^{-3}} \big( \tilde r(\lambda+i\varepsilon) \Res_H(\lambda+i\varepsilon) - \tilde r(\lambda-i\varepsilon) \Res_H(\lambda-i\varepsilon) \big) \mathrm{d}\lambda.
\end{align*}
First, we note that
\begin{align*}
\int_{\varepsilon^{-3}}^\infty \big( \tilde r(\lambda+i\varepsilon) \Res_H(\lambda+i\varepsilon) - \tilde r(\lambda-i\varepsilon) \Res_H(\lambda-i\varepsilon) \big) \mathrm{d}\lambda = \mathcal{O}(\varepsilon).
\end{align*}
Indeed, we have $|\tilde r(\lambda\pm i\varepsilon)|\le \mathrm{c}|\lambda|^{-1}$ for $\lambda$ large enough and hence, by the Cauchy-Schwarz inequality,
\begin{align*}
\int_{\varepsilon^{-3}}^\infty \big|\big\langle u, \tilde r(\lambda\pm i\varepsilon) \Res_H(\lambda\pm i\varepsilon)v\big\rangle_\Hi\big|\mathrm{d}\lambda &\le \mathrm{c}\varepsilon^{\frac32}\|u\|_\Hi \Big( \int_{\varepsilon^{-3}}^\infty \big\|  \Res_H(\lambda\pm i\varepsilon)v\big\|_\Hi^2 \mathrm{d}\lambda \Big )^{\frac12} \\
&\le \mathrm{c}\varepsilon \|u\|_\Hi \|v\|_{\Hi},
\end{align*}
the second inequality being a consequence of Lemma \ref{lm:res_bound}\ref{it:res_bound2}. Hence
\begin{align*}
& \frac{1}{2i\pi}\int_{e_p+\delta/2}^{\varepsilon^{-3}} \big\langle u, \big (\tilde r(\lambda+i\varepsilon) \Res_H(\lambda+i\varepsilon) - \tilde r(\lambda-i\varepsilon) \Res_H(\lambda-i\varepsilon) \big) v\big\rangle_\Hi \mathrm{d}\lambda \\
& = \frac{1}{2i\pi}\int_{e_p+\delta/2}^{\infty} \big\langle u, \big( \tilde r(\lambda+i\varepsilon) \Res_H(\lambda+i\varepsilon) - \tilde r(\lambda-i\varepsilon) \Res_H(\lambda-i\varepsilon) \big) v\big\rangle_\Hi \mathrm{d}\lambda +\mathcal{O}(\varepsilon)\|u\|_\Hi \|v\|_{\Hi}.
\end{align*}
Combining this with Proposition \ref{prop:functional_calculus_reg}, we obtain that
\begin{align*}
&\wlim_{\varepsilon\rightarrow 0^+} \frac{1}{2i\pi}\int_{e_p+\delta/2}^{\varepsilon^{-3}} \big( \tilde r(\lambda+i\varepsilon) \Res_H(\lambda+i\varepsilon) - \tilde r(\lambda-i\varepsilon) \Res_H(\lambda-i\varepsilon) \big) \mathrm{d}\lambda \\
&= \frac{1}{2i\pi}\int_{e_p+\delta/2}^{\infty} \tilde r(\lambda) \big( \Res_H(\lambda+i0^+) - \Res_H(\lambda-i0^+) \big) \mathrm{d}\lambda.
\end{align*}

The integral over the circle arc in $\Gamma_{\varepsilon,\mathrm{(iii)}}$ can be estimated as follows. Denote by $\mathcal{C}_\varepsilon$ this circle arc. For $z\in\mathcal{C}_\varepsilon$, $|\mathrm{Im}(z)|\le \varepsilon_0$, Lemma \ref{lm:res_bound}\ref{it:res_bound1} implies that $\|\Res_H(z)\|_{\mathcal{L}(\Hi)} \le c\varepsilon^{-2}$. Using in addition that $|\tilde r(z)|\le C|z|^{-1}=\mathcal{O}(\varepsilon^3)$ for $z\in\mathcal{C}_\varepsilon$, we obtain
\begin{align*}
\Big | \int_{\mathcal{C}_{\varepsilon}\cap \{|\mathrm{Im}(z)|\le\varepsilon_0\} }\tilde r(z) \Res_H(z) \mathrm{d}z \Big |=\mathcal{O}(\varepsilon).
\end{align*}
For $z\in\mathcal{C}_\varepsilon$, $|\mathrm{Im}(z)|\ge\varepsilon_0$, it suffices to use that $\|\Res_H(z)\|_{\mathcal{L}(\Hi)} \le c$ together with $|\tilde r(z)|\le C|z|^{-1}=\mathcal{O}(\varepsilon^3)$ to conclude that
\begin{align*}
\Big | \int_{\mathcal{C}_{\varepsilon}\cap \{|\mathrm{Im}(z)|\ge\varepsilon_0\} }\tilde r(z) \Res_H(z) \mathrm{d}z \Big |=\mathcal{O}(\varepsilon^3).
\end{align*}

Putting together the previous estimates, we have shown that
\begin{equation}
\wlim_{\varepsilon\rightarrow 0^+} - \frac{1}{2i\pi} \int_{\Gamma_{\varepsilon,\mathrm{(iii)}}}\tilde r(z) \Res_H(z) dz = \frac{1}{2i\pi}\int_{e_p+\delta/2}^{\infty} \tilde r(\lambda) \big( \Res_H(\lambda+i0^+) - \Res_H(\lambda-i0^+) \big) \mathrm{d}\lambda. \label{eq:Gamma3}
\end{equation}
Equations \eqref{eq:Gamma0}, \eqref{eq:Gamma1}, \eqref{eq:Gamma2} and \eqref{eq:Gamma3}, together with the fact that $\Res_H(\lambda+i0^+)=\Res_H(\lambda-i0^+)=\Res_H(\lambda)$ if $\lambda\in\rho(H)$, yield
\begin{align}\label{eq:Gamma_last}
\tilde r(H) &= \tilde r(H) \Pi_{\mathrm{disc}}(H) + \frac{1}{2i\pi}\int_{\sigma_{\mathrm{ess}}(H)} \tilde r(\lambda) \big( \Res_H(\lambda+i0^+) - \Res_H(\lambda-i0^+) \big) \mathrm{d}\lambda. 
\end{align}
It remains to show that one can replace $\tilde r$ by $r$ in the previous equation. For the first two terms, we have that $\tilde r(H)=\Res_H(z_0) r(H)$ by definition, while \eqref{eq:Projection_spectral_calcul_fonctionnel_resolvent_regul} in Proposition \ref{prop:functional_calculus_reg} implies
\begin{align*}
&\int_{\sigma_{\mathrm{ess}}(H)} \tilde r(\lambda) \big( \Res_H(\lambda+i0^+) - \Res_H(\lambda-i0^+) \big) \mathrm{d}\lambda \\
&= \Res_H(z_0) \int_{\sigma_{\mathrm{ess}}(H)} r(\lambda) \big( \Res_H(\lambda+i0^+) - \Res_H(\lambda-i0^+) \big) \mathrm{d}\lambda .
\end{align*}
Hence, applying $(H-z_0)$ to both sides of \eqref{eq:Gamma_last}, we obtain \eqref{eq:resolution}. This concludes the proof.
\end{proof}

	\subsection{Proof of Theorem \ref{thm:caracterisation_etat_disparaissent_à_infini}}\label{subsec:ads}

		Now we prove Theorem \ref{thm:caracterisation_etat_disparaissent_à_infini}. We recall that  $\Hi_\mathrm{p}^\pm(H)$ is the vector space spanned by all generalized eigenstates corresponding to eigenvalues $\lambda\in\mathbb{C}$ such that $\mp\mathrm{Im}\lambda >0$ and that $\Hi_\mathrm{ads}^\pm$ is the closure of $\{ u\in\Hi,~\lim_{t\rightarrow \infty}\|e^{\pm itH}u\|_\Hi=0\}$. We begin with proving the following easy inclusion, $\Hi_\mathrm{p}^\pm(H)\subset\Hi_\mathrm{ads}^\pm(H)$, which holds under the assumption that $H$ has finitely many eigenvalues with finite algebraic multiplicities.

		\begin{proposition}\label{prop:easy_incl}
			Suppose that Hypothesis \ref{hyp:valeurs_propres_H} holds. Then
			\begin{equation*}
				\Hi_\mathrm{p}^\pm(H)\subset\Hi_\mathrm{ads}^\pm(H).
			\end{equation*}
		\end{proposition}
	
		\begin{proof}
			We prove that $\Hi_\mathrm{p}^+(H)\subset\Hi_{\mathrm{ads}}^+(H)$. Let $\lambda\in\sigma_{\mathrm{disc}}(H)$ with $\im(\lambda)<0$ and let $u\in\Ran(\Pi_\lambda(H))$. For all $t>0$, we estimate
			\begin{equation*}
				\norme{e^{-itH}u}_\Hi\leq\frac{1}{2\pi}\int_\gamma e^{t\im(z)}\norme{\Res_H(z)u}_\Hi \mathrm{d}z,
			\end{equation*}
			where $\gamma$ is the circle defined as in \eqref{eq:Projection_de_Riesz_def_pour_H}. Since $\gamma\subset\mathbb{C}^-$, we have
			\begin{equation*}
				\norme{e^{-itH}u}_\Hi\leq\frac{e^{-t\delta}}{2\pi} \int_\gamma \norme{\Res_H(z)u}_\Hi \mathrm{d}z,
			\end{equation*}
			for some $\delta>0$. Hence $u\in\Hi_{\mathrm{ads}}^+(H)$. The proof of $\Hi_\mathrm{p}^-(H)\subset\Hi_{\mathrm{ads}}^-(H)$ is analogous.
			\end{proof}
		
		To prove the converse inclusion, we will use the following easy lemma. Recall that
		\begin{equation*}
			\Pi^\pm_{\mathrm{disc}}(H)=\sum_{\lambda\in\sigma_{\mathrm{disc}}(H),\mp\mathrm{Im}(\lambda)>0}\Pi_\lambda(H), \quad \Pi^0_{\mathrm{disc}}(H)=\sum_{\lambda\in\sigma_{\mathrm{disc}}(H),\mathrm{Im}(\lambda)=0}\Pi_\lambda(H).
		\end{equation*}
		
			\begin{lemma}\label{lm:incl}
			Suppose that Hypothesis \ref{hyp:valeurs_propres_H} holds. Then
			\begin{equation*}
				\Hi_\mathrm{ads}^\pm(H)\subset \mathrm{Ker}\big(\Pi_{\mathrm{disc}}^\mp(H)+\Pi_{\mathrm{disc}}^0(H)\big).
			\end{equation*}
		\end{lemma}
		
		\begin{proof}
			Let $u\in \Hi_\mathrm{ads}^+(H)$. We have
			\begin{equation}\label{eq:l1}
				\big(\Pi_{\mathrm{disc}}^-(H)+\Pi_{\mathrm{disc}}^0(H)\big)e^{-itH}u=e^{-itH}\big(\Pi_{\mathrm{disc}}^-(H)+\Pi_{\mathrm{disc}}^0(H)\big)u.
			\end{equation}
			The restriction of $H$ to $\mathrm{Ran}(\Pi_{\mathrm{disc}}^-(H)+\Pi_{\mathrm{disc}}^0(H))=\Hi_{\mathrm{disc}}^-(H)\oplus\Hi_{\mathrm{disc}}^0(H)$ is a linear mapping from a finite dimensional space to itself, whose eigenvalues have non-negative imaginary parts. Hence, by Lyapunov's Theorem,
			\begin{equation*}
				\lim_{t\to\infty}\big\|e^{-itH}\big(\Pi_{\mathrm{disc}}^-(H)+\Pi_{\mathrm{disc}}^0(H)\big)u\big\|_\Hi\neq0 \quad \text{unless}\quad \big(\Pi_{\mathrm{disc}}^-(H)+\Pi_{\mathrm{disc}}^0(H)\big)u=0.
			\end{equation*}
			Since the left-hand-side of \eqref{eq:l1} tends to $0$ as $t\to\infty$ (for $u\in \Hi_\mathrm{ads}^+(H)$), we conclude that indeed $(\Pi_{\mathrm{disc}}^-(H)+\Pi_{\mathrm{disc}}^0(H))u=0$. Hence we have proven that $\Hi_\mathrm{ads}^+(H)\subset \mathrm{Ker}(\Pi_{\mathrm{disc}}^-(H)+\Pi_{\mathrm{disc}}^0(H))$. The proof of $\Hi_\mathrm{ads}^-(H)\subset \mathrm{Ker}(\Pi_{\mathrm{disc}}^+(H)+\Pi_{\mathrm{disc}}^0(H))$ is similar.
%
		\end{proof}
		
		Now we are ready to prove Theorem \ref{thm:caracterisation_etat_disparaissent_à_infini}. We will use the regularizing function $r$ defined in \eqref{eq:def_r}, for some $z_0\in\rho(H)$, $z_0\in\mathbb{C}\setminus\mathbb{R}$.
			
		\begin{proof}[Proof of Theorem \ref{thm:caracterisation_etat_disparaissent_à_infini}]
			In view of Proposition \ref{prop:easy_incl}, it remains to show that $\Hi_\mathrm{ads}^\pm(H)\subset\Hi_\mathrm{p}^\pm(H)$. We prove that $\Hi_\mathrm{ads}^+(H)\subset\Hi_\mathrm{p}^+(H)$, the inclusion $\Hi_\mathrm{ads}^-(H)\subset\Hi_\mathrm{p}^-(H)$ can be proven in the same way.

			Let $u\in\Hi_\mathrm{ads}^+(H)$. By Proposition \ref{prop:resolution}, we can write
\begin{align}\label{eq:l20}
r(H) u = \Pi_{\mathrm{disc}}(H)r(H)u + r_{\mathrm{ess}}(H)u ,
\end{align}
where
\begin{equation*}
r_{\mathrm{ess}}(H):=\wlim_{\varepsilon\rightarrow 0^+}\frac{1}{2\pi i}\int_{\sigma_{\mathrm{ess}}(H)}r(\lambda)\big(\Res_{H}(\lambda+i\varepsilon)-\Res_{H}(\lambda-i\varepsilon)\big)\mathrm{d}\lambda 
\end{equation*}
is a bounded operator.

			Clearly, since $r(H)$ is bounded and commutes with $e^{-itH}$, we have that $r(H)u\in\Hi_\mathrm{ads}^+(H)$. Lemma \ref{lm:incl} then implies that $r(H)u\in	\mathrm{Ker}(\Pi_{\mathrm{disc}}^-(H)+\Pi_{\mathrm{disc}}^0(H))$. Inserting this into	\eqref{eq:l20} gives
\begin{align}\label{eq:l2}
r(H) u = \Pi_{\mathrm{disc}}^+(H)r(H)u + r_{\mathrm{ess}}(H)u .
\end{align}			
			Now we show that $u\in\mathrm{Ker}(r_{\mathrm{ess}}(H))$. We have $\mathrm{Ker}(r_{\mathrm{ess}}(H))=\mathrm{Ran}(r_{\mathrm{ess}}(H)^*)^\perp$ where, by Proposition \ref{prop:Projection_Spectrale_Stone_regul}, 
			\begin{equation*}
					r_{\mathrm{ess}}(H)^*=\wlim_{\varepsilon\rightarrow 0^+}\frac{1}{2\pi i}\int_{\sigma_{\mathrm{ess}}(H)}\overline{r(\lambda)}\big(\Res_{H^*}(\lambda+i\varepsilon)-\Res_{H^*}(\lambda-i\varepsilon)\big)\mathrm{d}\lambda. 
				\end{equation*}
				For all $v=r_{\mathrm{ess}}(H)^*w\in \mathrm{Ran}(r_{\mathrm{ess}}(H)^*)$, we can write
		\begin{equation}\label{eq:l3}
				|\langle v,u\rangle_\Hi|=\big|\langle e^{-itH^*}v,e^{-itH}u\rangle_\Hi\big|\le\big\|e^{-itH^*}v\big\|_\Hi\big\|e^{-itH}u\big\|_\Hi.
			\end{equation}
			By \eqref{eq:algebra_morph_regul} and \eqref{eq:Projection_spectral_calcul_fonctionnel_semi-group_regul} in Proposition \ref{prop:functional_calculus_reg}, we have
			\begin{equation*}
				\big\|e^{-itH^*}v\big\|_\Hi=\big\|e^{-itH^*}r_{\mathrm{ess}}(H)^*w\big\|_\Hi\le\mathrm{c}\|w\|_\Hi.
			\end{equation*}
			Inserting this into \eqref{eq:l3}, letting $t\to\infty$ and using that $u\in\Hi_\mathrm{ads}^+(H)$, we obtain that $\langle v,u\rangle=0$. Hence $u\in\mathrm{Ker}(r_{\mathrm{ess}}(H))$ and therefore \eqref{eq:l2} reduces to
			\begin{align}\label{eq:l4}
r(H) u = \Pi_{\mathrm{disc}}^+(H)r(H)u .
\end{align}
			
			We have proven that $r(H) u$ belongs to $\Hi_{\mathrm{disc}}^+(H)=\Hi_{\mathrm{p}}^+(H)$. Now the Riesz-Dunford functional calculus shows that the restriction of $r(H)$ to $\Hi_{\mathrm{p}}^+(H)$ is bounded invertible. Hence we deduce that $u\in\Hi_{\mathrm{p}}^+(H)$. This concludes the proof.
			\end{proof}

	\subsection{Proof of Theorem \ref{thm:caracterisation_espace_absolument_continu}}\label{subsec:abs_cont}

Recall that $\Hi_{\mathrm{p}}(H)$ is the closure of the vector space spanned by all generalized eigenvectors of $H$, and that the absolutely continuous spectral subspace of $H$ has been defined in Section \ref{subsec:subspaces}. We want to prove that $\Hi_\mathrm{ac}(H)=\Hi_\mathrm{p}(H^*)^\perp$. We begin with the following proposition which only requires that $\Hi_\mathrm{p}(H)$ be finite dimensional.

		\begin{proposition}\label{prop:easy_incl_ac}
			Suppose that Hypothesis \ref{hyp:valeurs_propres_H} holds. Then
			\begin{equation*}
				\Hi_\mathrm{ac}(H)\subset\Ran(\Id-\Pi_\mathrm{p}(H))=\Hi_\mathrm{p}(H^*)^\perp.
			\end{equation*}
		\end{proposition}
		\begin{proof}
			 Let $u\in\Hi_\mathrm{ac}(H)$. We decompose
		\begin{equation*}
			u=\Pi_\mathrm{p}(H)u+(\Id-\Pi_\mathrm{p}(H))u.
		\end{equation*} 
		Suppose by contradiction that $\Pi_\mathrm{p}(H) u\neq 0$. Then there exists $v\in\Hi_\mathrm{p}(H^*)$ such that $\scal{u}{v}_\Hi=1$. Indeed, if $\scal{u}{v}_\Hi=0$ for all $v\in\Hi_\mathrm{p}(H^*)$, then $u\in\Hi_\mathrm{p}(H^*)^\perp=\Ker(\Pi_\mathrm{p}(H))$.
		
		Now, if $v\in\Hi_\mathrm{p}(H^*)$ satisfies $\scal{u}{v}=1$, the map $t\mapsto\scal{e^{-itH}u}{v}_\Hi$ cannot belong to $\mathrm{L}^2(\R,\C)$ since
		\begin{equation*}
		\scal{e^{-itH}u}{v}_\Hi=\big \langle u,e^{itH^*}v \big\rangle_\Hi
		\end{equation*}
		and the restriction of $iH^*$ to $\Hi_\mathrm{p}(H^*)$ is a linear mapping on a finite dimensional vector space.
		
		This proves that $\Pi_\mathrm{p}(H) u=0$ and hence that $u\in\Ran(\Id-\Pi_\mathrm{p}(H))$.
		\end{proof}

To prove the converse inclusion, we will use the following easy lemma. Recall that $r$ has been defined in \eqref{eq:def_r}, for some $z_0\in\rho(H)$, $z_0\in\mathbb{C}\setminus\mathbb{R}$.
		
			\begin{lemma}\label{lm:closed_range}
			Suppose that Hypotheses \ref{hyp:valeurs_propres_H} and \ref{hyp:singularités_spectrales} hold. Then
			\begin{equation*}
				\mathrm{Ran}\big(r(H)(\Id-\Pi_\mathrm{p}(H))\big) \text{ is dense in } \Hi_\mathrm{p}(H^*)^\perp.
			\end{equation*}
		\end{lemma}

\begin{proof}
Since $H$ commutes with $\Pi_{\mathrm{p}}(H)$, $r(H)$ preserves $\Hi_\mathrm{p}(H^*)^\perp=\Ran(\Id-\Pi_\mathrm{p}(H))$. The adjoint of $r(H)(\Id-\Pi_\mathrm{p}(H))$ is given by
\begin{equation*}
\big[r(H)(\Id-\Pi_\mathrm{p}(H))\big]^* =  \Big [ \prod_{j=1}^n (H^*-\lambda_j)^{\nu_j} \Big ] \Res_{H^*}(\bar z_0)^{-(\nu_1+\cdots+\nu_n+\nu_\infty)} (\Id-\Pi_\mathrm{p}(H^*)).
\end{equation*}
Since the restriction of $H^*$ to $\Ran(\Id-\Pi_\mathrm{p}(H^*))$ has no eigenvalues, the right-hand-side of the previous equation is an injective operator, which concludes the proof of the lemma.
\end{proof}

Now we prove Theorem \ref{thm:caracterisation_espace_absolument_continu}. 

	\begin{proof}[Proof of Theorem \ref{thm:caracterisation_espace_absolument_continu}]

By Proposition \ref{prop:easy_incl_ac}, we know that $\Hi_\mathrm{ac}(H)\subset\Hi_\mathrm{p}(H^*)^\perp$. To prove that $\Hi_\mathrm{p}(H^*)^\perp\subset\Hi_\mathrm{ac}(H)$, since $\Hi_\mathrm{ac}(H)$ is closed, it suffices by 
Lemma \ref{lm:closed_range} to show that 
\begin{equation*}
\mathrm{Ran}\big(r(H)(\Id-\Pi_\mathrm{p}(H))\big)\subset\Hi_\mathrm{ac}(H).
\end{equation*}
		
		Let $u=r(H)(\Id-\Pi_\mathrm{p}(H))v \in \mathrm{Ran}(r(H)(\Id-\Pi_\mathrm{p}(H)))$. Let $(v_n)_{n\in\mathbb{N}}=(Cw_n)_{n\in\mathbb{N}}$ be a sequence in $\Hi_C=\mathrm{Ran}(C)$ such that $v_n\to v$ as $n\to\infty$ (recall that $\Hi_C$ is dense in $\Hi$). We claim that
		\begin{equation}\label{eq:j1}
			u_n:=r(H)(\Id-\Pi_{\mathrm{p}}(H))v_n\in\Hi_\mathrm{ac}(H).
		\end{equation}
To prove \eqref{eq:j1}, applying Proposition \ref{prop:resolution} and the fact that $(\Id-\Pi_{\mathrm{p}}(H))\Pi_{\mathrm{disc}}(H)=0$, we first observe that, for all $\varphi\in\Hi$,
\begin{align*}
\langle u_n,\varphi\rangle&=\big\langle r(H)(\Id-\Pi_{\mathrm{p}}(H))v_n,\varphi\big\rangle\\
&=-\frac{1}{2i\pi}\int_{\Lambda}r(\lambda)\big\langle\big(\Res_{H}(\lambda+i0^+)-\Res_{H}(\lambda-i0^+)\big)v_n,\psi\big\rangle\mathrm{d}\lambda ,
\end{align*}
where, to simplify notations, we set $\Lambda:=\sigma_{\mathrm{ess}}(H)$ and $\psi:=(\Id-\Pi_{\mathrm{p}}(H^*))\varphi$. 
Next we apply \eqref{eq:Projection_spectral_calcul_fonctionnel_semi-group_regul} in Proposition \ref{prop:functional_calculus_reg} and Plancherel's Theorem to obtain
		\begin{align}
			&\int_\R\abso{\scal{e^{-itH}u_n}{\varphi}}_\Hi^2\mathrm{d}t\nonumber\\
			&=\int_{\Lambda}\abso{\scal{r(\lambda)\left(\Res_H(\lambda-i0^+)-\Res_H(\lambda+i0^+)\right)v_n}{\psi}_\Hi}^2\mathrm{d}\lambda . \label{eq:t1}
		\end{align}
		Using the resolvent identity 
		\begin{equation*}
			\Res_H(z)=\Res_0(z)-\Res_0(z)V\Res_0(z)+\Res_0(z)V\Res_H(z)V\Res_0(z),
		\end{equation*}
		for all $z\in\rho(H)$, we decompose 
		\begin{align}
			&\int_{\Lambda}\abso{r(\lambda)\scal{\left(\Res_H(\lambda-i\varepsilon)-\Res_H(\lambda+i\varepsilon)\right)v_n}{\psi}_\Hi}^2\mathrm{d}\lambda\nonumber\\
			&\leq\int_{\Lambda}\abso{r(\lambda)\scal{\left(\Res_0(\lambda-i\varepsilon)-\Res_0(\lambda+i\varepsilon)\right)v_n}{\psi}_\Hi}^2\mathrm{d}\lambda\nonumber\\
			&\quad+\int_{\Lambda}\abso{r(\lambda)\scal{\left(\Res_0(\lambda-i\varepsilon)V\Res_0(\lambda-i\varepsilon)\right)v_n}{\psi}_\Hi}^2\mathrm{d}\lambda\nonumber\\
			&\quad+\int_{\Lambda}\abso{r(\lambda)\scal{\left(\Res_0(\lambda+i\varepsilon)\Res_0(\lambda+i\varepsilon)\right)v_n}{\psi}_\Hi}^2\mathrm{d}\lambda\nonumber\\
			&\quad+\int_{\Lambda}\abso{r(\lambda)\scal{\left(\Res_0(\lambda-i\varepsilon)V\Res_H(\lambda-i\varepsilon)V\Res_0(\lambda-i\varepsilon)\right)v_n}{\psi}_\Hi}^2\mathrm{d}\lambda\nonumber\\
			&\quad+\int_{\Lambda}\abso{r(\lambda)\scal{\left(\Res_0(\lambda+i\varepsilon)V\Res_H(\lambda+i\varepsilon)V\Res_0(\lambda+i\varepsilon)\right)v_n}{\psi}_\Hi}^2\mathrm{d}\lambda.\label{eq:carac-ses_abso_eq_géné}
		\end{align}
		We claim that each term of the right-hand-side of the previous equation is bounded by $\mathrm{c}_n\norme{\varphi}_\Hi^2$, for some positive constant $\mathrm{c}_n$ depending on $n$. We estimate each term separately.
		
		For the first term in the right-hand-side of \eqref{eq:carac-ses_abso_eq_géné}, it suffices to use that $H_0$ is a self-adjoint operator with purely absolutely continuous spectrum (by Hypothesis \ref{hyp:principe_absorption_limite}), which yields
		\begin{equation}\label{eq:carac_ses_abso_H_0}
			\int_{\Lambda}\abso{r(\lambda)\scal{\left(\Res_0(\lambda-i0^+)-\Res_0(\lambda+i0^+)\right)v_n}{\psi}}^2\mathrm{d}\lambda\leq \mathrm{c}_n\|r\|^2_{L^\infty}\norme{\psi}_\Hi^2.
		\end{equation}
		
		Next, remembering that $v_n=Cw_n$ and $V=CWC$, the second and third terms in the right-hand-side of \eqref{eq:carac-ses_abso_eq_géné} are estimated as 
			\begin{align}
				&\int_{\Lambda}\abso{r(\lambda)\scal{\left(\Res_0(\lambda\pm i\varepsilon)CWC\Res_0(\lambda\pm i\varepsilon)\right)Cw_n}{\psi}}_\Hi^2\mathrm{d}\lambda \nonumber\\
				&\leq\norme{W}_{\mathcal{L}(\Hi)}^2\|r\|_{L^\infty}^2\sup_{\lambda\in\Lambda}\left(\norme{C\Res_0(\lambda\pm i\varepsilon)C}_{\mathcal{L}(\Hi)}^2\right)\norme{w_n}^2_{\Hi}\int_{\Lambda}\norme{C\Res_0(\lambda\mp i\varepsilon)\psi}^2_\Hi\mathrm{d}\lambda \nonumber \\
				&\leq \mathrm{c}_n\|\psi\|^2_\Hi , \label{eq:t2}
			\end{align}
where we used Hypothesis \ref{hyp:principe_absorption_limite} (and \eqref{eq:Smooth_Ka65_01}) in the second inequality.

			Finally, to estimate the fourth and fifth terms in the right-hand-side of \eqref{eq:carac-ses_abso_eq_géné}, we write similarly
			\begin{align}
				&\int_\Lambda\abso{r(\lambda)\scal{\left(\Res_0(\lambda\pm i\varepsilon)CWC\Res_H(\lambda\pm i\varepsilon)CWC\Res_0(\lambda\pm i\varepsilon)\right)Cw_n}{\psi}_\Hi}^2\mathrm{d}\lambda\nonumber\\
				&\leq \norme{W}_{\mathcal{L}(\Hi)}^2\sup_{\lambda\in\Lambda}\left(|r(\lambda)|\norme{C\Res_H(\lambda\pm i\varepsilon)CW}_{\mathcal{L}(\Hi)}^2\right)\sup_{\lambda\in\Lambda}\left(\norme{C\Res_0(\lambda\pm i\varepsilon)C}_{\mathcal{L}(\Hi)}^2\right)\norme{w_n}^2_{\Hi} \nonumber \\
				&\quad\times\int_\Lambda\norme{C\Res_0(\lambda\mp i\varepsilon)\psi}_\Hi^2\mathrm{d}\lambda \nonumber \\
				&\leq \mathrm{c}_n\|\psi\|^2_\Hi, \label{eq:t3}
			\end{align}
			where we used Hypotheses \ref{hyp:principe_absorption_limite} and \ref{hyp:singularités_spectrales} in the second inequality.

Inserting \eqref{eq:carac-ses_abso_eq_géné}--\eqref{eq:t3} into \eqref{eq:t1} and using that $\|\psi\|_\Hi\le\|\varphi\|_\Hi$, we deduce that
		\begin{align*}
			&\int_\R\abso{\scal{e^{-itH}u_n}{\varphi}}_\Hi^2\mathrm{d}t \le C_n\|\varphi\|^2_\Hi.
		\end{align*}
Therefore, $u_n\in\Hi_{\mathrm{ac}}(H)$ for all $n\in\mathbb{N}$. Since $u_n\to u$ in $\Hi$ as $n\to\infty$, and since $\Hi_{\mathrm{ac}}(H)$ is closed, this implies that $u\in\Hi_{\mathrm{ac}}(H)$ and hence the proof of the theorem is complete.
	\end{proof}

Our last proposition shows that in the case where $H$ is dissipative, Hypothesis \ref{hyp:Conjugate_operator} can be dropped in the statement of Theorem \ref{thm:caracterisation_espace_absolument_continu}. Using the notations from Section \ref{sec:dissipative case}, we know that the only possible generalized eigenvectors corresponding to a real eigenvalue of $H$ are eigenvectors in the usual sense, and that they are also eigenvectors of $H_{V_1}$ (and of $H^*$). In other words, if $\lambda\in\mathbb{R}$ is an eigenvalue of $H$, then $\mathrm{Ker}((H-\lambda)^2)=\mathrm{Ker}(H-\lambda)$ and we have
\begin{equation*}
u\in\mathrm{Ker}(H-\lambda)\,\Rightarrow\,u\in\mathrm{Ker}(H_{V_1}-\lambda)\cap\mathrm{Ker}(V_2)\,\Rightarrow\,u\in\mathrm{Ker}(H^*-\lambda),
\end{equation*}
see also Lemma \ref{lm:eigen-diss}. Choosing an orthogonal basis $\{e_1,\dots,e_n\}$ in $\mathrm{Ker}(H-\lambda)$, the spectral projection corresponding to $\lambda$ can then be defined in the usual way, setting
\begin{equation*}
\Pi_\lambda(H)u:=\sum_{j=1}^n\langle e_j,u\rangle e_j, \quad u\in\Hi.
\end{equation*}
One readily checks that $\Pi_\lambda(H)^*=\Pi_\lambda(H^*)$. The spectral projection $\Pi_{\mathrm{p}}(H)$ onto the point spectral subspace of $H$ can then be defined as in Section \ref{subsec:subspaces}. Modifying the previous proof in a straightforward way, we deduce the following.
\begin{proposition}\label{prop:ac-diss}
Suppose that Hypotheses \ref{hyp:principe_absorption_limite}-\ref{hyp:singularités_spectrales} hold and that $H$ is dissipative, $\mathrm{Im}(H)\le 0$. Then
		 	\begin{equation*}
		 		\Hi_\mathrm{ac}(H)=\Ran(\mathrm{Id}-\Pi_\mathrm{p}(H))=\Hi_\mathrm{p}(H^*)^\perp.
		 	\end{equation*}
\end{proposition}

\appendix

\section{Appendix to Section \ref{sec:spectr_sing}}\label{App:Propriétés_point_spect_reg}

In this appendix, we provide the proofs of Propositions \ref{lm:equivalence_point_spectral_regulier} and \ref{prop:reg_spec_local}. They consist in a suitable adaptation of the corresponding proofs in \cite{FaNi18_01} where the particular case of dissipative operators have been considered.

	\begin{proof}[Proof of Proposition \ref{lm:equivalence_point_spectral_regulier}]
		We prove the result in the case of an outgoing regular spectral point, the proof in the case of an incoming regular spectral point is identical.
		
		First we prove that \ref{lm:equivalence_point_spectral_regulier_def_limite}$\Rightarrow$\ref{lm:equivalence_point_spectral_regulier_def_inversible}. Suppose that $\lambda$ is an outgoing regular spectral point of $H$. There exists $\varepsilon_0>0$ such that, for all $\varepsilon\in(0,\varepsilon_0)$, $\Res_H(\lambda+ i\varepsilon)$ exists in $\mathcal{L}(\Hi)$. The resolvent identity gives
		\begin{equation}\label{eq:lm_carac_equiv_point_spec_regu_ID_resolvante}
			(\Id-C\Res_H(\lambda+ i\varepsilon)CW)(\Id+C\Res_0(\lambda+ i\varepsilon)CW)=\Id.
		\end{equation}
		Thus $\Id+C\Res_0(\lambda+i\varepsilon)CW$ is surjective on $\Hi$. Since $C\Res_0(\lambda+i\varepsilon)CW$ is compact, the Fredholm alternative implies that $\Id+C\Res_0(\lambda+i\varepsilon)CW$ is invertible in $\mathcal{L}(\Hi)$. Letting $\varepsilon\to0^+$, using that the limits in \eqref{eq:LAP_lambda} exist, we obtain from \eqref{eq:lm_carac_equiv_point_spec_regu_ID_resolvante} that
		\begin{equation*}
			(\Id-C\Res_H(\lambda+i0^+)CW)(\Id+C\Res_0(\lambda+i0^+)CW)=\Id.
		\end{equation*}
		Thus $\Id+C\Res_0(\lambda+ i0^+)CW$ is surjective, and hence invertible in $\mathcal{L}(\Hi)$ by the Fredholm alternative. 
		
		Next we prove that \ref{lm:equivalence_point_spectral_regulier_def_inversible}$\Rightarrow$\ref{lm:equivalence_point_spectral_regulier_def_limite}. Suppose that $\Id+C\Res_0(\lambda+ i0^+)CW$ is invertible in $\mathcal{L}(\Hi)$. Suppose by contradiction that $\lambda$ is an accumulation point of eigenvalues of $H$ located in $\lambda +i( 0,\infty)$. Then there exists a sequence $(\varepsilon_n)_{n\in\mathbb{N}}$ of positive real numbers such that $\varepsilon_n\to0$ as $n\to\infty$ and, for all $n\in\mathbb{N}$, vectors $u_n\in\Hi$, $\|u_n\|_\Hi=1$, such that
		\begin{equation*}
			(H-(\lambda+ i\varepsilon_n))u_n=0.
		\end{equation*}
		Applying $C\Res_0(\lambda+ i\varepsilon_n)$ to this equations yields
		\begin{equation*}
			C\Res_0(\lambda+ i\varepsilon_n)(H-(\lambda+i\varepsilon_n))u_n=(\Id+C\Res_0(\lambda+ i\varepsilon_n)CW)Cu_n=0.
		\end{equation*}
		Since $\Id+C\Res_0(\lambda+ i0^+)CW$ is invertible in $\mathcal{L}(\Hi)$, for $n$ large enough, $\Id+C\Res_0(\lambda+ i\varepsilon_n)CW$ is also invertible. Therefore $C\Res_0(\lambda+ i\varepsilon_n)(H-(\lambda+ i\varepsilon_n))$ is injective, which is a contradiction since $u_n\neq 0$.
		
		It remains to show that $C\Res_H(\lambda+ i\varepsilon)CW$ converges in $\mathcal{L}(\Hi)$ as $\varepsilon\to0^+$. Since for $\varepsilon>0$ small enough $\Id+C\Res_0(\lambda+ i\varepsilon)CW$ is invertible in $\mathcal{L}(\Hi)$, \eqref{eq:lm_carac_equiv_point_spec_regu_ID_resolvante} gives
		\begin{equation*}
			C\Res_H(\lambda+ i\varepsilon)CW=\big(\Id+C\Res_0(\lambda+ i\varepsilon)CW\big)^{-1}-\Id,
		\end{equation*}
		This proves that
		\begin{equation*}
			C\Res_H(\lambda+ i0^+)CW=\big(\Id+C\Res_0(\lambda+ i0^+)CW\big)^{-1}-\Id
		\end{equation*}
		exists in $\mathcal{L}(\Hi)$.
	\end{proof}

Before proving Proposition \ref{prop:reg_spec_local}, we recall the proof of Lemma \ref{lm:tech_est} which was used several times in the main text.

\begin{proof}[Proof of Lemma \ref{lm:tech_est}]
		Consider for instance the operator $\Res_0\left(\lambda+ i\varepsilon\right)C$. Let $\varepsilon>0$, $u\in\Hi$. We have
		\begin{align*}
			\norme{\Res_0\left(\lambda+i\varepsilon\right)Cu}_\Hi^2&=\scal{\Res_0\left(\lambda+i\varepsilon\right)Cu}{\Res_0\left(\lambda+i\varepsilon\right)Cu}_\Hi\\
			&=\frac{1}{2i\varepsilon}\scal{Cu}{\left\lbrack\Res_0\left(\lambda+i\varepsilon\right)-\Res_0\left(\lambda-i\varepsilon\right)\right\rbrack Cu}_{\Hi}\\
			&=\frac{1}{\varepsilon}\im\left(\scal{u}{C\Res_0\left(\lambda+i\varepsilon\right)Cu}_\Hi\right).
		\end{align*}
		Since the limits \eqref{eq:LAP_lambda} exist, there exists $c_0>0$ such that
		\begin{equation*}
			\sup_{\varepsilon>0}\im\left(\scal{u}{C\Res_0\left(\lambda+i\varepsilon\right)Cu}_\Hi\right)\leq c_0^2\norme{u}_\Hi^2. 
		\end{equation*}
		Hence
		\begin{equation*}
			\norme{\Res_0\left(\lambda+i\varepsilon\right)Cu}_\Hi^2\leq \frac{1}{\varepsilon}c_0^2\norme{u}_\Hi^2.
		\end{equation*}
		This proves the lemma for $\Res_0\left(\lambda+ i\varepsilon\right)C$. The proof for $\Res_0\left(\lambda- i\varepsilon\right)C$ is identical.
\end{proof}	
	
	\begin{proof}[Proof of Proposition \ref{prop:reg_spec_local}]
Again, we prove the result in the case of an outgoing regular spectral point, the proof in the case of an incoming regular spectral point being identical. 

\ref{prop:reg_spec_local_voisi}$\Rightarrow$\ref{prop:reg_spec_local_fixe} is obvious. We prove that \ref{prop:reg_spec_local_fixe}$\Rightarrow$\ref{prop:reg_spec_local_voisi}. Suppose that $\lambda$ is an outgoing regular spectral point. By Proposition \ref{lm:equivalence_point_spectral_regulier},
	\begin{equation*}
		A(\lambda):=\Id+C\Res_0(\lambda +i0^+)CW
	\end{equation*}
	is invertible in $\mathcal{L}(\Hi)$. Since the maps in \eqref{eq:ext_cont_reg} extend by continuity to $\mathring{D}(\lambda,r)\cap\bar\C^\pm$, there exists a compact interval $K_\lambda\subset\R$ whose interior contains $\lambda$ such that, for all $\mu\in K_\lambda$, $A(\mu)$ is invertible. By Proposition \ref{lm:equivalence_point_spectral_regulier}, this implies that each $\mu\in K_\lambda$ is not an accumulation point of eigenvalues located in $\mu +i(0,\infty)$ and that, for all $\mu\in K_\lambda$, $C\Res_H(\mu+i0^+)CW$ exists in $\mathcal{L}(\Hi)$. Finally, the fact that the limit
		\begin{equation*}
			C\Res_H(\mu+ i0^+)CW=\lim_{\varepsilon\rightarrow 0^+} C\Res_H(\mu+ i\varepsilon)CW=(\Id+C\Res_0(\mu+ i0^+)CW)^{-1}-\Id
		\end{equation*}
		is uniform in $\mu\in K_\lambda$ follows from the continuity of the map $z\mapsto(\Id+C\Res_0(z)CW)^{-1}$ on $\mathring{D}(\lambda,r)\cap\bar\C^+$. 
	\end{proof}

	\section{Appendix to Section \ref{sec:spectr_res}}\label{App:Operateurs_spectraux}

\begin{proof}[Proof of Propositions \ref{prop:Projection_Spectrale_Stone} and \ref{prop:functional_calculus}]
To prove the existence of the weak limit in \eqref{eq:Formule de Stone}, we use twice the resolvent formula, which gives for $\varepsilon\in(0,\varepsilon_0)$, $\varepsilon_0>0$ small enough,
		\begin{align}
			\Res_H(\lambda\pm i\varepsilon)
			&=\Res_0(\lambda\pm i\varepsilon)-\Res_0(\lambda\pm i\varepsilon)V\Res_0(\lambda\pm i\varepsilon)+\Res_0(\lambda\pm i\varepsilon)V\Res_H(\lambda\pm i\varepsilon)V\Res_0(\lambda\pm i\varepsilon)\label{eq:resolvant2_n}. 
		\end{align}
Stone's formula for the self-adjoint operator $H_0$ shows that
			\begin{equation}\label{eq:Formule de Stone_for_H_0}
				\wlim_{\varepsilon\rightarrow 0^+}\frac{1}{2\pi i}\int_I\big(\Res_{0}(\lambda+i\varepsilon)-\Res_{0}(\lambda-i\varepsilon)\big)\mathrm{d}\lambda = \mathds{1}_{I}(H_0)
			\end{equation}
in $\mathcal{L}(\Hi)$. Since $C\Res_0(\lambda\pm i\varepsilon)u$ converge in $L^2(I;\Hi)$ as $\varepsilon\rightarrow 0^+$ by Hypothesis \ref{hyp:principe_absorption_limite}, we deduce that the weak limits
			\begin{equation}\label{eq:f1}
				\wlim_{\varepsilon\rightarrow 0^+}\int_I \Res_{0}(\lambda\pm i\varepsilon)V\Res_{0}(\lambda\pm i\varepsilon) \mathrm{d}\lambda = \wlim_{\varepsilon\rightarrow 0^+}\int_I \Res_{0}(\lambda\pm i\varepsilon)CWC\Res_{0}(\lambda\pm i\varepsilon) \mathrm{d}\lambda 
			\end{equation}
exist in $\mathcal{L}( \Hi )$. For the last term from \eqref{eq:resolvant2_n}, we write
			\begin{align}
				&\wlim_{\varepsilon\rightarrow 0^+}\int_I \Res_0(\lambda\pm i\varepsilon)CWC\Res_H(\lambda\pm i\varepsilon)CWC\Res_0(\lambda\pm i\varepsilon) \mathrm{d}\lambda \notag \\
				&= \wlim_{\varepsilon\rightarrow 0^+}\int_I \Res_0(\lambda\pm i0^+)CWC\Res_H(\lambda\pm i\varepsilon)CWC\Res_0(\lambda\pm i0^+) \mathrm{d}\lambda , \label{eq:f2}
			\end{align}
where we used that, for all $u\in\Hi$, $\lambda\mapsto C\Res_0(\lambda\pm i\varepsilon)u$ converge in $L^2(I;\Hi)$ as $\varepsilon\rightarrow 0^+$, together with the fact that $C\Res_H(\lambda\pm i\varepsilon)CW$ is uniformly bounded in $\varepsilon\in(0,\varepsilon_0)$ by \eqref{eq:limit_uniform_I}. Since $\lambda\mapsto C\Res_0(\lambda\pm i0^+)u$ belongs to $L^2(I,\Hi)$ for all $u\in \Hi$, combining the fact that $C\Res_H(\lambda\pm i\varepsilon)CW$ converges to $C\Res_H(\lambda\pm i0^+)CW$ in $\mathcal{L}(\Hi)$ for a.e. $\lambda\in I$ and again that $C\Res_H(\lambda\pm i\varepsilon)CW$ is uniformly bounded in $\varepsilon\in(0,\varepsilon_0)$, we obtain
			\begin{align}
				&\lim_{\varepsilon\rightarrow 0^+} \int_I \Big \langle u , \Res_0(\lambda\pm i\varepsilon)CWC\Res_H(\lambda\pm i\varepsilon)CWC\Res_0(\lambda\pm i\varepsilon) v  \Big \rangle \mathrm{d}\lambda \notag \\
				&= \int_I \Big \langle C \Res_0(\lambda\mp i0^+) u , WC\Res_H(\lambda\pm i0^+)CWC\Res_0(\lambda\pm i0^+) v \Big \rangle \mathrm{d}\lambda , \label{eq:f3}
			\end{align}
			by Lebesgue's dominated convergence theorem.
			
			Equations \eqref{eq:resolvant2_n}--\eqref{eq:f3} prove that the weak limit in \eqref{eq:Formule de Stone} exists. Moreover, for all $u,v\in\Hi$, we have that
			\begin{align*}
				\langle u , \mathds{1}_I(H) v\rangle &= \langle u , \mathds{1}_I(H_0) v\rangle - \frac{1}{2i\pi} \int_I \big \langle C \Res_{0}(\lambda\mp i0^+) u , WC\Res_{0}(\lambda\pm i0^+) v \big \rangle \mathrm{d}\lambda \\
				&\quad + \frac{1}{2i\pi} \int_I \Big \langle C \Res_0(\lambda\mp i0^+) u , WC\Res_H(\lambda\pm i0^+)CWC\Res_0(\lambda\pm i0^+) v \Big \rangle \mathrm{d}\lambda.
			\end{align*}
			By the same argument, we obtain that \eqref{eq:algebra_morphism} is a Banach algebra morphism. Equation \eqref{eq:Projection_spectral_propriété_d'adjonction} is easily proven, while \eqref{eq:Projection_spectral_propriété_de_projection}, \eqref{eq:Projection_spectral_calcul_fonctionnel_semi-group} and \eqref{eq:Projection_spectral_calcul_fonctionnel_resolvent} follow as in \cite{FaFr18_01}.
\end{proof}

\begin{proof}[Proof of Propositions \ref{prop:Projection_Spectrale_Stone_regul} and \ref{prop:functional_calculus_reg}]
The proof has the same structure as that of Propositions \ref{prop:Projection_Spectrale_Stone} and \ref{prop:functional_calculus}, with the following modifications. First, \eqref{eq:Formule de Stone_for_H_0} is replaced by the following argument. We write
\begin{align*}
&\int_I\big(h(\lambda+i\varepsilon)\Res_{0}(\lambda+i\varepsilon)-h(\lambda-i\varepsilon)\Res_{0}(\lambda-i\varepsilon)\big)\mathrm{d}\lambda\\
&=\int_Ih(\lambda)\big(\Res_{0}(\lambda+i\varepsilon)-\Res_{0}(\lambda-i\varepsilon)\big)\mathrm{d}\lambda\\
&\quad+\int_I\big([h(\lambda+i\varepsilon)-h(\lambda)]\Res_{0}(\lambda+i\varepsilon)-[h(\lambda-i\varepsilon)-h(\lambda)]\Res_{0}(\lambda-i\varepsilon)\big)\mathrm{d}\lambda.
\end{align*}
For the first term, since $H_0$ is self-adjoint, we have
\begin{align}
\wlim_{\varepsilon\rightarrow 0^+}\frac{1}{2\pi i}\int_Ih(\lambda)\big(\Res_{0}(\lambda+i\varepsilon)-\Res_{0}(\lambda-i\varepsilon)\big)\mathrm{d}\lambda=h(H_0). \label{eq:v1}
\end{align}
For the second term, we use the mean-value Theorem together with the Cauchy-Schwarz inequality, writing for all $u,v\in\Hi$,
\begin{align}
&\int_I\big|\big\langle u,[h(\lambda\pm i\varepsilon)-h(\lambda)]\Res_{0}(\lambda\pm i\varepsilon)v\big\rangle_\Hi\big|\mathrm{d}\lambda \notag \\
&\le\varepsilon\|u\|_\Hi\Big(\int_I \Big(\sup_{0<\varepsilon<\varepsilon_0}|h'(\lambda\pm i\varepsilon)| \Big)^2\mathrm{d}\lambda\Big)^{\frac12}\Big(\int_I\big\|\Res_{0}(\lambda\pm i\varepsilon)v\big\|_\Hi^2\mathrm{d}\lambda\Big)^{\frac12}. \label{eq:v2}
\end{align}
The first integral is bounded by \eqref{eq:h_assumpt}. The second integral can be rewritten as
\begin{align}
\int_I\big\|\Res_{0}(\lambda\pm i\varepsilon)v\big\|_\Hi^2\mathrm{d}\lambda &=\int_I\big\langle v,\Res_{0}(\lambda\mp i\varepsilon)\Res_{0}(\lambda\pm i\varepsilon)v\big\rangle_\Hi\mathrm{d}\lambda \notag \\
&=\frac{1}{2i\varepsilon}\int_I\big\langle v,\big(\Res_{0}(\lambda-i\varepsilon)-\Res_{0}(\lambda+i\varepsilon)\big)v\big\rangle_\Hi\mathrm{d}\lambda, \label{eq:v3}
\end{align}
from which we deduce that
\begin{equation}\label{eq:int_res_bound}
\Big(\int_I\big\|\Res_{0}(\lambda\pm i\varepsilon)v\big\|_\Hi^2\mathrm{d}\lambda\Big)^{\frac12}\le \mathrm{c}\varepsilon^{-\frac12}\|v\|_\Hi.
\end{equation}
Together with \eqref{eq:v1}, \eqref{eq:v2} and \eqref{eq:v3}, this implies that
\begin{align*}
\wlim_{\varepsilon\rightarrow 0^+}\frac{1}{2\pi i}\int_I\big(h(\lambda+i\varepsilon)\Res_{0}(\lambda+i\varepsilon)-h(\lambda-i\varepsilon)\Res_{0}(\lambda-i\varepsilon)\big)\mathrm{d}\lambda=h(H_0).
\end{align*}

The rest of the proof follows in the same way as in the proof of Propositions \ref{prop:Projection_Spectrale_Stone} and \ref{prop:functional_calculus} (see in particular \eqref{eq:f1} and \eqref{eq:f2}), using that $h(\lambda\pm i\varepsilon)$ and $h(\lambda\pm i\varepsilon)C\Res_H(\lambda\pm i\varepsilon)CW$ are uniformly bounded in $\varepsilon\in(0,\varepsilon_0)$.
\end{proof}
	
	In the proof of the spectral resolution formula stated in Proposition \ref{prop:resolution}, we used Lemma \ref{lm:res_bound} which we now prove. The arguments are similar to those used in the previous proofs.
		
	\begin{proof}[Proof of Lemma \ref{lm:res_bound}]
	To prove \ref{it:res_bound1}, it suffices to use the resolvent equation \eqref{eq:resolvant2_n} together with the fact that $V=CWC$ and the estimates $\|\Res_0(\lambda\pm i\varepsilon)\|_{\mathcal{L}(\Hi)}\le\varepsilon^{-1}$, $\|C\Res_0(\lambda\pm i\varepsilon)\|_{\mathcal{L}(\Hi)}\le\mathrm{c} \varepsilon^{-1/2}$ (see Lemma \ref{lm:tech_est}) and $\|C\Res(H\pm i\varepsilon)CW\|_{\mathcal{L}(\Hi)}\le\mathrm{c}$ by \eqref{eq:limit_uniform_I}.
	
	To prove \ref{it:res_bound2}, we use again \eqref{eq:resolvant2_n}, writing
	\begin{align*}
\int_I \|\Res_H(\lambda\pm i\varepsilon) u\|_\Hi^2\mathrm{d}\lambda &\le2 \int_I \|\Res_0(\lambda\pm i\varepsilon) u\|_\Hi^2\mathrm{d}\lambda \\
&\quad+2\int_I \|\Res_0(\lambda\pm i\varepsilon)CWC\Res_0(\lambda\pm i\varepsilon) u\|_\Hi^2\mathrm{d}\lambda \\
&\quad+2\int_I \|\Res_0(\lambda\pm i\varepsilon)CWC\Res_H(\lambda\pm i\varepsilon)CWC\Res_0(\lambda\pm i\varepsilon) u\|_\Hi^2\mathrm{d}\lambda.
\end{align*}
By \eqref{eq:int_res_bound}, the first term is bounded $\mathrm{c}\varepsilon^{-1}$. The same holds for the second and third terms, using again that $\|C\Res_0(\lambda\pm i\varepsilon)\|_{\mathcal{L}(\Hi)}\le\mathrm{c} \varepsilon^{-1/2}$ and $\|C\Res(H\pm i\varepsilon)CW\|_{\mathcal{L}(\Hi)}\le\mathrm{c}$.
	\end{proof}

\end{document}